\documentclass[12point]{amsart}
\usepackage{graphicx} % Required for inserting images
\usepackage{overpic}
\usepackage{amsfonts}
\usepackage{amssymb}
\usepackage{amsmath}
\usepackage{amsthm}
\usepackage{tikz}
\usepackage{enumitem}
\usepackage{xcolor}
\usepackage{mathtools}
\usepackage[matrix,arrow,curve,cmtip]{xy} % xy-pic, package for commutative
\usepackage{comment}
\usepackage{hyperref}
\hypersetup{
	colorlinks=true,
	linkcolor=black,
	filecolor=magenta,      
	urlcolor=cyan,
	pdftitle={Extension of rational maps},
}

\usetikzlibrary{cd}

\numberwithin{equation}{section}

\newcommand\C{\mathbb{C}}
\newcommand\CDach{\widehat{\C}}

\newcommand\N{{\mathbb N}}

\newcommand\R{{\mathbb R}}
\newcommand\D{{\mathbb D}}
\newcommand\Db{\overline{{\mathbb D}}}
\newcommand\Hy{\mathbb{H}}
\newcommand\Sp{{\mathbb S}}
\newcommand\U{{\mathbf U}}
\newcommand\V{{\mathbf V}}

\providecommand{\abs}[1]{\lvert#1\rvert}

\providecommand{\norm}[1]{\lVert#1\rVert}

\newcommand{\inte}  {\operatorname{int}}

\newcommand{\crit}{\operatorname{crit}}
\newcommand{\post}{\operatorname{post}}

\newcommand{\diam}{\operatorname{diam}}
\newcommand{\mesh}{\operatorname{mesh}}
\newcommand{\length}{\operatorname{length}}

\setenumerate{itemsep=3pt,topsep=3pt}
\setenumerate[1]{label=\upshape(\roman*)}

\newcounter{mylistnum}

\SetLabelAlign{parright}{\smash{\parbox[t]{\labelwidth}{\raggedleft#1}}}

\newtheorem{theorem}{Theorem}[section]
\newtheorem{corollary}[theorem]{Corollary}
\newtheorem{lemma}[theorem]{Lemma}
\newtheorem{proposition}[theorem]{Proposition}
\newtheorem{claim}{Claim}

\theoremstyle{remark}
\newtheorem*{remark}{Remark}

\numberwithin{figure}{theorem}

\theoremstyle{definition}
\newtheorem{definition}[theorem]{Definition}

\newtheorem{example}[theorem]{Example}

\title{Extending Rational Expanding Thurston Maps}

\author{Daniel Meyer, Julia M\"{u}nch}
\date{}
\begin{document}

\begin{abstract}
  We consider postcritically finite rational maps
  $f\colon \CDach \to \CDach$ whose Julia set is the whole
  Riemann sphere $\CDach$. We call such a map an expanding
  rational Thurston map. Identifying $\CDach$ with the unit
  sphere in $\R^3$, we show that $f$ may be extended on a 
  neighborhood $\Omega\subset \R^3$ of $\CDach$ to a
  quasi-regular map $F\colon \Omega \to \R^3$. In fact, $F$ is
  uniformly quasi-regular in the following sense. The sequence of
  iterates $F^n$, each of which is defined on a neighborhood
  $\Omega_n$ of 
  $\CDach= \Sp^2 \subset \R^3$, is uniformly regular. Here
  $\Omega_n$ shrink to $\CDach$, 
  meaning that $\bigcap \Omega_n = \CDach$.   

  This result may be viewed as a non-homeomorphic version of the
  extension of a quasi-conformal mapping $f:\R^2\to \R^2$ to a
  quasi-conformal mapping $F\colon \R^3 \to \R^3$ due to Ahlfors. 

  % Previously, Mayer proved that if $f: \CDach\to \CDach$ is a Latt\`{e}s map then there exists an extension $F:\Sp^3\to \Sp^3$ that is uniformly quasi-regular. 

	% Ahlfors proved that if $f:\R^2\to \R^2$ is a quasi-conformal mapping then there exists an extension $F:\Hy^3 \to \Hy^3$ such that $f$ is the induced boundary correspondence of $F$. There are some results about extensions of non-injective maps, for example Mayer proved that if $f: \CDach\to \CDach$ is a Latt\`{e}s map then there exists an extension $F:\Sp^3\to \Sp^3$ that is uniformly quasi-regular. 
	
	% We will show that one can extend every rational expanding Thurston map $f:\CDach\to \CDach$ on the Riemann sphere (identified with $\Sp^2\subset \R^3$) to a uniformly quasi-regular mapping defined domain $\Omega\subset \R^3$ containing $\CDach$ in its interior. 
\end{abstract}
\maketitle
\tableofcontents
\section{Introduction}
In this paper we extend certain rational maps (i.e., holomorphic
self-maps on the Riemann sphere $\CDach$) to a neighborhood $\Omega
\subset \R^3$ of $\CDach$, viewed as the unit sphere.

There are essentially two, very different, approaches to
generalize holomorphic maps. The first is the setting of several
complex variables, which is based on considering power series in
several complex variables. The second approach, which is the one
considered in this work, is to generalize the geometric properties of
holomorphic maps. A holomorphic function is conformal away from
its critical points. In dimensions three and higher there is a
rigidity phenomenon; by Liouville's theorem every locally
conformal map is (the restriction of) a M\"{o}bius
transformation. Since the maps we are considering are
non-invertible (i.e., cannot be M\"{o}bius transformations), we
have to consider a condition that is weaker than conformality.

The natural generalization of conformal maps are
\emph{quasi-conformal maps}. Roughly speaking, while conformal
maps preserve angles, quasi-conformal ones are allowed to change
angles up to a fixed multiplicative amount.
Another point of view is that such a map maps an infinitesimal
field of ellipsoids of uniformly bounded eccentricity to an
infinitesimal field of spheres (at almost every point). The bound
on the eccentricity of denoted as the \emph{distortion} of the
map, usually denoted by $K\geq 1$. 

Closely related are
\emph{quasi-regular maps}, which need not be homeomorphisms (as
quasi-conformal ones are by definition). Away from critical
points, they locally satisfy the same distortion estimates as
quasi-conformal ones. See Section~\ref{sec:quasi-regular-maps}
for precise definitions. 

Quasi-conformality/quasi-regularity is not preserved under
products (except in trivial cases). This is one of the main reasons
that extending such maps to higher dimensions is far from trivial. 
% Let us first recall important extension results in a related but
% different setting. These results are about quasi-conformal
% mappings, i.e., quasi-regular homeomorphisms.
Ahlfors and Beurling showed that a quasi-symmetric map $f:\R\to
\R$ (which is a variant of a quasi-conformal map appropriate to $\R$)
admits a quasi-conformal extension $F:\R^2\to \R^2$, see
\cite{AhlforsBeurling1956}. Later 
Ahlfors proved the analogous statement for quasi-conformal
mappings in two dimensions; each quasi-conformal mapping
$f:\R^2\to \R^2$ can be extended to a quasi-conformal mapping
$F:\R^3\to \R^3$, see \cite{Ahlfors1964}. Carleson constructed
quasi-conformal extension from $R^3$ to $\R^3$,
\cite{Carleson1974}. The question was finally 
settled in any dimension $n$ by Tukia and V\"{a}is\"{a}l\"{a} who
proved that every quasi-conformal map $f\colon \R^n\to \R^n$ admits a
quasi-conformal extension $F\colon \R^{n+1} \to \R^{n+1}$, see
\cite{Tukia1982}. 
% By
% reflection along $\R^n$ for $f:\R^n\to \R^n$ and $n\in \N$ all
% the extension results mentioned above can also be formulated as
% extensions of quasi-conformal mappings $f:\R^n\to \R^n$ to
% quais-conformal mappings $\tilde{F}:\R^{n+1}\to \R^{n+1}$ with
% $\tilde{F}|_{\R^n}=f$.

Quasi-regular maps are in many ways natural generalizations of
holomorphic maps. For example, Rickman showed an analogue of
Picard's theorem for quasi-regular maps $f\colon \R^n \to
\R^n$, see \cite{Ri80}. This means that such a map $f$ may only
omit a finite number of points, unless it is constant. The number
of possible omitted points depends on the dimension $n$, as well
as the distortion of $f$.
This result is sharp in the sense that for
any finite set $P\subset \R^n$ there exists a quasi-regular
mapping $f:\R^n\to \R^n\setminus P$. In dimension $n=3$ this is
due to Rickman \cite{Rickman1985} and in $n\geq 3$ the result is
proved by Drasin and Pankka \cite{DrasinPankka2015}.

% Miniowitz
% found a criterion for normality; suppose that $f:\R^n\to \R^n$ is
% uniformly $K$-quasi-regular, $\Omega \subset \R^n$ is a domain
% and the restrictions
% $\mathcal{F}\coloneqq\{f^m|_{\Omega}: m\in \N\}$ omit
% $Q'(n, K)\in \N$ distinct points, then the family $\mathcal{F}$
% is a normal on $\Omega$ \cite{Miniowitz1982}.

Another instant where quasi-regular maps serve as generalizations
of holomorphic maps is to consider their dynamics under
iteration. See for example \cite{Ber10} and \cite{BN14}. Here
however, the following problem arises. Given a quasi-regular map $f$
with distortion $K$, the distortion of the $n$-th iterate $f^n$
is $K^n$ in general. This poses a major obstacle to apply
standard techniques. However, in special cases the family of
iterates $\{f^n\}$ of a quasi-regular map $f$ may have a uniform bound
on their distortion $K$ independent of $n$. In this case, $f$ is
called \emph{uniformly quasi-regular}.

However it is not
easy to find examples of uniformly quasi-regular mappings, the
first example of a non-injective uniformly quasi-regular self-map
of $\Sp^n$ with interesting dynamical behavior was constructed
by Iwaniec and Martin in \cite{Iwaniec1996}.

Further examples of uniformly quasi-regular maps have been found
as extensions of \emph{Latt\`{e}s maps}. These are very special
rational maps $f\colon \CDach \to \CDach$. They are
\emph{postcritically finite}, meaning that each critical point
has a finite orbit, and  their Julia set is $J(f) =
\CDach$. Furthermore, they are obtained as quotients of maps of
the form $z\mapsto az +b$, where $a, b\in \C$, $\abs{a} >1$, by a
crystallographic group. This means they can be thought of as
having constant expansion (i.e., with respect to the so-called
canonical orbifold metric they expand by $\abs{a}$). See
\cite[Chapter~3]{BonkMeyer2017} and \cite{MilLat06} for details. 
Mayer showed that a Latt\`{e}s map has a uniformly quasi-regular
extension $F\colon \R^3\cup \{\infty\}\to\R^3\cup \{\infty\}$
see  \cite{Mayer1997}. Here $\CDach$ is identified with the unit
sphere in $\R^3$. 
The points $0$ and $\infty$ are super-attracting fixed points under $F$ with basin of attraction equal to $B(0, 1)$ respectively $\R^3\cup \{\infty\}\setminus \overline{B}(0, 1)$. The Julia set of the extension $F$ is equal to $\CDach$. If we restrict $F$ to a sphere $\partial B(0, r)$ for arbitrary $r\in (0, 1)$ the image is a sphere $\partial B(0, r')$ of smaller radius $r'<r$.

There are also similar results about the existence of
quasi-regular extensions of holomorphic maps
$f\colon\CDach\to\CDach$ whose Julia set is not the entire
Riemann sphere in \cite{Martin2004}. Martin shows that if
$f:\CDach\to \CDach$ (here we view $\CDach$ as the boundary of
$3$-dimensional hyperbolic space $\Hy^3$) admits a so-called \textit{inverse disc system} there exists a uniformly quasi-regular map $F: \Hy^3\to \Hy^3$ with induced boundary correspondence $f$. Note that by conjugation with a M\"{o}bius transformation we may also view this as an extension of a self-map of $\CDach$ to $B(0, 1)\subset \R^3$. A sufficient condition for the existence of an inverse disc system is that each critical point of $f$ lies in the immediate basin of an attracting fixed point. Further, Martin found criteria for the non-existence of uniformly quasi-regular extensions. We will name a few sufficient conditions. 
\begin{enumerate}
	\item If $f$ is a rational map that has a super-attracting fixed point,
	\item if $f$ has a super-attracting cycle,
	\item if there exists an $N\in \N$ such that $f^N$ has at least two fixed points which are either attracting or indifferent,
	\item if $f$ has a Siegel disk or Herman ring or a cycle of such,
\end{enumerate}
then there does not exist a uniformly quasi-regular extension $F : \Hy^3\to\Hy^3$ so that $f$ is an induced map on the boundary \cite{Martin2004}. 

\smallskip
In this paper we will consider functions $f\colon \CDach \to \CDach$ satisfying the following. 

\begin{itemize}
	\item The map $f$ is holomorphic.
	\item It is postcritically finite, meaning that the postcritical
	set
	\begin{equation*}
		\post(f) = \{f^n(c) : \text{$c$ is a critical point of $f$},
		n\geq 1\}
	\end{equation*}
	is a finite set. 
	\item $f$ is expanding, meaning that its Julia set is all of
	$\CDach$. 
\end{itemize}

Such a map $f$ is often called an \emph{expanding rational Thurston map}. Latt\`{e}s maps are a subfamily of rational expanding Thurston maps.

In a recent paper Li, Pankka and Zheng suggest \textit{expanding cellular Markov branched covers} $f:\Sp^n\to \Sp^n$ as analogues of expanding Thurston maps in higher dimensions \cite{Li2024}. In particular they prove that a \textit{visual metric} (a certain metric arising from the dynamics of $f$) is quasi-symmetric to the spherical metric if and only if $f$ is uniformly quasi-regular and the \textit{local multiplicities} of the iterates is finite. Comparing to expanding Thurston maps on $\CDach$, Bonk and the first author proved that the visual metric is quasi-symmetrically equivalent to the spherical metric if and only if the map is topologically conjugate to a rational map \cite[Theorem 18.1]{BonkMeyer2017}.

We will turn to our main result now. 
Let $F:\Omega \subset \R^n \to \R^n$ be any map. Define 
\begin{equation}
	\label{eq:Omega_n}
	\Omega_n
	\coloneqq 
	F^{-(n-1)}(\Omega),
\end{equation} 
for all $n\in \N$, where $F^{0}$ is the identity. 
Note that on $\Omega_n$ the $n$-th iterate of $F$ is defined. This notion is introduced for mappings that are defined on a domain that is strictly contained in its image.

\begin{theorem}
	\label{theorem: extension result}
	Let $f:\CDach \to \CDach$ be a rational expanding Thurston map. There exists a neighborhood $\Omega\subset\R^3$ so that each $\Omega_n$ (for $\Omega_n$ as in \eqref{eq:Omega_n}) contains $\CDach$ in its interior and a map $F:\Omega \to \R^3$ with $F|_{\CDach}\equiv f$ and a constant $K\geq 1$ so that $F^n|_{\Omega_n}$ is $K$-quasi-regular for all $n\in\N$. 
\end{theorem}

The constant $K$ does not depend on $n$ or $\Omega_n$. The extension $F$ is defined in such a way that for a point $p\in \Omega\setminus\CDach$ there exists a number $n=n(p)\in \N$ such that for $k< n$ we have that $F^k(p)\in \Omega$ and $F^n(p) \in F(\Omega)\setminus \Omega$. For points $p$ that are close to $\CDach$ this number $n(p)$ is large and when iterating $F$ the distance of $F^m(p)$ to $\CDach$ increases, assuming $m\leq n(p)$.

The neighborhood $\Omega$ depends on the function $f$ we start with. It is symmetric about $\CDach$ in the sense that if $M$ is the conformal map that reflects about $\CDach$ then $M(\Omega)=\Omega$.

Our proof strategy for expanding rational Thurston map differs to
the one in \cite{Mayer1997}. Mayer's proof relies on structures
that are available for  Latt\`{e}s maps but not for all Thurston
maps. % For more background on Latt\`{e}s maps see for example
% \cite{Milnor2006} or \cite{BonkMeyer2017}.
Our construction can also be applied to Latt\`{e}s maps and it will yield a different extension than the one found by Mayer.

The maps we consider do not admit an inverse disk system, so the result is not covered by \cite{Martin2004}. The method of the proof is also different. 

An inspiration for our construction is the extension result in \cite{Meyer2010}. The first author proves that every snowsphere $S$ (a two-dimensional fractal homeomorphic to $\Sp^2$ with a similar construction as the Koch snowflake) is quasi-symmetric to $\Sp^2$ and additionally that this quasi-symmetry can be extended to a quasi-conformal map of the snowball (that is the union of the bounded component of $\R^3\setminus S$ and $S$) to the closed unit ball.

\subsection{Organization of the paper}

In Section \ref{sec:background} we review quasi-regular mappings and results that we will need for the proof of Theorem \ref{theorem: extension result}. 

We will prove Theorem \ref{theorem: extension result} first with the additional assumption that every critical value ($f(c)$ for $c\in \crit(f)$) is a fixed point of $f$. In Section \ref{section:ext_crit_value_fixed} we will define the objects needed in the construction of the extension. In order to extend the map $f$ we will define a sequence of sets $S_n\subset B(0, 1)$, each $S_n $ is homeomorphic to $\Sp^2$ and $S_n\to \CDach$ in the Hausdorff sense as $n\to \infty$. More precisely, the sets $S_n$ are defined as follows 
\begin{equation*}
	S_n=\{(z, r_n(z)): z\in \CDach, r_n(z)\in (0, 1)\},
\end{equation*} 
we think of them as spheres with a varying radius given by a continuous function $r_n:\CDach\to (0, 1)$. We will call them \textit{approximating spheres}. The radii in turn are defined via the derivative of an iterate of $f$. The crux lies in defining $r_n$ at a critical point. In order to do that, we will define a sequence of functions $f_n$ that converges uniformly to $f$ and each $f_n$ is differentiable almost everywhere with non-vanishing derivative. In exchange $f_n$ is not holomorphic near a critical point. The functions $f_n$ are obtained by locally replacing the map $f$ around a critical point by the winding map. 

With the sequence $f_n$ and the spheres $S_n$ we will define the extension $F|_{S_n}$ so that the $n$-th approximating sphere is mapped to the $(n-1)$-th approximating sphere, $F|_{S_n}(S_n)=S_{n-1}$. 

Let $M:\R^3\cup \{\infty\}\to \R^3\cup \{\infty\}$ denote (conformal) reflection along $\CDach$. The domain $\Omega$ is bounded by $S_{N}\cup M(S_{N})$ for a suitable index $N\in \N$ (we will see that in some cases $N=2$). 

In Section \ref{section:Def_and_qr_of_Extension} we extend the definition of the extension $F$ from the sets $S_n$ to the domain $\Omega$ by interpolation between the maps defined on the approximating spheres. We will also prove that there exists a constant $K$ such that $F^n$ is $K$-quasi-regular on $\Omega_n$ for all $n\in \N$ with $K$ independent of $n$. 

Section \ref{section:removing_assumption} deals with removing the assumption that each critical value is fixed. We will illustrate how to adapt the constructions from Section \ref{section:ext_crit_value_fixed} and argue that the proofs of uniform quasi-regularity can be adapted.

\section{Notation}
\label{sec:notation}

We identify the Riemann sphere $\CDach=\C \cup \{\infty\}$ with the unit sphere in $\R^3$, i.e., with $\{p\in\R^3: |p|=1\}$. Here and in the following
$|\cdot |$ denotes the Euclidean norm. The (open) unit disk is denoted by $\D=\{z\in \C : \abs{z}<1\}$,
the closed unit disk by $\overline{\D}$. The interior of a set $A$ will be denoted by $\inte(A)$. We will denote the unit circle by $\Sp^1= \{z\in\R^2: |z|=1\}$. The $n$-dimensional Lebesgue measure will be denoted by $\mathcal{L}^n$.

We will denote a ball centered at $z$ of radius $r>0$ by $B(z, r)$.  

For the Euclidean distance of two points $p, q\in \R^n$ or $z, w\in \mathbb{C}$ we will use the notation $|p-q| $ and $|z-w|$, the diameter of a set $A\subset \R^n$ will be denoted by $\diam(A)$. If we consider the spherical metric on $\CDach$ we will use the notation $|z-w|_{\CDach}$ for the spherical distance between $z, w\in\CDach$ and for a set $A\subset \CDach$ we will denote its diameter by $\diam_{\CDach} (A)$. 

Let $S\subset \R^3$ be homeomorphic to the sphere. By $\textup{interior}(S)$ we denote the bounded component of $\R^3\setminus S$. In general $\textup{interior}(S)$ need not be homeomorphic to a ball, but in the cases that we will encounter it will be. For $T\subset \textup{interior}(S)$ homeomorphic to the sphere we denote the open set with boundary components $S, T$ by $A(S, T)$. 

For two quantities $a, b\geq0 $ we say that they are \textit{comparable}, denoted by $a\asymp b$ if there exists a constant $C$ such that \begin{equation*}
	\frac{1}{C}a\leq b\leq Ca 
\end{equation*}
In some cases, the constant $C$ depends on parameters $X_1, \dots, X_n$, $n\in \N$ and in this case we write $C(\asymp )=C(X_1, \dots, X_n)$ to emphasize the dependence. If $a, b$ are as above and there exists a constant $C>0$ such that \begin{equation*}
	a\leq Cb \quad \textup{respectively} \quad a\geq Cb
\end{equation*} 
then we write \begin{equation*}
	a\lesssim b \quad \textup{respectively}\quad a\gtrsim b,
\end{equation*}
and here too if $C$ depends on parameters $X_1, \dots, X_n$ we indicate the dependence by $C(\lesssim)=C(X_1, \dots, X_n)$ respectively $C(\gtrsim)=C(X_1, \dots, X_n)$. 

\section{Background}
\label{sec:background}

\subsection{Thurston maps}

A \textit{Thurston map} $f\colon\CDach\to\CDach$ is a branched covering where every critical point has a finite orbit. If $f$ is additionally holomorphic, we will call it a \textit{rational} Thurston map. 
We will recall what it means for a Thurston map to be expanding and record some consequences and properties of rational expanding Thurston maps. The main focus will be holomorphic Thurston maps but expansion is a concept that also makes sense from a topological viewpoint. We are following \cite{BonkMeyer2017} but we will use a condition that is equivalent to their definition of expansion, see \cite[Proposition 6.4]{BonkMeyer2017}.

The definition is in terms of coverings. Let $(X, d)$ be a metric space, the \textit{mesh} of a covering $\mathcal{U}=\{U\in \mathcal{U}: \bigcup_{\mathcal{U}} U=X\}$ is defined as $\textup{mesh}(\mathcal{U})=\sup_{\mathcal{U}} \{\textup{diam}(U)\}$.

Let $\mathcal{U}$ be a covering of $\CDach$ consisting of connected sets. Then we define $f^{-n}(\mathcal{U})$ as \begin{equation*}
	f^{-n}(\mathcal{U}):=\{V\subset \CDach: V \textup{ is a connected component of }f^{-n}(U), \, U \in \mathcal{U}\}.
\end{equation*}
\begin{definition}
  \label{def:exp_Tmaps}
  Let $f\colon\CDach\to \CDach$ be a Thurston map. It is
  \textit{expanding} if there exists a constant $\delta_0>0$ such
  that if $\mathcal{U}$ is an open cover of $\CDach$ by open and
  connected subsets satisfying $\textup{mesh}(\mathcal{U})<
  \delta_0$ then  
  \begin{equation*}
    \lim_{n\to \infty}\mesh(f^{-n}(\mathcal{U}))= 0.
  \end{equation*}
\end{definition}

For rational Thurston maps it follows that the Julia set $J(f)$ is equal to the whole Riemann sphere. 
\begin{proposition}
	Let $f\colon \CDach\to\CDach$ be a rational Thurston map. Then the following are equivalent: 
	\begin{enumerate}
		\item $f$ is expanding, 
		\item $J(f)=\CDach$, 
		\item there do not exist periodic critical points.
	\end{enumerate}
\end{proposition}
This is shown in \cite[Proposition 2.3]{BonkMeyer2017}

\subsection{Quasi-regular maps}
\label{sec:quasi-regular-maps}
Quasi-regular maps can be viewed as generalizations of holomorphic maps. The definition is similar to the one of quasi-conformal maps with the difference that we do not assume the map to be a homeomorphism. For a map $f\colon\Omega\to\R^n$ we define the \textit{branching set} $B_f\subset\Omega$ as the set where $f$ is not locally a homeomorphism. 
We will give the definition in terms of Sobolev maps and the weak derivative but then also state a characterization that is easier to verify in some cases. 

\begin{definition}
  Let $\Omega\subset \R^n$ be a domain and $f\colon\Omega \to
  \R^n$ be a Sobolev map, $f\in W^{1, n}_{\textup{loc}}(\Omega,
  \R^n)$. If there exists a constant $K<\infty$ such that for
  almost all $z\in \Omega$ we have that
  \begin{equation*}
    \|Df(z)\|^n \leq K |\textup{det}(Df(z))|
  \end{equation*}
  then $f$ is called a \emph{$K$-quasi-regular} map. We call $K$
  the distortion of $f$.
  % The infimum
  % $K(f)$ taken over all constants $K$ such that $f$ is
  % quasi-regular is called the \textit{distortion} of $f$. 
\end{definition}

A map $f\colon \Omega \to \R^n$ is called quasi-regular, if it is
$K$-quasi-regular for some $K>0$. 

\begin{remark}
  If $f\colon \mathbb{C}\to \mathbb{C}$ is $K$-quasi-regular and
  $g\colon\mathbb{C}\to \mathbb{C}$ is holomorphic then the
  compositions $g\circ f$ as well as $f\circ g$ are
  $K$-quasi-regular as well.
	
  For maps in $\R^n$ we have that if
  $f\colon \R^n\to \R^n$ is $K$-quasi-regular and
  $g\colon\R^n\to \R^n$ is a M\"{o}bius
  transformation then $f\circ g$ as well as $g\circ f$ are
  $K$-quasi-regular.
\end{remark}
For more details on the definitions see for example
\cite{Rickman1993} and for the background on quasi-conformal
mappings see \cite{Vaisala1971}.

Instead of the definition we will verify the following criterion to prove that the extension is quasi-regular.
\begin{theorem}
  A non-constant mapping $f\colon \Omega\to \R^n$ is
  $K$-quasi-regular if and only if the following conditions are
  satisfied
  \begin{enumerate} 
  \item $f$ is orientation-preserving, open and discrete, 
  \item there exists a constant $K<\infty$ such that for almost
    all $x\in \Omega$ the distortion of $f$ at $x\in \Omega$ satisfies
    \begin{equation*}
      K_f(x)
      :=
      \limsup_{\epsilon \searrow 0}
      \frac{\max_{|x-y|=\epsilon}|f(x)-f(y)|}{\min_{|x-z|=\epsilon}|f(x)-f(z)|}
      \leq K.
    \end{equation*}
    % is bounded by $K$ and the distortion is locally bounded in
    % $\Omega$.  
  \end{enumerate}
\end{theorem}

For a proof of this characterization see \cite[Theorem 6.2]{Rickman1993}.

Let $f\colon\R^n\to\R^n$ and $g\colon\R^n\to \R^n$ be quasi-regular mappings with quasi-regularity constant $K$, respectively $K'$. Then their composition is a quasi-regular mapping as well with distortion bounded by $KK'$. So in particular for iterates $f^n$ we get a distortion bound of $K^n$. 

In some special cases, the distortion does not increase under composition. 
\begin{definition}
	Let $f\colon\R^n\to \R^n$ be a $K$-quasi-regular mapping such that for all $n\in\N$ the iterate $f^n\colon\R^n\to \R^n$ is also $K$-quasi-regular. Then $f$ is called \textit{uniformly} $K$-quasi-regular.
\end{definition}

For a $K$-quasi-regular mapping $f\colon \Omega \to \R^n$ with $\Omega \subset \R^n$ and $\Omega \subset F(\Omega)$ we define $\Omega_n\coloneqq f^{n-1}(\Omega)$. Then on $\Omega_n$ the $n$-th iterate of $f$ is defined, and we can find constants $K_n$ such that $f^n$ is $K_n$-quasi-regular on $\Omega_n$ (a priory $K_n\leq K^n$).

\begin{definition}
	Let $f: \R^n\to \R^n$ be a differentiable function. We
	introduce
	\begin{equation*}
		\begin{split}
			\| Df(x)\| &:= \max_{\|v\|=1} \|Df(x)v\|\\
			\|Df(x)\|_{*} &:= \min_{\|v\|=1} \|Df(x)v\|.
		\end{split}
	\end{equation*}
	the \textit{maximal}, respectively \textit{minimal expansion}
	of $f$ at $x$.  
\end{definition}
Here, the Euclidean norm is used on the right hand side. 
Note that $\norm{Df(x)}$ is simply the operator norm. Furthermore
$\norm{Df(x)}/\norm{Df(x)}_*= K_f(x)$. 

%\begin{remark}
%	Note that if $f: \R^n\to\R^n$ is quasi-regular the distortion can be estimated by \begin{equation*}
%		K \leq \frac{\|Df(x)\|}{\|Df(x)\|_{*}}
%	\end{equation*}
%\end{remark}

We will recall the definition of bi-Lipschitz maps in metric spaces and introduce quasi-similarities. They are both simple examples of quasi-regular maps. 

\begin{definition}
	Let $(X, d_X), (Y, d_Y)$ be metric spaces. A map $f\colon X\to Y $ is \textit{bi-Lipschitz} if there exist constants $L, L'$ such that for all $x, y\in X$ we have that \begin{equation*}
		L'\cdot d_X(x, y)\leq d_Y(f(x), f(y))\leq L\cdot d_X(x, y).
	\end{equation*}
	In case we wish to emphasize the constants we will say $f$ is $(L, L')$-bi-Lipschitz. If we wish to refer to only one constant we will set $M =\max\{L, 1/L'\}$ and say that $f$ is $M$-bi-Lipschitz.
\end{definition}

\begin{definition}\label{definition: quasi-similarity}
	Let $(X, d_X), (Y, d_Y)$ be metric spaces. We say that a map $f\colon X\to Y$ is a $(L, L', \mu)$\textit{-quasi-similarity} if there exist constants $L, L', \mu>0$ such that for all $x, y\in X$ \[L'\cdot d_X(x, y)\leq \frac{1}{\mu}\cdot  d_Y(f(x), f(y))\leq L\cdot d_X(x, y).\]
	We will call $\mu$ the \textit{stretch factor} of $f$. 
\end{definition}

\begin{remark}
	It is immediate that $(L, L')$-bi-Lipschitz maps as well as $(L, L', \mu)$-quasi-similarities are $K$-quasi-regular with $K \leq L/L'$. 
\end{remark}

\subsection{Metrics on the Riemann sphere and differentiation}
We will be interested in the spherical derivatives of holomorphic functions and an adaptation of the spherical derivative to functions that are real differentiable but not holomorphic.

Let $U\subset \C$ be a region equipped with the Euclidean metric. If $\sigma\colon U \to (0, \infty)$ is a positive and continuous function we can define another metric, called \textit{conformal metric} where the distance between $z, w\in U$ is given by the infimum of \begin{equation*}
	d_\sigma(z, w):= \inf_\gamma \int_\gamma \sigma(u)\, |du| 
\end{equation*}
where the infimum is taken over all rectifiable curves joining $z $ to $w$.

The density $\sigma$ for the spherical metric is given by $\sigma(z)=2/(1+|z|^2)$ on $\C$. 
For now we will denote the open set equipped with a \textit{conformal metric} given by the density $\sigma$ by $(U, \sigma)$. 
If $U, V\subset \mathbb{C}$ are open sets and $f\colon (U, \sigma)\to (V, \sigma')$ is a holomorphic map then we can express the derivative with respect to the conformal metrics by \begin{equation*}
	|f'(z)|_{\sigma, \sigma'}:= \lim_{w\to z}\frac{d_{\sigma'}(f(z), f(w))}{d_{\sigma}(z, w)}= \frac{\sigma'(f(z))}{\sigma(z)}|f'(z)|. 
\end{equation*}

In the special case that $U=V=\C$ and $\sigma'(z)=\sigma(z)=2/(1+|z|^2)$ we obtain the spherical metric.

\begin{definition} For a holomorphic function $f\colon \CDach\to\CDach$ the \textit{spherical derivative} is given by
	 \begin{equation*}
		f^{\#}(z)=\frac{1+|z|^2}{1+|f(z)|^2}|f'(z)|.
	\end{equation*} 
	If we suppose that $f\colon\CDach\to\CDach$ is real differentiable then we define \begin{equation*}\|Df(z)\|_{\CDach}:= \frac{1+|z|^2}{1+|f(z)|^2} \|Df(z)\| \end{equation*} 
	where $\|Df(z)\|$ denotes the operator norm of the differential at $z$.
	In both cases we take a limit if $z=\infty$ or $f(z)=\infty$.
 \end{definition}
For a holomorphic function both expressions agree. In some cases the notation is shorter if we do not reserve $\| Df(z)\|_{\CDach}$ strictly for non-holomorphic differentiable functions so we will sometimes use it for holomorphic functions as well.

One can show that for holomorphic maps $f, g\colon \CDach\to
\CDach$ and the spherical derivative the chain rule holds as
well;
\begin{equation*}
  (f\circ g)^{\sharp}(z)
  =
  f^{\#}(g(z))g^{\#}(z). 
\end{equation*}
It follows from the usual chain rule for holomorphic functions and the conformal factors cancel. 

In the same way we can prove a chain rule for maps $f,
g\colon\CDach\to \CDach$ that are differentiable in the real
sense we get an estimate
\begin{equation*}
  \|D(f\circ g)(z)\|_{\CDach}
  \leq
  \|Df(g(z))\|_{\CDach} \cdot\|Dg(z)\|_{\CDach}.
\end{equation*}
Often we will pre- and/or post-compose a function $g\colon\CDach \to
\CDach$ that is real differentiable (not necessarily holomorphic) with
holomorphic functions $f,h\colon\CDach \to
\CDach$. In this case we have
\begin{equation}
  \label{eq:Dfgh}
  \norm{D(f\circ g \circ h)(z)}
  =
  f^{\#}(z'') \cdot \norm{D g(z')} \cdot h^{\#}(z), 
\end{equation}
where $z'= f(z)$, $z''=g(z')= g\circ f(z)$. 

% and
% in that case we have
% \begin{equation*}
%   \|D(f^n\circ g)(z)\|_{\CDach}
%   =
%   (f^n)^{\#}(g(z))\cdot \|Dg(z)\|_{\CDach},
% \end{equation*}
% respectively
% \begin{equation*}
%   \|D(f^n \circ g\circ f^m)(z)\|_{\CDach}
%   =
%   (f^n)^{\#}(g\circ f^m(z))\cdot \|Dg(f^m(z))\|_{\CDach}\cdot (f^m)^{\#}(z).
% \end{equation*}

We will also consider holomorphic functions from $\CDach$ with the spherical metric to  $\C$ with the Euclidean metric and vice versa. For convenience we provide the formulas for the derivative adapted to the metrics. For holomorphic functions $f\colon U\subset \CDach \to \C$ where we consider the spherical metric on $\CDach$ we will -- slightly abusing notation -- denote the derivative by 
\begin{equation}
	\label{equation: spherical dericative convention 1}
	f^{\#}(z):= \frac{1+|z|^2}{2}|f'(z)|
\end{equation}
and for holomorphic functions $g\colon V\subset \C \to \CDach$ we denote the  derivative by \begin{equation*}
	g^{\#}(z):=\frac{2}{1+|f(z)|^2}|f'(z)|. 
\end{equation*}
\begin{remark}
	We are using the same notation for three different expressions. By the domain and target of the functions it will be clear which one we mean at each appearance. Usually functions denoted by $f$ or $g$ will be from the Riemann sphere to itself, and then we are considering the spherical derivative $f^{\#}(z)=\frac{1+|z|^2}{1+|f(z)|^2}|f'(z)|$. If we use coordinates, the maps will be denoted by $\phi$ or $\psi$ and the maps will be defined on a subset of $\CDach$ with $\C$ as the target. There the spherical derivative is $\phi^{\#}(z)=\frac{1+|z|^1}{2}|\phi'(z)|$.
\end{remark}

\subsection{Koebe Distortion Theorem}
We will use Koebe's distortion theorem for holomorphic functions. For a proof see \cite[Theorem 1.3]{Pommerenke1992}. For completeness we will state the version on the unit disc and an adaptation to the spherical metric. 

\begin{theorem}
	Let $f\colon \D\to \mathbb{C}$ be a holomorphic univalent function. Then \begin{equation*}
		|f'(0)|\frac{|z|}{(1+|z|)^2}\leq |f(z)-f(0)|\leq |f'(0)|\frac{|z|}{(1-|z|)^2}.
	\end{equation*}
\end{theorem}
\begin{remark}
	If we restrict to $z\in B(0, R)$ for some $R<1$ then we get the following estimate \begin{equation*}
		|f'(0)|\frac{1}{(1+R)^2} |z|\leq |f(z)-f(0)|\leq |f'(0)|\frac{1}{(1-R)^2}|z|,
	\end{equation*}
	we can view this as $f$ being a quasi-similarity (see \ref{definition: quasi-similarity}) with constants $(1/(1+R^2), 1/(1-R^2), |f'(0)|)$, i.e., the scaling factor $\mu=|f'(0)|$ and the constants $L', L$ only depend on $R$ and not on $f$.
\end{remark}

For the spherical version of the Koebe distortion theorem we refer to \cite[Theorem A.2]{BonkMeyer2017} for a proof. Note that there the chordal metric is used instead of the spherical metric. These metrics are bi-Lipschitz equivalent so we can easily adapt the statement.
\begin{theorem}
  \label{theorem: spherical Koebe}
  Let $0<r<R<\textup{diam}_{\CDach} (\CDach)=\pi$ and let
  $B(z_0, R)\subset \CDach$. Suppose that
  $f\colon B(z_0, R)\to \CDach$ is holomorphic, injective and
  additionally that $f(B(z_0, R))$ is contained in a hemisphere
  of $\CDach$. Then for all $z, w\in B(z, r)$ we have
  that
  \begin{align*}
    f^{\#}(z)&\asymp f^{\#}(w)\\
    f^{\#}(z) L_1|z-w|_{\CDach} \leq |f(z)&-f(w)|_{\CDach}
                                            \leq
                                            f^{\#}(z)L_2|z-w|_{\CDach}.
  \end{align*}
  The constants $L_1=L_1(r/R)\nearrow 1$ and
  $L_2=L_2(r/R)\searrow 1$ as $r/R\to 0$. The notation
  $f^{\#}(z)\asymp f^{\#}(w)$ stands for the existence of
  constants $C_1, C_2>0$ such that
  $C_1f^{\#}(w)\leq f^{\#}(z)\leq C_2f^{\#}(w)$.
\end{theorem}

\subsection{Coordinates around critical points and Koenig's Theorem}\label{subsection: coordinates}
Let $f\colon \CDach\to\CDach$ be a rational map. Suppose $c$ is a critical point that is mapped to $p$. There exist local coordinates near $c$ respectively $p$, with respect to which $f$ is given as $z^d$ with $d=\deg(f, c)$ respectively $f'(p)\cdot z$. We will explain this in more detail. Let $U$ be a simply connected neighborhood of $p$ and suppose $p$ is the only critical value contained in $U$. Suppose $\partial U$ is locally connected and that $\partial U$ does not contain critical points or values. Then there exists a pre-image $V$ of $U$ containing the critical point $c$. It follows that $V$ is also simply connected and that $\partial V$ is locally connected. By the Riemann mapping theorem there exist bi-holomorphic maps $\phi\colon V\to \D$ and $\psi\colon U\to \D$ that are normalized with $\phi(c)=0$ and $\psi(p)=0$. The map $\psi\circ f\circ \phi^{-1}\colon \D\to \D$ is a holomorphic map and since the image is the disc it is bounded. By the following theorem we understand the map better. 

\begin{theorem}\cite[Exercise~IX.2.1]{Gamelin2001}
  \label{thm:Blaschke_coord}
	Let $f\colon \D\to\C $ be a holomorphic function that can be extended continuously to $\Sp^1$. If $\lim_{z\to 1}|f(z)|=1$ then $f$ is of the form $f(z)=e^{i\theta}b(z)$, where $b(z)$ is a Blaschke product.
\end{theorem}

Since $0$ is the only zero of $\psi\circ f \circ \phi^{-1}$ the Blaschke product $b(z)=e^{i\theta}z^d$ for $d\in\N$ suitable. We may change the normalization of $\psi$, and denote the normalized function also by $\psi$ in order to obtain \begin{equation*}
	\psi\circ f \circ \phi^{-1}(z)=z^d
\end{equation*}
for all $z\in \D$, where $d$ is the local degree of $f$ at $\phi^{-1}(z)$.

Now we will look at the behavior at fixed points. Let $f\colon\C\to\C$ be a holomorphic function and let $p\in \C$ be a fixed point. We call the derivative at $p$ the \textit{multiplier} and we will denote it by $\lambda:= f'(p)$. If $|\lambda|\neq 0, 1$ then we can find local coordinates that conjugate $f$ with the map $z\mapsto \lambda \cdot z$. 

\begin{theorem}\label{theorem: Koenig}
	Let $f\colon \C\to\C$ be a holomorphic map and $p$ be a fixed point with multiplier $\lambda$ of modulus $|\lambda|\neq 0, 1$. Then there exists a local holomorphic coordinate $w=\phi(z)$ with $\phi(p)=0$ so that $\phi\circ f \circ\phi^{-1}$ is the linear map $w\mapsto \lambda \cdot w$ for all $w$ in a neighborhood of zero. 
\end{theorem}

For a proof see \cite[Theorem 8.2]{Milnor2006}.

\section{The extension under additional assumptions}\label{section:ext_crit_value_fixed}
In this section we will define an extension $F\colon \Omega \subset \R^3 \to
\R^3$ of $f\colon\CDach \to \CDach$ under the simplifying
assumption that each critical value is a fixed point, i.e., for a
critical point $c$ we have $c\mapsto p\circlearrowleft$. Note
that this implies that the pre-image $\{f^{-n}(c): n\in
\N\}$ of $c$ does not contain any critical point. This means we
prove Theorem \ref{theorem: extension result} in this case. 
The general case will be treated in Section \ref{section:removing_assumption}. 

\subsection{Outline}
\label{sec:outline}

Here we give an outline of the construction in this section. In
particular, we introduce the principal objects that will be
defined. 

We will consider a sequence of  $2$-spheres
$\{S_n\}_{n\in \N_0}\subset B(0,1)\subset \R^3$, called
\textit{approximating spheres}. These will be defined in terms of
a \emph{sequence of radii} $\{r_n\}_{n\in \N_0}$, where
$r_n\colon \CDach \to (0,1)$ is a continuous function for each
$n\in \N_0$. In turn, these 
radii are defined using a \emph{sequence of functions}
$\{f_n\}_{n\in \N_0}$, where $f_n \colon \CDach \to \CDach$. The
desired extension $F$ will then be defined using these sequences.

The approximating spheres $S_n$ accumulate at the boundary 
$\CDach$. This means that for every $\epsilon>0$, there is an
$N\in \N$, such that $S_n$ is contained in the
$\epsilon$-neighborhood of $\CDach$ for all $n\geq N$. 

Each $S_n$ will be parametrized by spherical
coordinates $(z,r)$, where $z\in \CDach$, $r\in (0,1)$. We refer
to $z$ as the \emph{spherical part} and to $r$ as the
\emph{radial part} of the coordinates of a point $p=(z,r)$.

The radii of the approximating spheres will not be constant, but
will be given via the aforementioned sequence (of radii)
$\{r_n\}_{n\in \N_0}$. This means we have
\begin{equation}
  \label{eq:def_Sn}
  S_n = \{(z,r_n(z))\in \R^3 : z\in \CDach\}.
\end{equation}
Note that this implies that $\textup{interior}(S_n)$ is
homeomorphic to a ball for all $n\in\N$.

Recall that $\post(f)$ denotes the postcritical set of $f$. For
each $p\in \post(f)$, we consider a suitable closed
neighborhood $U_p\subset \CDach$ of $p$. Let us denote by $\U=
\bigcup_{p\in \post(f)}U_p$ and by $\mathbf{U}^n = f^{-n}(\U)$.

For $z\in \CDach \setminus \U^n$ we define
\begin{equation}
  \label{equation: naive r_n}
  r_n(z)= 1-s\,\frac{1}{(f^n)^{\#}(z)}
\end{equation}
for all $n\geq 1$. Here $s\in(0, 1)$ is a constant to guarantee
that for all $n\in\N$ and all $z\in\CDach$ we have that
$r_n(z)\in(0, 1)$. For $n=0$ we will set $r_0(z)\equiv r_0$ a
suitable constant $r_0\in (0, 1)$. This will be made precise in
Section \ref{section: sequence of radii}.

Clearly, at critical points of $f^n$ we are not able to use the
previous definition. We first explain the definition of
$r_1\colon \CDach \to (0,1)$. We define a new function $f_1\colon
\CDach\to \CDach$ that is obtained by modifying $f$ in a
neighborhood of critical points. On $\CDach \setminus \U^1$ we
set $f_1 = f$. 

Fix a point $p \in \post(f)$ and
consider a point $c\in f^{-1}$. Then there is a (unique) closed
component $U^1_c$ of $f^{-1}(U_p)$ that contains $c$. When $c$
is not a critical point, we set $f_1= f$ on $U^1_c$. When $c$
is a critical point, the map $f\colon U^1_c\to U_p$ is given as
$z\mapsto z^d$, where $d= \deg(f,c)$, in suitable coordinates in
domain and range. Using these coordinates, $f_1$ is given on
$U^1_c$ by the \emph{winding map}. Then $\norm{Df_1}_{\CDach}$ is
continuous, never $0$, and equal to $f^\#$ on $\CDach \setminus
\U^1$. 

For $z\in \CDach$ we define
\begin{equation}
  \label{equation: naive r_1}
  r_1(z)= 1-s\,\frac{1}{\norm{Df_1(z)}}_{\CDach}. 
\end{equation}

The maps $f_n\colon \CDach \to \CDach$ are similarly defined by
replacing $f$ in a neighborhood $U^n_c$ of each critical point
$c$ by the winding map. Here $U^n_c$ is a component of
$f^{-n}(U_p)$ containing $c$, where $p = f(c)$ is a postcritical
point.

This defines the sequence $\{f_n\}$. It converges uniformly to
$f$. The radii are now defined by

\begin{equation}
  \label{eq:def_rn}
  r_n(z)
  =
  1-s\, \frac{1}{\|D(f_1\circ\dots\circ f_n)(z)\|_{\CDach}},
\end{equation}
for all $z\in \CDach$ and $n\in \N_0$. We now choose $s\in(0, 1)$ is a constant chosen such that $r_n(z)\in (0,
1)$ for all $z\in \CDach$ and all $n\in \N$. We will set
$r_0 :=\frac{1}{2}\min_{z\in\CDach, n\geq 1}\{r_n(z)\}$ and
$r_0(z)=r_0$. 

Having defined the sequence $\{r_n\}$, the
sequence of approximating spheres $\{S_n\}$ is now defined by
\eqref{eq:def_Sn}. The desired extension $F$ is first defined on each
approximating sphere $S_n$ by 
\begin{equation}
  \label{eq:def_extension_outline}
  F(z, r_n(z)) = (z', r_{n-1}(z')).
\end{equation}
for each $z\in \CDach$. Here $z' = f_n(z)$. This means that
$F\colon S_n\to S_{n-1}$.

% The extension $F$ of $f$ will first be defined as a map from
% $S_n$ to $S_{n-1}$ for $n\geq 1$. For the points that do not lie
% on an approximating sphere we will interpolate.  

% The second principal object we will construct is a sequence 

% The first subsection \ref{subsection: sequence f_n} deals with a sequence of functions $(f_n)$ for $n\in\N$ that converges uniformly to $f$ and has a non-vanishing derivative wherever it is differentiable. The caveat is that the functions $f_n$ are not holomorphic everywhere. 

% We will then use this sequence in order to define a sequence of continuous functions $r_n:\CDach\to (0, 1)$ in Section \ref{section: sequence of radii}. These will be used to define what we will call approximating spheres; for each $n\in\N$ the $n$-th approximating sphere $S_n$ is parametrised by \begin{equation*}
% 	S_n= \{(z, r_n(z)): z\in\CDach\}.
% \end{equation*}
% Because $r_n$ is continuous each $S_n$ is homeomorphic to $\CDach$. 

Sections \ref{subsection: extension far from critical points} and \ref{subsection: extension near critical points} are dedicated to the definition of the extending map away from respectively close to critical points. We also prove quasi-regularity in these sections.

\subsection{Local coordinates}
\label{subsection: local coordinates}
We will assume that each critical value is a fixed
point.  Let $p$ be a critical value. Because $f$ is expanding the
multiplier $\lambda=f'(p)$ satisfies $|\lambda|>1$. In a
neighborhood of $p$, the map $f$ may be given as $z\mapsto
\lambda z$ in suitable coordinates using Koenig's
Theorem \ref{theorem: Koenig}. 

The precise meaning is that there is a neighborhood $V$ of $p$, which is a Jordan domain, and a conformal map $\psi\colon V\to B(0, R)$ for some $R>1$ such that $\psi\circ f(z)=\delta_{\lambda}\circ \psi(z)$ for all $z\in V$. 

We may restrict $\psi $ to a closed set $U= U_p^0$ such that the following diagram commutes:

\begin{equation}
%  \label{eq:diag_UcUp}
  \label{eq:loc_coord_Up}.
  \xymatrix{
    U_p^1 \ar[r]^{f} \ar[d]_{\psi}
    & U_p^0 \ar[d]^{\psi} \\
    \lambda^{-1}\overline{\D} \ar[r]^{\delta_{\lambda}}
    &\overline{\D}\rlap{.}
  }
\end{equation}
Here $U_p^1$ is the (unique) component of $f^{-1}(U^0_p)$
containing $p$ and $\delta_{\lambda}(z):=\lambda\cdot z$. We say that $\psi$ is conformal because it is a restriction of $\psi$ to a closed subset $U_p^1$ that is compactly contained in $V$.

Note that we may replace $U^0_p$ by $\psi^{-1}(r\D)$ for
any $r\in (0, 1)$. This means we can assume that $U_p^0$ is a
Jordan domain with analytic boundary, in particular $\partial
U_p^0$ and $\partial U_p^1$ are locally connected and are
null-sets with respect to the Lebesgue measure of $\CDach$.

Furthermore, we assume that for
distinct $p,p'\in \post(f)$ the sets $U^0_p$ and $U^0_{p'}$ are
disjoint. From now on we will assume that the sets
$U_p^0$ are of this form. 
The sets $U_p=U_p^0$, for $p\in
\post(f)$ and consequently $\U\coloneqq \bigcup_{p\in\post(f)}
U_p$ will be fixed from now on, with a possibility of choosing them smaller only for Lemma \ref{lemma:mon_radii_crit}. 

Note that it follows that each component $U^n$ of $f^{-n}(U_p^0)$ is a closed
disk with analytic boundary, in particular a Jordan domain whose
boundary is a null-set for $2$-dimen\-sional Lebesgue
measure. Each $U^n$ contains a unique point $z_n\in f^{-n}(p)$;
conversely for each such $z_n$, there is a unique such component
$U^n$ containing $z_n$. Let us denote
\begin{equation*}
  \U^n
  =
  \{U^n : U^n \text{
    is a component of } f^{-n}(U_p)
  \text{ for some } p\in \post(f)\}.
\end{equation*}
By choosing each $U_p$ sufficiently small (for $p\in \post(f)$), we
can assume that each component $U^1$ of $\U^1$ that does not
contain a postcritical point is disjoint from $\U$. 

Recall that each $p\in \post(f)$ is a fixed point by assumption, meaning that
$p\in f^{-n}(p)$ for each $n\in \N_0$. We denote by
\begin{equation}
  \label{eq:defUnp}
  U^n_p \text{ the unique component of } f^{-n}(U_p)
  \text{ containing } p,
\end{equation}
for each $n\in \N_0$. 
% Hence for each
% $n\in \N_0$, there is a unique component of $f^{-n}(U_p)$ that
% contains $p$, which we denote by $U^n_p$.
Note that from the
above it follows that $\psi(U^n_p) = \abs{\lambda}^{-n} \overline{\D}$, i.e., we
could have defined $U^n_p = \psi^{-1}(\abs{\lambda}^{-n} \overline{\D})$ (i.e., in
our Koenig's coordinates $U^n_p$ is
$\abs{\lambda}^{-n}\overline{\D}$). This description immediately implies that 
\begin{align*}
  &U_p = U^0_p \supset U^1_p \supset \dots,
  \bigcap_n U^n_p = \{p\}, \text{ and}
  \\
  &\diam_{\CDach} (U^n_p) \to 0, \text{ as } n\to \infty.
\end{align*}
Let us denote
\begin{equation}
  \label{eq:defUnpost}
  \U^n_{\post}\coloneqq \bigcup_{p\in \post} U^n_p. 
\end{equation}

Now consider a critical point $c\in f^{-1}(p)$ for some $p\in
\post(f)$. Let $U_c= U^1_c$ 
be the unique component of $f^{-1}(U_p)$ containing $c$. It will
be convenient to choose the sets $U_p$, for $p\in \post(f)$,
sufficiently small such that all the 

In
suitable local coordinates, we may express the map $f$ in a
neighborhood of $c$ as $z\mapsto z^d$, where $d= \deg(f,z)$. 
To make this precise, we note that by
Theorem~\ref{thm:Blaschke_coord}, and the discussion following it,
there exists a conformal map $\phi: U_c\to \overline{\D}$, mapping
$c$ to $0$, such that the
following diagram commutes: 
\begin{equation}
  \label{eq:loc_coord_zd}
  \xymatrix{
    U_c \ar[r]^{f} \ar[d]_{\phi}
    & U_p \ar[d]^{\psi} \\
    \overline{\D} \ar[r]^{z^d}
    & \overline{\D}\rlap{.}
  }
\end{equation}

We denote by
\begin{equation}
  \label{eq:defUnc}
  \text{$U^{n+1}_c$ the component of
  $f^{-1}(U_p^n)$ that contains $c$},
\end{equation}
for each $n\in \N_0$. Note that $U^{n+1}_c\subset \CDach$ is the
component of $f^{-(n+1)}(U_p)$ that contains $c$, i.e.,
$U^{n+1}_c$ exists and is unique. In our  
local coordinates provided by $\phi$, $U^{n+1}_c$ is given by
$|\lambda|^{-n/d} \overline{\D}$, i.e., $U^{n+1}_c =
\phi^{-1}(|\lambda|^{-n/d} \overline{\D})$.
Similar to the above, we have
\begin{align*}
  &U_c = U^1_c \supset U^2_p \supset \dots,
  \bigcap_n U^n_c = \{c\}, \text{ and}
  \\
  &\diam_{\CDach} (U^n_c) \to 0, \text{ as } n\to \infty.
\end{align*}

Let us denote
\begin{equation}
  \label{eq:defUncrit}  
  \U^{n+1}_{\crit} \coloneqq \bigcup_{c\in \crit} U^{n+1}_c, 
\end{equation}
for $n\in \N_0$. 
Note that $\U^{n+1}_{\crit}$ is the union of the components of
$f^{-(n+1)}(\U)$, respectively the union of components of
$f^{-1}(\U^n_{\post})$, containing a critical point. 

Note that in \eqref{eq:loc_coord_Up} as well as in
\eqref{eq:loc_coord_zd} the maps $\psi, \phi$, as well as the
local degree $d$ and the multiplier $\lambda$ depend on $p$,
respectively $c$. We omitted the respective subscripts to lighten
notation. We will continue doing this in the following. However,
these maps (and numbers) will not depend on $n$, even when
applied to the subsets $U^n_p\subset U_p$, respectively
$U^{n+1}_c\subset U_c$. 

\begin{figure}[htb]\label{figure: local coordinates}
	\begin{overpic}[scale=0.4]{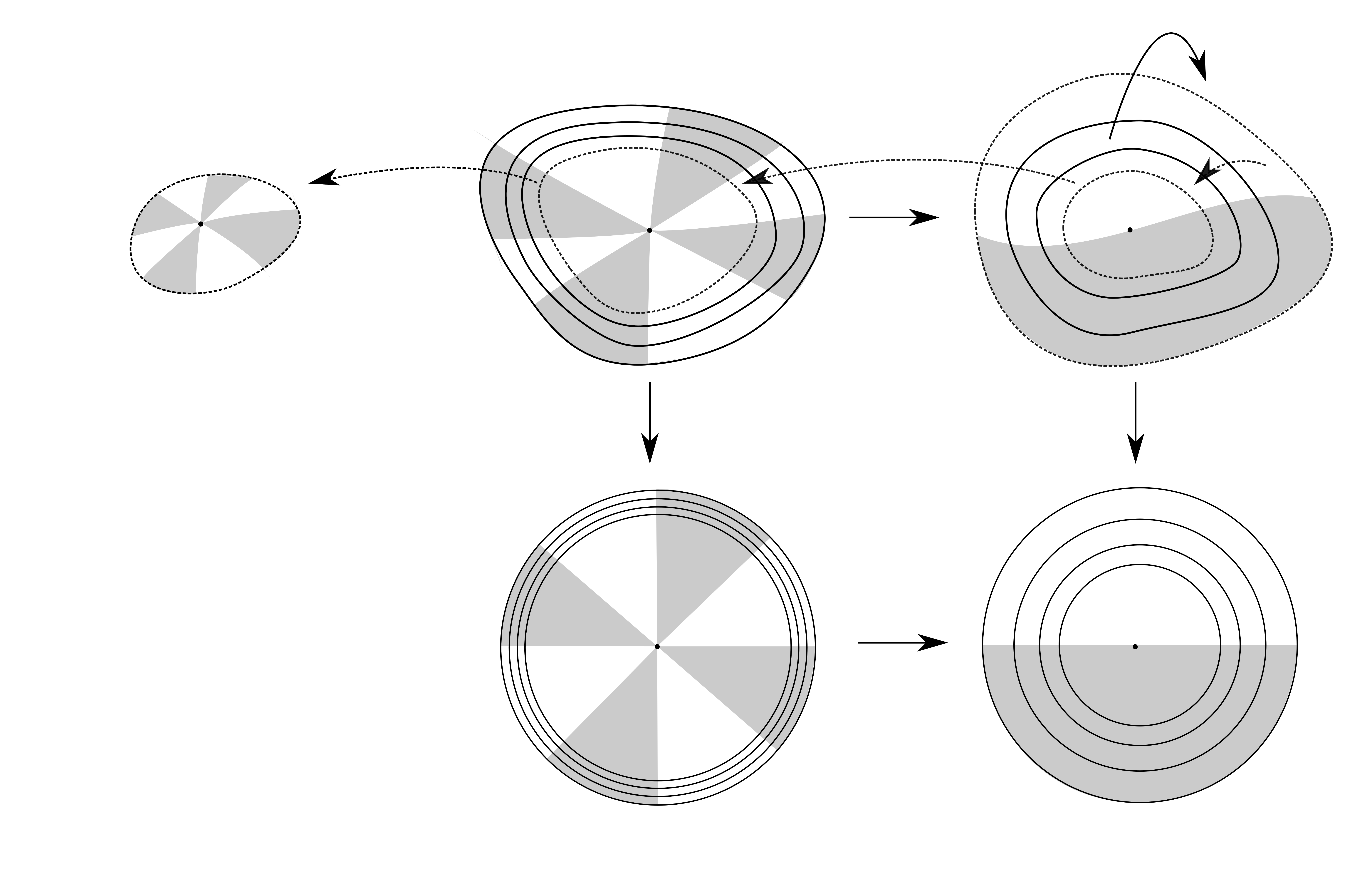}
		\put(28, 54){\footnotesize $f^{-m}$}
		\put(65, 55){\footnotesize $f^{-1}_c$}
		\put(65, 45){\footnotesize $f$}
		\put(65, 20){\footnotesize $\xi^d$}
		\put(44, 34){\footnotesize $\phi$}
		\put(84, 34){\footnotesize $\psi$}
		\put(82, 63){\footnotesize $f$}
		\put(93, 53){\footnotesize $f^{-3}_p$}
	\end{overpic}
	\caption{This indicates the backwards iteration of a post-critical fixed point. The dashed arrows and domains are an example where we iterate backwards for a few iterates with a conformal branch, then we take a pre-image containing the critical point and then again conformal branches of the inverse after passing the critical point.}
\end{figure}
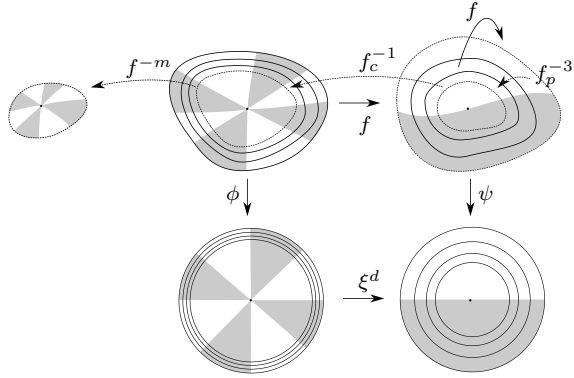

\smallskip
Let us record some estimates for the coordinate maps $\psi,
\phi$, which will immediate consequences of Koebe's distortion
theorem. In fact, we will be estimating $\psi^{-1}\colon \Db \to
U_p\subset \CDach$ and $\phi\colon U_c \subset \CDach \to
\Db$. Recall that $\CDach$ is equipped with the spherical metric,
while $\Db\subset \C$ is equipped with the Euclidean
metric. Recall also that we are abusing notation and denote the
derivative with respect to these metrics by $\phi^{\#}$,
respectively $(\psi^{-1})^{\#}$. These are given by
\begin{equation*}
  \phi^{\#}= \frac{1+|z|^2}{2}|\phi'(z)|,
\end{equation*}
for all $z\in U_c$ and
% $\psi\colon U_p \subset \CDach \to \Db$ and
% $\phi\colon U_c\subset \CDach \to \C$, so the derivative adapted
% to the spherical metric on $\CDach$ is given by
\begin{align*}
  (\psi^{-1})^{\#}(w')
  =
  \frac{2}{1+\abs{\psi^{-1}(w')}^2}|(\psi^{-1})'(w')|
  =
  \frac{2}{1+\abs{w}^2} \abs{\psi'(w)}^{-1},
\end{align*}
for all $w'\in \Db$, where $w=\psi^{-1}(w') \Leftrightarrow w' =
\psi(w)$. In the case when $z=\infty$ and/or $w=\infty$, these
expressions are to be understood as a suitable limit.

\begin{lemma}
  \label{lem:Koebe_phi_psi}
  There is a sequence $\{L(n)\}$ with $L(n) \searrow 1$ as $n\to
  \infty$ with the following properties.
  \begin{enumerate}    
  \item
    \label{item:Koebe_phi}
    Whenever $z,z_0,z_1\in U^n_c$ for a $n\in \N$ and
    $c\in \crit(f)$, we have 
    \begin{align*}
      &\phi^{\#}(z)
      \asymp
      \phi^{\#}(c),
      \\
      &\abs{\phi(z_0)-\phi(z_1)}
      \asymp
      \phi^{\#}(c)\abs{z_0-z_1}_{\CDach},
    \end{align*}
    where $C(\asymp) = L(n)$.
  \item 
    \label{item:Koebe_psi}
    Consider a $p\in \post(f)$, $n\in \N_0$, as well as the
    corresponding set $U^n_p\subset \CDach$. 
    Whenever $w', w'_0, w'_1\in \abs{\lambda}^{-n}\Db =
    \psi(U^n_p)$ we have  
    \begin{align*}
      &(\psi^{-1})^{\#}(w')
      \asymp
      (\psi^{-1})^{\#}(0)
      =
        (\psi^{\#})^{-1}(p),
        \intertext{as well as}
        &\abs{w_0 - w_1}_{\CDach}
        =
        \abs{\psi^{-1}(w'_0)-\psi^{-1}(w'_1)}
        \asymp
          (\psi^{-1})^{\#}(0)\abs{w_0'-w'_1}
        \text{ or }
      \\
      &\abs{w'_0 - w'_1}
        =
        \abs{\psi(w_0)-\psi(w_1)}
        \asymp
      (\psi^{\#})(p)\abs{w_0-w_1}_{\CDach},
    \end{align*}
    where $w_0= \psi^{-1}(w'_0), w_1 = \psi^{-1}(w'_1)\in U^n_p$
    ($\Leftrightarrow \psi(w_0) = w'_0, \psi(w_1) =w'_1$). 
  \end{enumerate}
  Here $C(\asymp) = L(n)$ in all the above estimates. 
\end{lemma}

\begin{proof}
  This follows from Theorem \ref{theorem: spherical Koebe}
  applied to the maps $\phi, \psi^{-1}$ on the sets $U^n_c$,
  respectively $\abs{\lambda}^{-n}\Db= \psi(U_p^n)$. Since there
  are only finitely many critical as well as postcritical points,
  we may use the same constant $C(\asymp) = L(n)$ throughout.  
\end{proof}

Let us record some elementary estimates on the derivative of $f$. 

\begin{lemma}
  \label{lem:fn_Unp}
  \mbox{}
  \begin{enumerate}
  \item
    \label{item:fn_Unp}
    Let $z\in U^n_p$, for a $p\in \post(f)$. Then
    \begin{equation*}
      (f^n)^{\#}(z)
      =
      \abs{\lambda}^n(\psi^{-1})^{\#}(w)\psi^{\#}(z)
      \asymp
      \abs{\lambda}^n.
    \end{equation*}
    Here $w= \lambda^n \psi(z)$ and $C(\asymp)$ is a universal
    constant.
  \item 
    \label{item:fn_Un}
    Assume the component $U^1$ of $\U^1$ does not contain a
    critical point. Then
    \begin{equation*}
      f^{\#}(z)
      \asymp
      1,
    \end{equation*}
    for all $z\in U^1$ where $C(\asymp)$ is a constant. 
  \end{enumerate}
  
\end{lemma}

\begin{proof}
  Property~\ref{item:fn_Unp} follows immediately from the chain
  rule, the fact that 
  $\psi$ can be extended conformally to a neighborhood of $U^0_p$
  and Koebe distortion; Property~\ref{item:fn_Un} is simply an
  application of the Koebe distortion. 
\end{proof}

\subsection{The winding map}
\label{sec:winding-map}

As outlined at the beginning of this section, the maps $f_1$ respectively $f_n$ will be constructed by replacing $f$ in the neighborhood $U_c^1$, respectively $U_c^n$ by the winding map.

% We will replace $f$ locally by the \emph{winding map}. As a first step towards the definition of $f_n$ we will record some properties of the winding map. 

\begin{definition}[The winding map]
	\label{definition: winding map}
  The \emph{winding map} of degree $d\geq 1$ is the map $h_d\colon
  \C \to \C$ given as follows.
  For $z= r e^{it}$, where $r\geq 0$, $t\in \R$, we define
  \begin{equation*}
    h_d(z) = h_d(re^{it})
    \coloneqq
    \begin{cases}
      r e^{idt}, & 0\leq r \leq 1;
      \\
      z^d, & r\geq 1. 
    \end{cases}
  \end{equation*}
\end{definition}
Note that $h_d$ is well-defined and continuous. Furthermore,
$h_d$ agrees with the map $z\mapsto z^d$ on $\Sp^1 = \partial
\D$. We will be mostly interested in $h_d$ on $\overline{\D}$,
but it will be convenient to have $h_d$ defined on a larger
domain.

We will consider sets of the form
\begin{equation}
  \label{eq:def_sector}
  \Db(\theta_0, \alpha)
  \coloneqq
  \left\{re^{i \theta} \in \Db : 0\leq r < 1,\,
    \theta_0 - \frac{\alpha}{2}
    \leq
    \theta
    \leq
    \theta_0 + \frac{\alpha}{2}
  \right\} \subset \Db,
\end{equation}
for $\theta_0,\in \R$, $\alpha>0$. Such a set is called a
\emph{sector of angle $\alpha$}. On suitable such sectors $h_d$
will be injective, which will help to resolve certain issues
later on. 

% Alternatively, this may be expressed as
% \begin{equation*}
%   h_d(z)
%   =
%   \begin{cases}
%     z^d/\abs{z}^d, &z\neq 0,
%     \\
%     0, & z=0.
%   \end{cases}
% \end{equation*}

% Note that for $z\in\Sp^1$ we have $h_d(z) = z^d$. 

\begin{lemma}
  \label{lemma: h is bilipschitz}
  The winding map $h_d$ satisfies the following.
  \begin{enumerate}    
  \item
    \label{item:prop_hd_2}
    The map $h_d$ is real differentiable
    everywhere except for 
    the origin and on $\Sp^1$.
    The operator norm of the differential $\|Dh_d(z)\|$ extends
    continuously to $\C$ by
    \begin{equation*}
      \norm{Dh_d(z)}
      =
      \begin{cases}
        d, & z\in \overline{\D};
        \\
        d \abs{z}^{d-1}, & \abs{z}>1. 
      \end{cases}
    \end{equation*}
    The minimal distortion satisfies
    \begin{equation*}
      \norm{Dh_d(z)}_{*}
      =
      \begin{cases}
        1, &z\in \D;
        \\
        d\abs{z}^{d-1} &|z|>1.
      \end{cases}
    \end{equation*}
  \item
    \label{item:prop_hd_1}
    On each sector $\Db(\theta_0, \pi/d)$ of angle $\pi/d$ (see
    \eqref{eq:def_sector}), we have the following
    estimate 
    \begin{equation*}
      \begin{split}
        |z-w|
        \leq 
        |h_d(x)-h_d(w)|
        \leq 
        d\cdot |z-w|,
          % = 
          % \|Dh_d(z)\| |z-w|
      \end{split}
    \end{equation*}
    for $z,w\in \D(\theta_0, \alpha)$. 
  \item
    \label{item:prop_hd_3}
    The distortion of $h_d$ is equal to $d$ of $\overline{\D}$
    and $0$ on $\C\setminus \overline{\D}$. 
  \end{enumerate} 
  % Moreover we can extend $\|Dh_d(z)\|$ continuously to where the winding map is not differentiable by setting $\|Dh_d(z)\|=d$ for $z\in \Sp^1\cup \{0\}$.
\end{lemma}

\begin{proof}
  \ref{item:prop_hd_2}
  Clearly, for $\abs{z}>1$ we have
  \begin{equation*}
    \norm{Dh_d(z)}
    =
    \norm{Dh_d(z)}_*
    =
    \abs{d z^{d-1}} = d \abs{z}^{d-1}. 
  \end{equation*}
  Let us now assume that $z= r e^{i\theta}\in \D \setminus\{0\}$.
  Recall that the partial derivatives in polar coordinates are
  (as can easily be verified using the chain rule)
  \begin{align*}
    \frac{\partial}{\partial x}
    &=
      \cos(\theta)\frac{\partial}{\partial r}
      -\frac{\sin(\theta)}{r}\frac{\partial}{\partial\theta}\\
    \frac{\partial}{\partial y}
    &=
      \sin(\theta)\frac{\partial}{\partial r}
      +\frac{\cos(\theta)}{r}\frac{\partial}{\partial\theta}. 
  \end{align*}
  By rotational symmetry, it suffices to compute $\norm{Dh(z)},
  \norm{Dh(z)}_*$ for $\theta=0$ (on the positive real line). The
  Jacobian at $\theta=0$ is 
  \begin{equation*}
    Dh_d
    =
    \begin{pmatrix}
      1&0\\ 0&d
    \end{pmatrix}.
  \end{equation*}
  Thus we have for $z\in \D\setminus \{0\}$
  \begin{equation*}
    \norm{Dh_d(z)} = d,
      \quad
      \norm{Dh_d(z)}_* = 1. 
    \end{equation*}
    The claim follows.
  %   and so the operator norm is $d$. By setting $\|Dh_d(0)\|=d$ on
  % $\overline{\D}$ we can extend the norm of the differential
  % continuously as desired. 

  \smallskip
  \ref{item:prop_hd_3}
  Clearly $h_d$ is holomorphic on $\C\setminus \overline{\D}$,
  meaning its distortion on this set is $0$.
  On $\D\setminus \{0\}$ we have $\norm{Dh_d} = d$ and (its
  Jacobian is) $Jh_d =d$
  by the expression for $Dh_d$ above. Thus its distortion is $K=K(h_d) = d$.

  \smallskip
  \ref{item:prop_hd_1}
  Let $S= \Db(\theta_0, \pi/d)$ be a sector of angle $\pi/d$ as in in \eqref{eq:def_sector} for
  some $\theta_0 \in \R$. 
  Note that $S$ is mapped by $h_d$ homeomorphically to $S'=
  \Db(d\theta_0,\pi)$, a sector of angle $\pi$.     

Let $z,w\in S$ be arbitrary. Note that the sector $S$ is convex,
meaning that the line segment between $z$ and $w$, denoted by
$\gamma_{zw}$, is contained in $S$. Then
  \begin{equation*}
    \abs{h_d(z) - h_d(w)}
    \leq
    \length(h_d\circ \gamma_{zw})
    \leq
    \max_{\zeta\in h_d\circ \gamma_{zw}}\norm{Dh_d(\zeta)}
    \abs{z-w}
    = d \abs{z-w}, 
  \end{equation*}
  which is the right of the desired inequalities.
  
  To see the other inequality, consider now
  $z'=h_d(z),w'=h_d(w)\in S'$. Since $\Db(d\theta_0,\pi)$ is convex, 
  the line segment between $z'$ and $w'$ denoted by
  $\gamma_{z'w'}$, is contained in $S'$. Let us
  denote the inverse of $h_d \colon S\to S'$ by $h^{-1}$.
  Using \ref{item:prop_hd_2} we obtain for each $\zeta' \in S'$ 
  \begin{equation*}
    \norm{D h^{-1}(\zeta')}
    =
    \norm{ \left(Dh_d(\zeta)\right)^{-1}}
    = \norm{D h_d(\zeta)}_*^{-1} =1,
  \end{equation*}
  where $\zeta = h^{-1}(\zeta')$. 
  
  Analogous as above, we have 
  \begin{align*}
    \abs{z-w}
    &=
    \abs{h^{-1}(z') - h^{-1}(w')}
    \leq
      \length(h^{-1}\circ \gamma_{z'w'})
      \\
    &\leq
    \max_{\zeta'\in h^{-1}\circ \gamma_{z'w'}}
      \norm{Dh^{-1}(\zeta')} \abs{z'-w'}
      % =     \max_{\zeta'\in h^{-1}\circ \gamma_{z'w'}}(\norm{D
      % h(\zeta)}_*)^{-1} \abs{h_d(z) - h_d(w)}
    \\
    &= \abs{h_d(z) - h_d(w)}, 
  \end{align*}
  finishing the proof. 
\end{proof}

\subsection{Defining the maps $f_n$}
\label{sec:defining-maps-f_n}

In Section~\ref{subsection: local coordinates} we have defined
local coordinates, which respect to which $f$ is given as
$z\mapsto z^d$ near a critical point. The latter will be replaced
by the winding map $h_d$, given in Section~\ref{sec:winding-map}
(suitably rescaled). This will yield the maps $f_n$, for all
$n\in \N$. These maps have the property that their derivative
does not vanish anywhere, which is the reason to define them. 
% has 
% Having defined local coordinates, as well as the winding map $h_d$, we
% now proceed to define the sequence $\{f_n\}$ of maps $f_n \colon
% \CDach \to \CDach$. Roughly speaking, in the neighborhood of
% critical point $f$ is given as $z\mapsto z^d$ in suitable local
% coordinates; $f_n$ is obtained by replacing $z\mapsto z^d$ by
% $h_d$.

Recall from \eqref{eq:defUncrit} the definition of the set
$\U^{n+1}_{\crit}$, which was the union of sets $U^{n+1}_c$, for $c\in
\crit(f)$, see \eqref{eq:defUnc} (for each $n\in \N_0$). 

Let us first define $f_1$.
On $\CDach \setminus \U^1_{\crit}$, we set
$f_1 = f$.

It remains to define $f_1$ on each component $U^1_c$
for $c\in\crit(f)$.
Recall that we have defined local coordinates in this case, with
respect to which $f\colon U^1_c \to U_p$ is given by $z\mapsto
z^d$. Here $p = f(c)\in \post(f)$ and $d= \deg(f, c)$, see
\eqref{eq:loc_coord_zd}. This 
means, we have conformal maps $\phi\colon U^1_c \to \Db$ and
$\psi\colon U_p \to \Db$, such that $\psi\circ f \circ \phi^{-1}
(z) = z^d$. Here, we replace $z^d$ by the winding map $h_d$. More
precisely, we define for $z\in U^1_c$
\begin{equation*}
  f_1(z)
  \coloneqq
  (\psi^{-1} \circ h_d\circ \phi)(z). 
\end{equation*}
  
In order to visualize this definition, we will introduce the following diagram:
\begin{equation*}
%  \label{diagram: maps on both sides}
  \xymatrix{ U_c^1 \ar[r]^{f}_{f_1}
    \ar[d]_{\phi}
  & U_p \ar[d]^{\psi} \\
  \overline{\D} \ar[r]^{z^d}_{h_d}
  &
  \overline{\D}. 
}
\end{equation*}

The diagram combines two separate diagrams in one; the diagram
commutes whenever we consider either only the maps on top of the
right-pointing arrows or only the maps below right-pointing
arrows. Note that if we consider the top entry for one arrow and
the bottom entry for the other the diagram does not commute. In
case two maps coincide we will only denote one above the
corresponding arrow. Recall that $z\mapsto z^d$ and $h_d$ agree
on $\Sp^1 = \partial \Db$. Consequently $f$ and $f_1$ agree 
$\partial U^1_c$. Thus $f_1$ is continuous on $\CDach$. 

\smallskip
Let us now define $f_{n+1} \colon \CDach \to \CDach$ for $n\geq
0$ which is fixed from now on (defining $f_{n+1}$ instead of $f_n$
will simplify the discussion).
We now define $f_{n+1}$ as
follows. 

On $\CDach \setminus \U^{n+1}_{\crit}$ we set $f_{n+1}= f$.

Consider now a component $U^{n+1}_c$ of $\U^{n+1}_{\crit}$
containing a critical point $c$. 
%
% \smallskip
% {\it Case 1:}
% $U^{n+1}$ does not contain a critical point. In this case, we set
% $f_{n+1}|U^{n+1} = f$. 
%
% \smallskip
% {\it Case 2:} $U^{n+1}= U^{n+1}_c$ contains a critical point
% $c$.
Then $f(U^{n+1}_c)= U^n_p$ where $p = f(c) \in \post(f)$. Let
$\psi= \psi_p$ and $\phi= 
\phi_c$ be the local coordinates near $p$, respectively
$c$. Furthermore, let $\lambda = \lambda_p$ be the multiplier of
$p$ and $d= \deg(f,c)$ be the local degree at $c$. Note that $U^n_p
\subset U_p$, and $U^n_p$ is given in local coordinates by
$\abs{\lambda}^{-n}\Db$, meaning that $\psi(U^n_p) =
\abs{\lambda}^{-n}\Db$. Similarly, $U^{n+1}_c \subset U^1_c$ and
is given in local coordinates by $\abs{\lambda}^{-n/d}\Db$, meaning
that $\phi(U^{n+1}_c) = \abs{\lambda}^{-n/d}\Db$.

Let us define a
winding map on these scaled disks, namely, we define
\begin{align*}
  &h_{n+1,c}\colon |\lambda|^{-n/d}\overline{\D}\to
  |\lambda|^{-n}\overline{\D}
  \\
  &h_{n+1,c}(z)
  =
  \delta_{|\lambda|^{-n}}\circ
  h_d\circ\delta_{|\lambda|^{n/d}}(z). 
\end{align*}

\begin{lemma}
  \label{lemma:hnc is continuous}
  The map $h_{n+1, c}$ agrees with $z^d$ on
  $|\lambda|^{-n/d}\,\Sp^1$. Furthermore, $h_{n+1,c}$ is real
  differentiable on $\abs{\lambda}^{-n/d}\D\setminus \{0\}$.
  We may extend 
  $\norm{Dh_{n+1,c}}$ and $\norm{Dh_{n+1,c}}_*$
  continuously to $|\lambda|^{-n/d}\Db$ by 
  \begin{equation*}
    \norm{Dh_{n+1,c}(z)} = d\abs{\lambda}^{-n + n/d},
    \quad
    \norm{Dh_{n+1,c}(z)}_* = \abs{\lambda}^{-n + n/d},
  \end{equation*}
  for all $z\in |\lambda|^{-n/d}\Db$. 
\end{lemma}
Note that for $z\in \abs{\lambda}^{-n/d}\,\Sp^1$ this means
\begin{equation*}
  \abs{(z^d)'}
  =
  \abs{d z^{d-1}}
  =
  d \abs{\lambda}^{-n +n/d}
  =
  \norm{Dh_{n+1,c}(z)}. 
\end{equation*}

\begin{proof}
  In the setting above, consider an arbitrary $z=
  \abs{\lambda}^{-n/d}e^{i\theta}\in |\lambda|^{-n/d}\Sp^1$ (where
    $\theta\in [0,2\pi)$). Then $z^d=
    \abs{\lambda}^{-n}e^{id\theta}$.  On the other hand
    \begin{equation*}
      \begin{split}
        h_{n+1, c}(z)
        &=
          \abs{\lambda}^{-n} h_d(\abs{\lambda}^{n/d} z)\\
        &=
          \abs{\lambda}^{-n} e^{id\theta} = z^d. 
    \end{split}
  \end{equation*}
  To see the second statement, we compute for $z\in
  \abs{\lambda}^{-n/d}\D\setminus \{0\}$
  \begin{equation*}
    \norm{Dh_{n+1,c}(z)}
    =
    \abs{\lambda}^{-n} \norm{Dh_d(\abs{\lambda}^{n/d} z)}
    =
    d\abs{\lambda}^{-n +n/d},
  \end{equation*}
  using Lemma~\ref{lemma: h is bilipschitz}~\ref{item:prop_hd_2}. 
  The statement about $\norm{Dh_{n+1,c}}_*$ is shown analogously.  
\end{proof}
 We define for all $z\in U^{n+1}_c$
 \begin{equation*}
   f_{n+1}(z)
   =
%   \begin{cases}
     (\psi^{-1}\circ h_{n+1,c}\circ \phi)(z). 
   %   &\textup{if } z \in U_c^{n+1} \text{ for a } c\in \crit(f);
   %   \\ 
   %   f(z), \quad &\textup{else}.
   % \end{cases}
 \end{equation*}
 Roughly speaking this means $f_{n+1}$ is defined by a (suitably
 scaled) winding map on $U^{n+1}_c$; everywhere else by $f$. 

 In order to visualize $f_{n+1}$ we will be using the following
 diagrams: 
 \begin{equation*}
   %\label{diagram: maps on both sides}
  \xymatrix{ U_c^{n+1} \ar[r]^{f}_{f_{n+1}}
    \ar[d]_{\phi}
  & U_p^{n} \ar[d]^{\psi} \\
  |\lambda|^{-n/d}\Db\ar[r]^{z^d}_{h_{n+1, c}}
  &
  |\lambda|^{-n}\Db. 
}
\end{equation*}

Here the same convention as above is used, namely the
diagram commutes if we either use the maps on top or on the
bottom of the horizontal arrows. Note that by
Lemma~\ref{lemma:hnc is continuous} we know that $f_{n+1} = f$
and $\norm{D f_{n+1}}+{\CDach} = f^{\#}$ on $\partial U^{n+1}_c$. Thus
$f_{n+1}$ and $\norm{Dh_{n+1},c}$ are continuous. Put 
differently, we define $f_{n+1}$ as
\begin{equation}
  \label{eq:def_fn}
  f_{n+1} \coloneqq
  \begin{cases}
    f, &\text{on } \CDach \setminus \U^{n+1}_{\crit}\\
    \psi^{-1}\circ h_{n+1,c}\circ \phi, &\text{on } U^{n+1}_c
                                          \text{ for each } c\in
                                          \crit.     
  \end{cases}
\end{equation}

\begin{remark}
  The maps $h_{n+1, c}, \phi, \psi$ depend on $c\in \crit(f)$ and $p =
  f(c)\in \post(f)$, but in order to lighten the
  notation subscripts were dropped. We will use the same
  convention in the future. 
  % we will drop the index $c$ and only keep $n+1$. For
  % each result about the functions $f_n$ that is local we will
  % keep in mind that the coordinates and the degree depend on the
  % critical point.
\end{remark}

When we restrict the winding map to sectors $\Db(\theta_0,
\alpha)$ (compare \eqref{eq:def_sector}) of a suitable angle
$\alpha>0$ it is injective. In the following it will be
convenient to restrict $f_n$ to corresponding subsets of $U_c^n$
such that it 
is injective. For $z\in U_c^n$ and $\alpha>0$ we will choose the
pre-image of $\Db(\theta_0, \alpha)$ for $\theta_0\in [0,2\pi)$ with
$\phi(z)= |\phi(z)|e^{i\theta_0}$ under the coordinate $\phi$
intersected with $U_c^n$ 
\begin{equation}
  \label{eq:Sectors_in_sphere}
  U_c^n(z, \alpha)
  \coloneqq 
  \phi^{-1}(\Db(\theta_0, \alpha))\cap U_c^n.
\end{equation}

Let us collect some properties of the sequence of functions
$\{f_n\}$ just defined. 

\begin{lemma}
%  \label{lemma:Koebe for f_n}
  \label{lem:prop_fn}

  There is a sequence $\{L(n)\}$ with $L(n)\searrow
  1$ as $n\to \infty$ such that the following is true.

  The functions $f_n\colon \CDach\to\CDach$ defined above for $n\in \N$
  satisfy the following.
  \begin{enumerate}
  \item 
    \label{item:prop_fn_1}
    For all $n\in\N$ the maps $f_n$ are continuous.
  \item
    \label{item:prop_fn_2}
    The sequence $f_n$ converges uniformly to $f$ in
    $\CDach$.  
  \item
    \label{item:prop_fn_3}
    Each $f_n$ is real differentiable everywhere except at each
    critical point $c$ and on the sets $\partial
    U_c^n$; $\|Df_n\|_{\CDach}$ has a 
    continuous extension to $\CDach$ and
    \begin{equation*}
      \min\{\norm{Dh_{n+1,c}(z)} : z\in \CDach\} >0.  
    \end{equation*}

  \item
    \label{item:prop_fn_3a}
    For $z\in U^{n+1}_c$, where $c\in \crit(f)$, $p= f(c)\in
    \post(f)$, and $d= \deg(f,c)$, we have
    \begin{align*}
      \norm{D f_{n+1}(z)}_{\CDach}
      &=
        d \abs{\lambda}^{-n+ n/d}
      (\psi^{-1})^{\#}(w') \cdot
        \phi^{\#}(z)
       %(\psi^{-1})^{\#}(w') \cdot \norm{D h_{n+1,c}(z')} \cdot \phi^{\#}(z)
      \\
      &\asymp
        d \abs{\lambda}^{-n+n/d} (\psi^{\#}(p))^{-1} \phi^{\#}(c)
        \text{ and}
      \\
      \norm{D f_{n+1}(z)}_{* \CDach}
      &\asymp
        \abs{\lambda}^{-n+n/d} (\psi^{\#}(p))^{-1} \phi^{\#}(c). 
    \end{align*}
    Here $C(\asymp) = L(n)$ and $w'= h_{n+1,c}\circ \psi(z)$ .
  \item
    \label{item:prop_fn_4}
    For each $z\in \CDach \setminus \U^{n}_{\crit}$ there exist
    $R=R(z)>0$ such that 
    on $B(z, R)$ the function $f_n$ is holomorphic and univalent.
    For any $0<r<R$ and for all $w\in B(z, r)$ we have 
    \begin{equation*}
      \abs{f_n(z)-f_n(w)}_{\CDach}
      \asymp
      f_n^{\#}(z)\abs{z-w}_{\CDach}.
    \end{equation*}
    Here $C(\asymp) = L'(r/R) \searrow 1$ as $r/R\to 0$. 
    %
    % Here $L'(r/R) \to 0$ as $r/R \to 0$. 
    % For any $0<r<R$ there exists a constant $L'=L'(r/R)$ such
    % that for all $w\in B(z, r)$ we have that
    % \begin{equation*}
    %   \frac{1}{L'}f_n^{\#}(z)\abs{z-w}_{\CDach}
    %   \leq
    %   \abs{f_n(z)-f_n(w)}_{\CDach}
    %   \leq
    %   L'f_n^{\#}(z)\abs{z-w}_{\CDach}.
    % \end{equation*}
    % Here $L'(r/R) \to 0$ as $r/R \to 0$. 
%
%
% save counter of enumeration
    \setcounter{mylistnum}{\value{enumi}}
  \end{enumerate}

  % Furthermore, there is a sequence $\{L(n)\}$ with $L(n)\searrow
  % 1$ as $n\to \infty$ such that the following is true. 
  \begin{enumerate}
    % reset counter
    \setcounter{enumi}{\value{mylistnum}}
  \item
    \label{item:prop_fn_5}
    % Let $z\in U_c^n\setminus \{c\}$, denote $z':= \phi(z)$ and let 
    % \begin{equation*}
    %   A(z')
    %   =
    %   \{w'\in \D:
    %   \textup{arg}(y)\in \left(\textup{arg}(x)-\pi/2d,
    %     \textup{arg}(x)+\pi/2d\right)\}. 
    % \end{equation*} 

    For each $c\in \crit(f)$ the following holds. 
    Let $z\in U^n_c$ be arbitrary. Then 
    for all $w\in U_c^n(z,\pi/d)$ we have that
    \begin{equation*}
      \frac{1}{L}\| Df_n(z)\|_{* \CDach}\abs{z-w}_{\CDach}
      \leq \abs{f_n(z)-f_n(w)}_{\CDach}
      \leq L\|Df_n(z)\|_{\CDach} \abs{z-w}_{\CDach},
    \end{equation*}
    where $L= L(n)$. 
  \end{enumerate}
\end{lemma}

\begin{remark}
  Recall that on $\CDach\setminus \U_{\crit}^n$ the map
  $f_n$ is holomorphic and $f^{\#}= \norm{Df}$. There we can rewrite
  \ref{item:prop_fn_5} as follows
      \begin{equation*}
      \frac{1}{L'}\| Df_n(z)\|_{\CDach}\abs{z-w}_{\CDach}
      \leq \abs{f_n(z)-f_n(w)}_{\CDach}
      \leq L'\|Df_n(z)\|_{\CDach} \abs{z-w}_{\CDach},
    \end{equation*}
    where $L'(r/R)\geq 1$ is decreasing as $r/R$ decreases.

  Note that $R$ in the lemma above depends on $z$, but in order to have control of how small the constant $L'$ is we may choose $r$ sufficiently small.
	
  Moreover, 
  the set in \ref{item:prop_fn_3} where $f_n$ is not
  differentiable is
  \begin{equation*}
    \partial\U^n_{\crit} \cup \crit
    =
    \bigcup_{c\in \crit(f)}\partial U_c^n\cup \crit.
  \end{equation*}
  This is a set of measure zero for all $n\in\N$. Hence each $f_n$ is
  differentiable almost everywhere. 
\end{remark}

\begin{proof}
  \ref{item:prop_fn_1}
  We have already seen continuity before~\eqref{eq:def_fn}.  
	% Continuity is obvious on $U_c^n$ as well as on $\CDach \setminus \bigcup_{c\in \crit(f)} \overline{U}_c^n$, we need to verify that on $\partial U_c^n$ the map $f_n$ agrees with $f$. Recall that in local coordinates given by $\phi$ the map is $z^d$ that is then substituted locally by $h_n$. From Lemma \ref{lemma:hnc is continuous} it follows that the map $h_n$ is continuous and therefore $f_n$ is continuous as well.

  \smallskip
  \ref{item:prop_fn_2}
  Let $\epsilon >0$ be arbitrary. Then since $f$ is expanding
  around $p$ there exists a number $N\in\N$ such that
  $\max_{p\in\post(f)}\diam(U_p^n)\leq \epsilon$ for all $n\geq N$. So for
  $n\geq N$ we have that
  \begin{equation*}
    \norm{f_n - f}_\infty
    =
    \max_{z\in \CDach}\abs{f_n(z)-f(z)}
    =
    \max_{z\in \U^n_{\crit}} \abs{f_n(z)-f(z)}
    \leq \max_{p\in\post(f)}\diam(U_p^{n-1})
    \leq
    \epsilon,
  \end{equation*}
  proving the claim. 
  
\smallskip
\ref{item:prop_fn_3}
Lemma~\ref{lemma:hnc is continuous} implies immediately that
$f_{n+1}$ is real differentiable on $\CDach \setminus (\partial
\U^n_{\crit} \cup \crit(f))$ as claimed and that $\norm{Df_{n+1}}_{\CDach}$
has a continuous extension. On $\CDach \setminus \U^{n+1}$ we
have a lower bound for $\norm{Df_{n+1}}_{\CDach} =f^{\#}$ and on
$\U^{n+1}$ we have lower bound for $\norm{Df_{n+1}}_{\CDach}$ by
  Lemma~\ref{lemma:hnc is continuous}. 

%     This follows from Lemma \ref{lemma: h is bilipschitz} \ref{item:prop_hd_2} and the
% definition of $f_n$ as a composition of the winding map with
% holomorphic maps. In Lemma \ref{lemma: h is bilipschitz} we also
% proved that we can extend $\|Dh\|$ continuously to
% $\Sp^1\cup\{0\}$ and therefore we can extend $\|Df_n\|$
% continuously to $\partial U_c^n\cup\{c\}$.

\smallskip
\ref{item:prop_fn_4}
Suppose that
$z\in \CDach\setminus \U_{\crit}^n$. Then there exists a
neighborhood $B(z, R)$ such that $f_n(w)=f(w)$ and $f$ is
univalent on $B(z, R)$. We additionally assume that $f(B(z, R))$
is contained in a hemisphere. Let $0<r<R$ meaning that
$ B(z, r)\subset B(z, R)$. The statement is an immediate
application of (Koebe's)
Theorem~\ref{theorem: spherical Koebe}.
% to conclude that there
% exists a uniform constants $L_1, L_2$ both depending on $R/r$
% such that
% \begin{equation*}
%   L_1f^{\#}(z)|z-w|_{\hat{\mathbb{C}}}
%   \leq
%   \left|f(z)-f(w)\right|_{\hat{\mathbb{C}}}
%   \leq
%   L_2f^{\#}(z)|z-w|_{\hat{\mathbb{C}}},
% \end{equation*}
% and the claim follows.

\smallskip

To prepare the proof of \ref{item:prop_fn_3a} and
\ref{item:prop_fn_5}, fix a set
$U^{n+1}_c$ for a $c\in \crit(f)$, let $p=f(c)\in
\post(f)$. Consider $z,z_0, z_1\in 
U^{n+1}_c$; let $w,w_o,w_1\in U^n_p$ be their images by $f_{n+1}$; $z',
z'_0, z'_1\in \abs{\lambda}^{-n/d}\Db$ their images by $\phi$;
and $w', w'_0,w'_1\in \abs{\lambda}^{n}\Db$ images of
$z',z'_0,z'_1$ respectively $w,w_0,w_1$ by $h_{n+1,c}$
respectively $\psi$. We visualize this in the following diagram.
\begin{equation*}
  \xymatrix{
    z,z_0,z_1 \in U^{n+1}_c \ar[r]_{f_{n+1}} \ar[d]_{\phi}
    & w,w_0,w_1 \in U^n_p \ar[d]^\psi
    \\
    z',z_0',z_1' \in \abs{\lambda}^{-n/d}\Db \ar[r]_{h_{n+1,c}}
    & w',w'_0,w'_1\in \abs{\lambda}^{-n}\Db\rlap{.}
  }
\end{equation*}
We have the following estimates
\begin{align*}
  &\norm{D h_{n+1,c}(z')}
  = \abs{\lambda}^{-n}
  \cdot\norm{Dh_d(\abs{\lambda}^{n/d}z')}
  \cdot\abs{\lambda}^{n/d}
  =
  d \abs{\lambda}^{-n(1-1/d)},
    \text{ and similarly}
  \\
  &\norm{D h_{n+1,c}(z')}_*
  = 
    \abs{\lambda}^{-n(1-1/d)},
\end{align*}
where Lemma~\ref{lemma: h is bilipschitz}~\ref{item:prop_hd_2}
was used. 
Using Lemma~\ref{lem:Koebe_phi_psi} we obtain
\begin{align*}
  \phi^{\#}(z)
  &\asymp
  \phi^{\#}(c),
  \\
  (\psi^{-1})^{\#}(w')
  &\asymp
    (\psi^{-1})^{\#}(0) = (\psi^\#(p))^{-1},
    \text{ which yields }
  \\
  \norm{D f_{n+1}(z)}_{\CDach}
  &= (\psi^{-1})^{\#}(w') \cdot \norm{D h_{n+1,c}(z')} \cdot
    \phi^{\#}(z)
    \\
  &\asymp
    d \abs{\lambda}^{-n(1-1/d)} (\psi^{\#}(p))^{-1} \phi^{\#}(c)
    \text{ and}
  \\
  \norm{D f_{n+1}(z)}_{*\CDach}
  &\asymp
    \abs{\lambda}^{-n(1-1/d)} (\psi^{\#}(p))^{-1} \phi^{\#}(c). 
\end{align*}
Here $C(\asymp) = L(n)$ for a fixed sequence $\{L(n)\}$ with
$L(n)\searrow 1$ as $n\to \infty$. We have proved
\ref{item:prop_fn_3a}. 

\smallskip
To show \ref{item:prop_fn_5}, we continue with the previous
setting. We will prove the statements for $f_{n+1}$.
From Lemma~\ref{lemma: h is
  bilipschitz}~\ref{item:prop_hd_1} we obtain for $z_o, z_1 \in
U^{n+1}_c(z_0,\pi/d)$ ($\Leftrightarrow z'_0, z'_1\in\Db(\theta,
\pi/d)$ for some $\theta \in [0,2\pi)$).
\begin{align*}
  \abs{w'_0 - w'_1}
  =
  &\abs{h_{n+1,c}(z'_0) - h_{n+1,c}(z'_1)}
  \leq
    d \abs{\lambda}^{-n} \cdot \abs{\lambda}^{n/d}\abs{z'_0 - z'_1}
    \\
  &= d \abs{\lambda}^{-n(1-1/d)}\abs{z'_0 - z'_1}
  \intertext{and similarly}
  &\abs{h_{n+1,c}(z'_0) - h_{n+1,c}(z'_1)}
    \geq
      \abs{\lambda}^{-n(1-1/d)}\abs{z'_0 - z'_1}. 
\end{align*}
Using Lemma~\ref{lem:Koebe_phi_psi} we know that
\begin{align*}
  \abs{z'_0 - z'_1}
  &=
  \abs{\phi(z_0) - \phi(z_1)}
  \asymp
  \phi^{\#}(c)\abs{z_0- z_1}_{\CDach}, \text{ and}
  \\
  \abs{w'_0 - w'_1}
  &=
  \abs{\psi(w_0) - \psi(w_1)}
  \asymp
    \psi^{\#}(p) \abs{w_0-w_1}_{\CDach}
    \\
  =
  &\psi^{\#}(p) \cdot \abs{f_{n+1}(w_0) - f_{n+1}(w_1)}_{\CDach}. 
\end{align*}
Putting this all together, we obtain
\begin{align*}
  \abs{f_{n+1}(z_0) - f_{n+1}(z_1)}_{\CDach}
  &=
  \abs{w_0 -w_1}_{\CDach}
  \asymp
    (\psi^{\#}(p))^{-1} \abs{w'_0- w'_1}
    \\
  &\leq
    d \abs{\lambda}^{-n(1-1/d)}(\psi^{\#}(p))^{-1}\abs{z'_0-z'_1}
    \\
  &\asymp
  d \abs{\lambda}^{-n(1-1/d)}(\psi^{\#}(p))^{-1}
    \phi^{\#}(c)\abs{z_0-z_1}
  \\
  &\asymp
    \norm{Df_{n+1}(z_0)}_{\CDach} \abs{z_0-z_1}_{\CDach}, \text{ here
    \ref{item:prop_fn_3a} was used.}
\end{align*}
Similarly,
\begin{align*}
  \abs{f_{n+1}(z_0) - f_{n+1}(z_1)}_{\CDach}
  \gtrsim
  \norm{Df_{n+1}(z_0)}_{* \CDach} \abs{z_0-z_1}_{\CDach}, \text{ here
  \ref{item:prop_fn_3a} was used}. 
\end{align*}
In all the estimates above, we have $C(\asymp)= L(n)$ and
$C(\gtrsim)= L(n)$, where $L(n)\searrow 1$ as $n\to \infty$. 
\end{proof}

\subsection{Composing the maps $f_n$}
\label{sec:composing-maps-f_n}

We will now look at the compositions of the maps $f_n$ defined
above, i.e., at
\begin{equation*}
  f_1\circ\dots\circ f_n(z). 
\end{equation*}
Roughly speaking, we will replace $f^n$ by these
compositions. Furthermore, the radii of the approximating spheres
will be defined in terms of the derivative of the above
compositions.

% We denote $\U^n =
% f^{-n}(\U)$, for each $n\in \N$, and $\U^0= \U$.
% Note that $\U^n$ contains exactly one component
% $U^n_p$ for each $p\in \post(f)$. Similarly, when $n\geq 1 $, $\U^n$ contains
% exactly one component $U^n_c$ for each $c\in \crit(f)$ (since
% $U^n_c$ is the unique component of $f^{-1}(U^{n-1})_p$ containing
% $c$, where $p= f(c)$). 

Let us fix $n\geq 1$.
On $\CDach \setminus \U^n$, we have $f_1\circ \dots
\circ f_n = f^n$.

Consider now a component $U^n$ of $\U^n$, which we fix for
now. Let $U^{n-1}= f(U^n), U^{n-2}= f^2(U^n),\dots,U^0 =
f^n(U^n)$. Here $U^j\in \U^j$, for $j= \{0,\dots,n\}$. Note that
on $U^n$ we have
\begin{equation*}
  (f_1\circ\dots\circ f_n)|_{U^n}
  =
  (f_1|_{U_1}) \circ \dots \circ (f_n|_{U_n}).
\end{equation*}
In this composition we have $  f_j|_{U^j} = f$, whenever $U^j$ does not
contain a critical point. Let $k\in\{0,\dots, n\}$ be the largest
number such that $U^k= U^k_p$ contains a point
$p\in\post(f)$. Since this is the case for $U^0$, such a
$k$ always exists. Note that then $U^{k-1}= U^{k-1}_p, \dots, U^0
= U^0_p$, meaning that these sets all contain $p$.
When $k<n$, the set $U^{k+1}$ is the only set in this sequence
that may contain a critical point $c$. This means, the
mapping behavior of $f_1\circ \dots \circ f_n$ on $U^n$ is given
as follows.

% \begin{equation*}
%   U^n \stackrel{f}\longrightarrow
%   U^{n-1}\stackrel{f}\longrightarrow
%   \dots \stackrel{f}\longrightarrow
%   U^{k+1}_c \stackrel{f_{k+1}}\longrightarrow
%   U^k_p \stackrel{f}\longrightarrow
%   \dots
%   \stackrel{f}\longrightarrow
%   U^0. 
% \end{equation*}

\begin{equation}
  \label{eq:UnU0}  
  \xymatrix{
    U^n \ar[r]^f &
    U^{n-1} \ar[r]^f &
    \dots \ar[r]^f &
    U^{k+1}_c \ar[r]_{f_{k+1}} &
    U^k_p \ar[r]^f &
    \dots \ar[r]^f &
    U^0\rlap{.}} 
\end{equation}
This means
\begin{equation}
  \label{eq:f_1n_fff}
  (f_1 \circ \dots \circ f_n)|_{U^n}
  =
  f^k \circ f_{k+1} \circ f^{n-k-1}. 
\end{equation}
In particular, in these compositions, we replace $f$ only once by
a non-holomorphic map $f_{k+1}$. 

We will often use an explicit expression for the derivative of $f_1\circ \dots\circ f_n(z)$ for $z\in U_c^n$. 
\begin{lemma}
  \label{lem:Df1n}
  In the setting above we have the following.
  \begin{enumerate}    
  \item 
    \label{item:Df1n} follows immediately from 
    Let $z\in U^n$. Then
    \begin{equation*}
      \norm{D(f_1\circ\dots \circ f_n)(z)}_{\CDach}
      =
      (f^k)^{\#}(z_{k}) \cdot \norm{Df_{k+1}(z_{k+1})}_{\CDach}
      \cdot (f^{n-k-1})^{\#}(z).
    \end{equation*}
    Here $z_{k+1}= f^{n-k-1}(z)\in U^{k+1}$, $z_k =
  f_{k+1}(z_{k+1})\in U^k_p$.
\item 
  \label{item:Df1n_c}
  Let $z\in U_c^{k+1}$, where $U^{n+1}_c$ is a component of
  $\U^{k+1}_{\crit}$. Then  
  \begin{equation*}
    \norm{D(f_1\circ \dots\circ f_{k+1})(z)}_{\CDach}
    =
    (\psi^{-1})^{\#}(z_0') \cdot d \cdot |\lambda|^{k/d}\cdot \phi^{\#}(z),
  \end{equation*}
  where $z_0' = \lambda^{k}\cdot h_{k+1,c}(\phi(z))\in \Db$. Note
  that $(\psi^{-1})^{\#}(z_0') = (\psi^{\#}(z_0))^{-1}$, where
  $z_0 = f_1 \circ \dots \circ f_{k+1} (z)\in U^0_p$.
  
  Here $\phi= \phi_c$, $\psi= \psi_p$ are the local coordinates
  at $c$, respectively $p$, see and
  \eqref{eq:loc_coord_zd} and \eqref{eq:loc_coord_Up} (with
  corresponding $d$ and $\lambda$).
\end{enumerate}
\end{lemma}

\begin{proof}
  \ref{item:Df1n} follows immediately from \eqref{eq:f_1n_fff},
  see also \eqref{eq:Dfgh};
  \ref{item:Df1n_c} follows using
  Lemma~\ref{lem:prop_fn}~\ref{item:prop_fn_3a} and
  Lemma~\ref{lem:fn_Unp}~\ref{item:fn_Unp}. 
\end{proof}

% The maps $f_1\circ \dots\circ f_n$ will also occur when iterating the extension of $f$, we will examine a bit better what this composition looks like. We will explain the backward and forward orbits of points under these maps, depending on the style one might seem more natural and easy to understand but both explanations will be equivalent.

Recall that the reason we introduced the maps $f_n$ was to avoid
the contractive behavior at critical points. However, when $z\in
U^{n+1}_c$ for $c\in \crit(f)$, we have $\norm{D f_{n+1}(z)}_{\CDach} \to
0$ as $n\to \infty$ by
Lemma~\ref{lem:prop_fn}~\ref{item:prop_fn_3a}.  

However, the derivative of the compositions
$f_1\circ \dots \circ f_n$ goes to infinity. 

\begin{lemma}
  \label{lem:Df1n_infty}
  We have
  \begin{equation*}
    \min_{z\in \CDach}\norm{D(f_1\circ \dots \circ f_n)(z)}_{\CDach}
    \to \infty
    \text{ as } n\to \infty. 
  \end{equation*}
  % For all $R \geq 0$ there exists a number $N \in \N$ such that
  % for all $z\in \CDach$ and all $n\geq N$ we have that $\|D(f_1 \circ \dots \circ f_N)(z)\|_{\CDach}>R$.
\end{lemma}

To prove this, we consider two cases in the following two lemmas; points
in $\CDach \setminus \U^n$ and points in $\U^n$.

\begin{lemma}
  \label{lem:Df1n_CUn}
  We have the following:
  \begin{equation*}
    \inf\{(f^n)^{\#}(z) : z\in \CDach \setminus \U^n\}
    \to \infty,
    \text{ as }
    n\to \infty.
  \end{equation*}
\end{lemma}

Note that on $\CDach \setminus \U^n$ we have $f_1\circ \dots
\circ f_n = f^n$, meaning that $\norm{D(f_1 \circ \dots \circ
  f_n)}_{\CDach} = (f^n)^{\#}$ on this set. 

\begin{proof}
  Let us consider a finite open cover $\V$ of $\CDach$ with the
  following properties.
  \begin{itemize}
  \item
    Each $V\in \V$ is an open (spherical) ball with $\diam(V)<
    \delta_0$, where $\delta_0$ is the constant from
    Definition~\ref{def:exp_Tmaps}.
  \item
    Each $V\in \V$ contains a closed disk $D= D(V)$. 
  \item We consider the subset of $\V$ not containing
    postcritical points, i.e.,
    \begin{equation*}
      \V' = \{V\in \V : V\cap \post(f) = \emptyset\}.       
    \end{equation*}
    We assume the corresponding closed disks cover the complement
    of $\U^0$:
    \begin{equation*}
      \bigcup_{V\in \V'} D(V) \supset \CDach \setminus \U^0. 
    \end{equation*}
  \end{itemize}
  Clearly such an open cover $\V$, together with the
  corresponding sets $D= D(V)$, exists.
  
  Let $\V^n$ be the set of components of $f^{-n}(V)$ for
  $V\in\V$ and $n\in \N$. 
  From the definition of $\delta_0$, we know that
  \begin{equation}
    \label{eq:meshto0}
    \mesh(\V^n) \to 0
    \text{ as }
    n\to \infty. 
  \end{equation}
  Consider now an $V\in \V'$, the corresponding set $D=D(V)$, and
  a component $V^n$ of $f^{-n}(V)$. We denote the inverse of the
  conformal map 
  $f^n\colon V^n \to V$ by $g=g_{V^n}$. Let $z\in g(D) \subset V^n$ be
  arbitrary.  
  From \cite[Lemma~A.2]{BonkMeyer2017} it follows that
  \begin{equation*}
    \diam(V^n)
    \geq
    \diam(g(D))
    \asymp
    g^{\#}(w)
    =
    [(f^n)^{\#}(z)]^{-1}.
  \end{equation*}
  Here $C(\asymp) = C(V,D)$ and $w= f^n(z)$. Since there are only
  finitely many constants $C(V,D)$, we can assume that
  $C(\asymp)$ is a single constant (depending on our chosen cover
  $\V$ but not on $n\in \N$ or $V^n$). 
  It follows that
  \begin{align*}
    \mesh(\V^n)
    % &\geq
    %   \mesh\{\text{component of } f^{-n}(V') : V'\in \V'\}
    %   \\
    &\geq
      \max\{\diam(V^n) : V^n\in \V \text{ with }
      f^n(V^n) = V' \in \V'\}
    % & \phantom{XXXXX}\text{
    %   with } f^n(V^n) = V' \in \V'
    %
    \\
    &\gtrsim
      [(f^n)^{\#}(z)]^{-1},
  \end{align*}
  with a fixed constant $C=C(\gtrsim)$. 
  Here $z\in f^{-n}(D)$ for any $D=D(V')$ with $V'\in \V'$. Since
  the sets $D(V')$ cover $\CDach \setminus \U^0$, it follows that
  the sets $f^{-n}(D)$ cover $\CDach \setminus \U^n$. Together
  with \eqref{eq:meshto0}, the claim follows.   
\end{proof}

Consider now the second case.

\begin{lemma}
  \label{lem:DfnUn}
  Let $z\in \U^n$. Then
  \begin{equation*}
    \min_{z\in \U^n}\norm{D(f_1\circ \dots \circ f_n)(z)}_{\CDach}
    \to \infty
    \text{ as }
    n\to \infty.
  \end{equation*}
\end{lemma}

\begin{proof}
  Consider a component $U^n$ of $\U^n$. Consider the sequence of
  sets $U^n, U^{n-1},\dots, U^0$, where $U^{n-j}= f^j(U^n)$ for
  $j=0,\dots,n$. As in \eqref{eq:UnU0}, let $k$ be the smallest
  index such that $U^k= U^k_p$ contains a $p\in \post(f)$. Then
  $U^{k-1}= U^{k-1}_p,\dots, U^0_p$, meaning all these sets
  contain $p$. Let $\lambda$ be the multiplier of $p$.

  When $k=n$, it follows that $f_1\circ \dots \circ f_n = f^n$ on
  $U^n$. Hence
  \begin{equation*}
    \norm{Df}_{\CDach}
    =
    (f^n)^{\#}
    \asymp
    \abs{\lambda}^n,
  \end{equation*}
  using Lemma~\ref{lem:fn_Unp}~\ref{item:fn_Unp}, where $C(\asymp)$ is a
  constant. The statement follows in this case.

  \smallskip
  We assume from now on that $n> k$. From \eqref{eq:f_1n_fff} it
  follow that for $z\in U^n$ we have
  \begin{equation*}
    \norm{D(f_1\circ\dots \circ f_n)(z)}_{\CDach}
    =
    (f^k)^{\#}(z_{k}) \cdot \norm{Df_{k+1}(z_{k+1})}_{\CDach}
    \cdot (f^{n-k-1})^{\#}(z).
  \end{equation*}
  Here $z_{k+1}= f^{n-k-1}(z)\in U^{k+1}$, $z_k =
  f_{k+1}(z_{k+1})\in U^k_p$. Let $c$ be the unique point in
  $U^{k+1}= U^{k+1}_c$ with $f(c) =p$ and $d= \deg(f,c)$. Using
  Lemma~\ref{lem:fn_Unp}~\ref{item:fn_Unp}, as well as 
  Lemma~\ref{lem:prop_fn}~\ref{item:prop_fn_3a} or
  Lemma~\ref{lem:fn_Unp}~\ref{item:fn_Un}, we obtain
  \begin{equation*}
    \norm{D(f_1\circ\dots \circ f_n)}_{\CDach}
    \gtrsim
    \abs{\lambda}^{k/d} \cdot M(n-k-1),
  \end{equation*}
  with a constant $C(\gtrsim)$. 
  Here
  \begin{equation*}
    M(j)
    \coloneqq
    \inf\{(f^j)^{\#}(z) : z\in \CDach \setminus \U^j\},
  \end{equation*}
  for $j\in \N$, and $M(0)\coloneqq 1$. Note that $M(j)>0$ for all $j\in
  \N_0$. Lemma~\ref{lem:Df1n_CUn} says that $M(j) \to
  \infty$ as $j\to \infty$. Let us also define
  \begin{align*}
    M_* &\coloneqq \min_{n\in \N_0} M(n), \text{ and }\\
    m(n) &\coloneqq \min_{j\geq n} M(n). 
  \end{align*}
  Note that $k\geq \lfloor n/3 \rfloor$ or $n-k-1\geq \lfloor n/3
  \rfloor$. Thus
  \begin{align*}
    \norm{D(f_1\circ\dots\circ f_n)(z)}_{\CDach}
    \gtrsim
    \min\{M_*\abs{\lambda}^{\lfloor n/3d \rfloor}, m(\lfloor n/3
    \rfloor)\}
    \to \infty,
  \end{align*}
  as $n\to \infty$.
\end{proof}

\subsection{Sequence of radii and approximating Spheres}
\label{section: sequence of radii}
In this section we will define the functions $r_n(z)$ in terms of the composition $f_1\circ\dots\circ f_n$ and parametrize the approximating spheres. 

\begin{definition}\label{definition: r_n}
	Let $z\in \CDach$ and let $n\in \N_{\geq 1}$. We
        define the $n$-th \textit{radius} by
        \begin{equation*}r_n(z)=
		1-s\, \frac{1}{\|D(f_1\circ\dots\circ f_n)(z)\|_{\CDach}}
		\end{equation*}
	where $s\in(0, 1)$ is a constant chosen such that $r_n(z)\in (0, 1)$ for all $z\in \CDach$ and all $n\in \N$. We will set $r_0 :=\frac{1}{2}\min_{z\in\CDach, n\geq 1}\{r_n(z)\}$ and $r_0(z)=r_0$.
\end{definition}
	In Corollary \ref{corollary: r_n in (0, 1)} we will show that such a constant $c$ exists.

\begin{corollary}\label{corollary: r_n in (0, 1)}
	Let $r_n$ be as in Definition \ref{definition: r_n}. There exists a constant $s\in (0, 1]$ such that $r_n(z)\in (0, 1)$ for all $z\in \CDach$ and all $n\in \N$.
\end{corollary} 
\begin{proof}
	If there is no $z\in \CDach$ and no $n\in \N$ so that \begin{equation*}
		1-\frac{1}{\|D(f_1\circ \dots\circ f_n)(z)\|_{\CDach}}\notin (0, 1)
	\end{equation*}
	we set $s=1$. 
	Else, we have shown in Lemma \ref{lem:Df1n_infty} there exists $N\in \N$ such that for all $z\in \CDach$ and all $n\geq N$ the derivative $\|D(f_1\circ \dots\circ f_n)(z)\|_{\CDach}>1$ and therefore $\|D(f_1\circ \dots\circ f_n)(z)\|_{\CDach}^{-1}\in (0, 1)$. 
	
	Because the functions \begin{equation*}
		z\mapsto \frac{1}{\|D(f_1\circ\dots\circ f_n)(z)\|_{\CDach}}
	\end{equation*}
	are continuous for all $n$ and $\CDach$ is compact there exists a maximum \begin{equation*}
		\max_{n\leq N}\bigg\{\frac{1}{\|D(f_1\circ\dots\circ f_n)(z)\|_{\CDach}}\bigg\}.
	\end{equation*}
	Then we choose $s> \max_{n\leq N}\{\|D(f_1\circ \dots\circ f_n)(z)\|_{\CDach}: z\in \CDach\}$ and we have shown the claim.
\end{proof}

\begin{definition}
	We define a sequence of sets $S_n\subset B(0, 1)$ called \textit{approximating spheres} by \[S_n:= \{(z, r_n(z)): z\in \CDach\}\] for all $n\in\N$. 
\end{definition}
For an illustration see \ref{figure: non-monotone spheres}.

\begin{figure}\label{figure: non-monotone spheres}
	\begin{center}
		\begin{overpic}[scale=0.7]{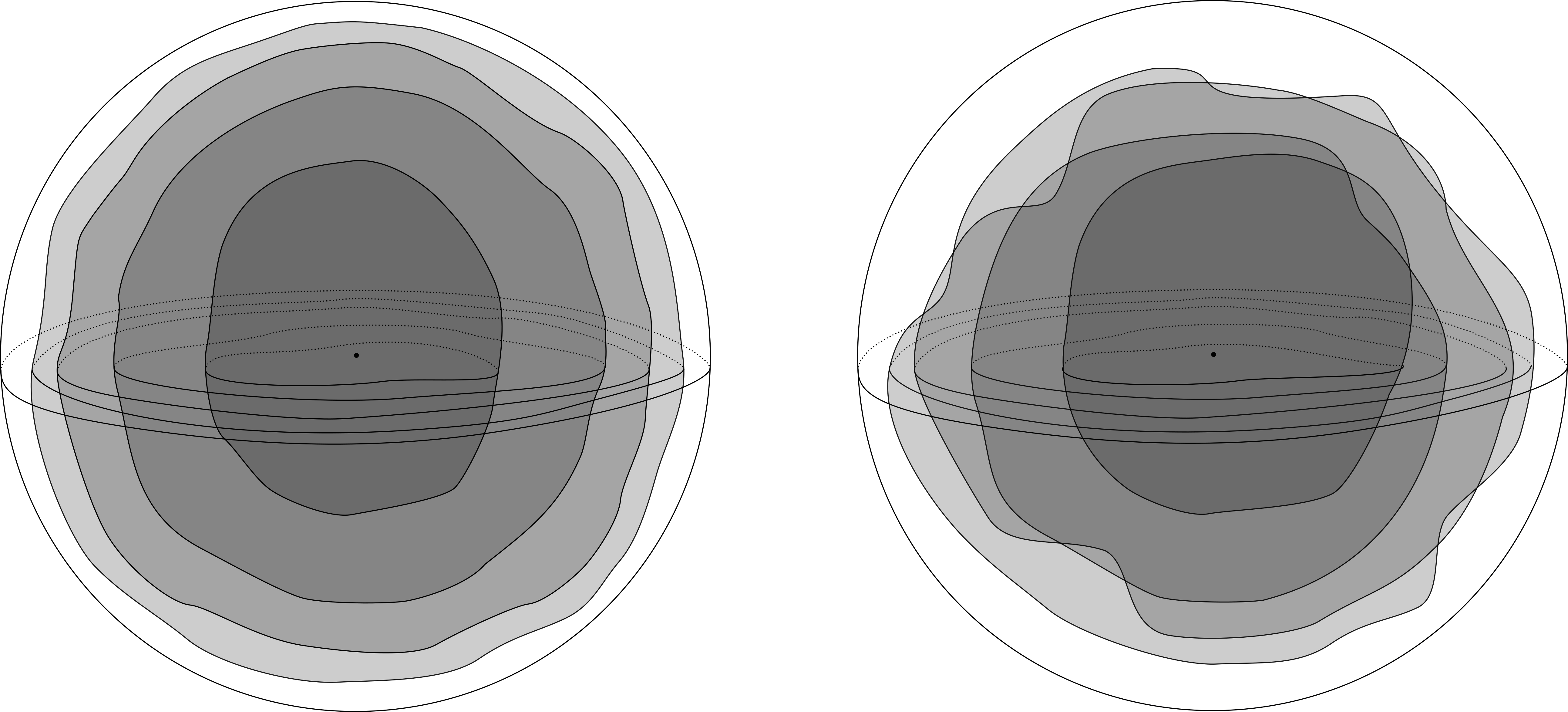}
			\put(15, 28){\footnotesize $S_1$}
			\put (10.5, 30){\footnotesize $S_2$}
			\put(6.5, 32){\footnotesize $S_3$}
			\put(69, 28){\footnotesize $S_1'$}
			\put(65, 29){\footnotesize $S_2'$}
			\put(60.5, 29){\footnotesize $S_3'$}
		\end{overpic}
	\end{center}
	\caption{This illustrates the spheres $S_1, \dots, S_4$ in the case where the approximating spheres are monotone and in the case where the spheres are not monotone and $S_1'\cap S_2'\neq \emptyset$. The degree of the map $f'$ giving rise to $S_1', \dots, S_4'$ is 2.}
\end{figure}
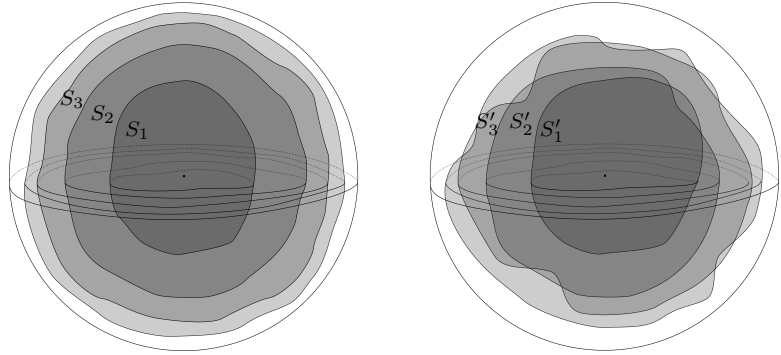

\begin{lemma}
	\label{lemma: approximating spheres}
	Let $S_n$ be the sequence above. We have that \begin{enumerate}
		\item 
		\label{item: approximating spheres 1}
		the approximating spheres accumulate at $\CDach$, i.e., for all $\epsilon >0$ there exists an $N \in \N$ such that $1>r_n(z)\geq 1-\epsilon$ for all $n\geq \N$ and all $z\in \CDach$
		\item 
		\label{item: approximating spheres 2}
		the radius is continuous in $z$ and hence $S_n $ is homeomorphic to $\Sp^2$ for all $n\in\N$.
	\end{enumerate}
\end{lemma}
\begin{proof}
	\ref{item: approximating spheres 1}
	Let $\epsilon >0$ be arbitrary. By Lemma \ref{lem:Df1n_infty} there exists $N\in \N$ so that $\|D(f_1\circ\dots\circ f_{n+1})(z)\|_{\CDach} > s/\epsilon$ for all $n\geq N$. By the definition of $r_n$ we then obtain that for all $z\in \CDach$ \begin{equation*}
		r_n(z)
		= 1-s\frac{1}{\|D(f_1\circ \dots\circ f_{n+1})(z)\|_{\CDach}}> 1-\epsilon
	\end{equation*}
	when $n\geq N$.

	\smallskip
	\ref{item: approximating spheres 2}
	By Lemma \ref{lem:prop_fn} \ref{item:prop_fn_3} have that each $f_n$ is differentiable almost everywhere and the derivative can be continuously extended to the points where the function is not differentiable. So by the chain rule, $r_n$ is continuous. And since $r_n(z)> 0$ for all $z\in \CDach$ and all $n\in \N$ the $n$-spheres are all homeomorphic to $\Sp^2$.	
\end{proof}

\begin{definition}
	Let $S_n$ be the sequence of approximating spheres. If we have that $r_n(z)<r_{n+1}(z)$ for all $n\in \N$ and all $z\in\CDach$ we call the sequence $S_n$ \textit{monotone}.
\end{definition}

\begin{remark}
	It is possible that two approximating spheres intersect. However there exists a number $N_0\in \N$ such that $S_1 \cap S_{N_0}=\emptyset$ and $S_{k}\cap S_{k+N_0}=\emptyset$. We can get a monotone sequence $\Tilde{S}_n$ by relabeling $\Tilde{S}_n= S_{n\cdot N_0}$. 
\end{remark}

\begin{corollary} \label{cor:S1_in_SN0}
	There exists a number $N_0\in \N$ such that $S_1\subset \textup{interior}(S_{N_0})$.
\end{corollary}
\begin{proof}
	We can take $N_0=N$ where $N$ is as in Lemma \ref{lem:Df1n_infty}.
\end{proof}
Recall that in case the approximating spheres are monotone, $K=2$. 

\medskip

\section{Definition and quasi-regularity of the extending map}
\label{section:Def_and_qr_of_Extension}
In this section we will define the extension of $f$ to a map $F: \Omega \to \R^3$ so that $F^n|_{\Omega_n}$ is $K$-quasi-regular form some $K$ and all $n\in \N$. The set $\Omega_n:= F^{-(n-1)}(\Omega)$. Recall that in the outline \eqref{eq:def_extension_outline} we have defined an extension $F|_{S_n}$ restricted to the approximating spheres. Let $p=(z, r_n(z))\in S_n$, then 
\begin{equation*}
	F(z, r_n(z))
	=
	(z', r_{n-1}(z'))
\end{equation*}
for $z'=f_n(z)$. 

In case the approximating spheres are monotone, we will extend the maps $F|_{S_n}$ to a map $F: A(\CDach, S_2)\to \R^3$. When the approximating spheres are not monotone we will only define an extension on $A(\CDach, S_{N_0})$, where $N_0\in \N$ is chosen so that $S_1\subset \textup{interior}(S_{N_0})$. We choose this restriction, because for points $(z, r)$ with $z\in \CDach\setminus \U_{\crit}^n$ we will define $F$ on line segments of the form 
\begin{equation*}
	(z, (1-t)\cdot r_n(z)+t\cdot r_m(z))
\end{equation*}
where $n, m\in \N$ are so that they bound a connected component of $[0, 1]\setminus \{r_n(z): n\in \N\}$, $z\in \CDach $ and $t\in [0, 1]$. Such a line segment will be mapped to a line segment of the form 
\begin{equation*}
	(z', (1-t)\cdot r_{n-1}(z')+t\cdot r_{m-1}(z'))
\end{equation*}
with $z':=f(z)$. That is the reason we wish to avoid $n$ or $m$ being equal to 1 and hence we restrict to $A(\CDach, S_{N_0})$.
When the approximating spheres are monotone the domain is $A(\CDach, S_2)$. 

Let us now turn to an outline of the arguments why the maps $F|_{S_n}$ can be extended to a uniformly quasi-regular map. We will first restrict to a subset of $A(\CDach, S_{N_0})$ that are sufficiently far from critical axes (i.e., from $\{(c, r): c\in \crit(f), r\in (0, 1]\}$). 
Let $n\in \N$ be arbitrary and $z\in \CDach \setminus \U_{\crit}^n$. The spherical component of the map $F|_{S_n}$ is defined by the holomorphic map $f$ on $\CDach$. Let $U\subset\CDach$ be a domain containing $z$. We may choose $U$ so that $f|_U$ is univalent and we can use Koebe's distortion theorem \ref{theorem: spherical Koebe} to conclude that for all $w\in B(z, \rho)\subset U$ we have that 
\begin{equation*}
	|f(z)-f(w)|_{\hat{\mathbb{C}}}\asymp f^{\#}(z)|z-w|_{\hat{\mathbb{C}}}
\end{equation*} 
with $C(\asymp)=C(U, \rho)$, i.e., $C(\asymp)$ depends on how $B(z, \rho)$ is contained in $U$. Let us turn to the distance of $(z, r_n(z))$ to $\CDach$. For a point $(z, r_n(z))$ the distance to $\CDach$ is given by 
\begin{equation*}
	d_n(z)=\textup{dist}\big((z, r_n(z)), \CDach\big)=s\,\frac{1}{(f^n)^{\#}(z)}
\end{equation*} 
and in the image it is, writing $z'=f(z)$,
\begin{equation*}
	d_{n-1}(z')=\textup{dist}\big((z', r_{n-1}(z')\big), \CDach)=s\, \frac{1}{(f^{n-1})^{\#}(z')}.
\end{equation*}
It is immediate from the chain rule that
\begin{equation*}
	d_{n-1}(z')=d_n(z)\cdot f^{\#}(z).
\end{equation*} 
Then since the stretch factor in the spherical component that we obtain from Koebe's distortion theorem roughly agrees with the factor by which the distance to $\CDach$ is increased, we can argue that the map is quasi-regular. 

If $z\in \CDach\setminus \U_{\crit}^n$ is so that $f^k(z)\notin \U_{\crit}^{n-k}$ for all $k\leq m$ for some $m\leq n$, we can argue similarly to the above that $F|_{S_{n-m}}\circ\dots \circ F|_{S_n}$ is quasi-regular in a neighborhood of $(z, r_n(z))$. Let $V\subset \CDach\setminus \U_{\crit}^n$ be a domain containing $z$ so that $f^m|_V$ is holomorphic and univalent. By the definition of $F|_{S_n}$ the spherical part $z$ is mapped as $f^m(z)$. Then by Theorem \ref{theorem: spherical Koebe} we obtain that for all $w\in B(z, \rho)\subset V$
\begin{equation*}
	|f^m(z)-f^m(w)|
	\asymp 
	(f^m)^{\#}(z)|z-w|_{\CDach}
\end{equation*}
where $C(\asymp)=C(U, \rho)$ does not depend on $m$. 
As above we obtain by the chain rule that the distance of $z''=f^m(z)$ to $\CDach$ is given by 
\begin{equation*}
	d_{n-m}(z'')
	= 
	\textup{dist} ((z'', r_{n-m}(z'')), \CDach)
	= 
	d_n(z)\cdot (f^m)^{\#}(z),
\end{equation*}
and so again the stretch factor in the spherical part and factor by which the distance to $\CDach$ is increased roughly agree.
This implies that the extension $F$ is uniformly quasi-regular for all admissible iterates. The extension of $f$ to such neighborhoods is treated in Section \ref{subsection: extension far from critical points}. 

When $(z, r)$ is a point so that the spherical coordinate $z\in \U_{\crit}^n$ we will extend the maps $F|_{S_n}$ slightly differently to a neighborhood of that point. The main idea is to interpolate between the maps $F|_{S_n}$ and $F|_{S_{n+1}}$. This will be treated in Section \ref{subsection: extension near critical points}.

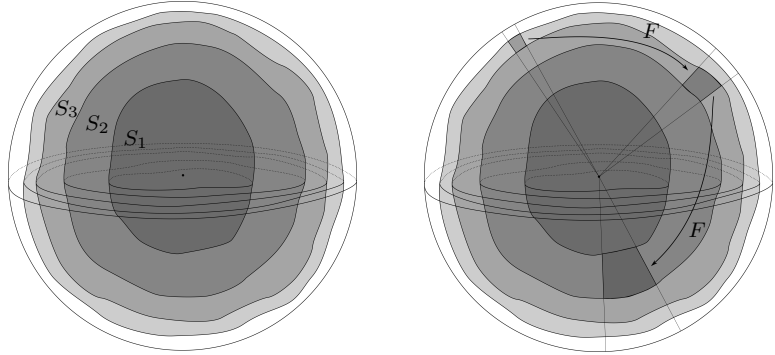
\begin{figure}[htb]
	\begin{overpic}[scale=0.7]{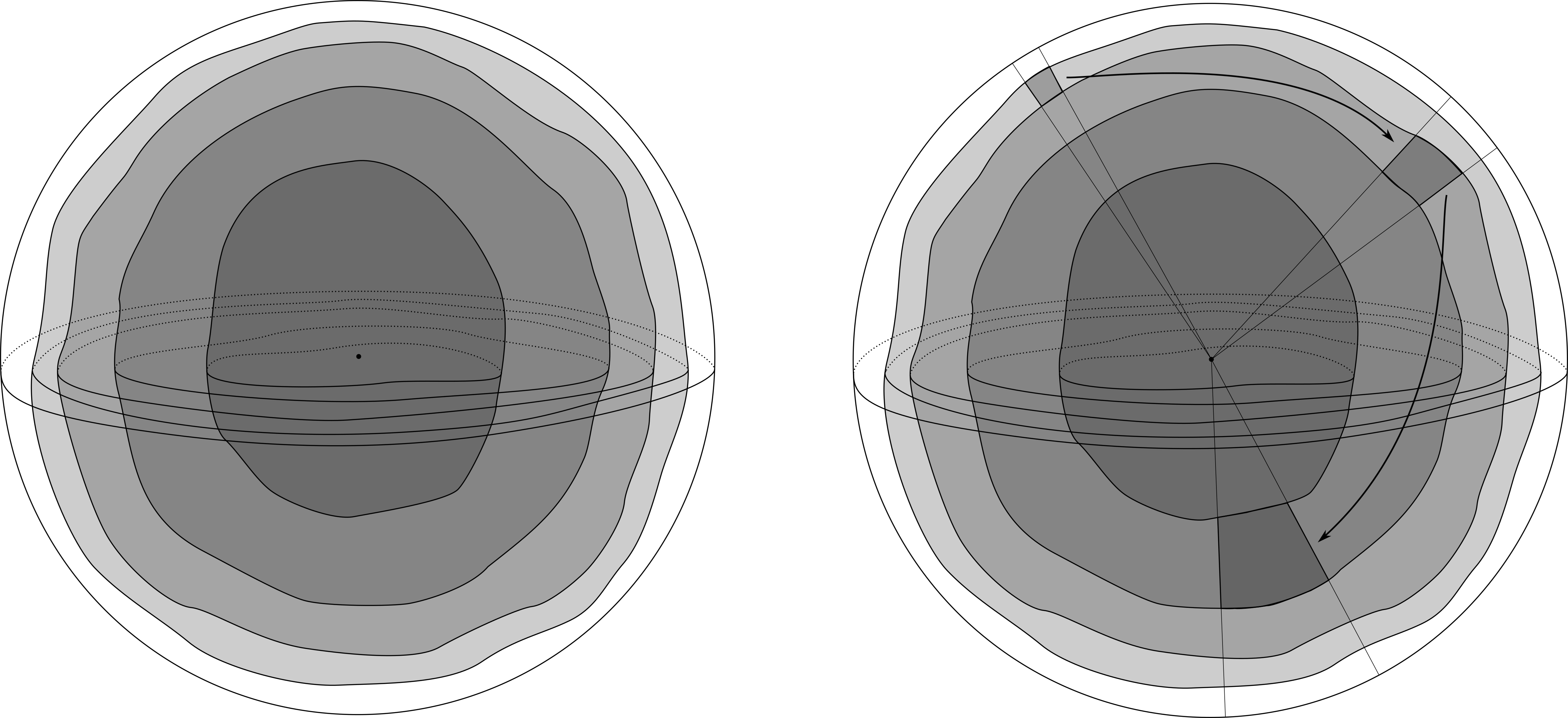}
		\put(89, 15){\footnotesize $F$}
		\put(15, 27){\footnotesize $S_1$}
		\put (10, 29){\footnotesize $S_2$}
		\put(6, 31){\footnotesize $S_3$}
		\put(83, 41){\footnotesize $F$}
	\end{overpic}
	\label{figure: Extension}
	\caption{This image shows the approximating spheres $S_1, S_2, S_3$ and $S_4$ as well as the extension $F$ schematically.}
\end{figure}

In both cases we will use the following coordinates. 

\begin{definition}[Interpolating coordinates]\label{definition: interpolating coordinates}
	Let $z\in\CDach$ and $n\in\N$ be arbitrary. Let $m\in\N$ be such that $r_n(z)$ and $r_m(z)$ bound a connected component of $[0, 1]\setminus \{r_j(z): j\in \N\}$ and $r_m(z)> r_n(z)$. Then we define \begin{equation}\label{equation: interpolating coordinates}
		\rho_n(z, t):= (1-t)\cdot r_n(z)+t\cdot r_m(z),
	\end{equation}
	for $t\in[0, 1]$. We say that $m$ is a \textit{consecutive index of $n$ at $z\in\CDach$}.
\end{definition}

\begin{remark}
	If the approximating spheres are monotone, i.e., $r_n(z)<r_{n+1}(z)$ for all $n\in\N$ we have that the consecutive index of $n$ is equal to $n+1$ for all $n\in\N$. 
	
	It is obvious that all points $p\in A(\CDach, S_1)$ can be written in coordinates of the form $(z, \rho_n(z, t))$ for $z\in\CDach$, $n\in\N$ and $t\in[0, 1]$ suitable. However note that the coordinate representation is not necessarily unique. We will show that this does not raise any problems for having a well-defined extension in Lemma \ref{lemma: extension is well-defined and continuous}.
\end{remark}

\begin{definition}[Distance coordinates] 
	For the distance of a point $z\in S_n$ to the unit sphere we use the notation \[d_n(z):= 1-r_n(z)\] and for the points in between the approximating spheres $S_n$ we use \[\delta_n(z, t):=1-\rho_n(z, t).\]
\end{definition}

\subsection{Extension away from critical axes}\label{subsection: extension far from critical points} In this section we will treat the case where a point with coordinates $(z, r)$ is sufficiently far from a critical axis $l_c=\{(c, t):t\in[0, 1]\}$ for a critical point $c$. The term "sufficiently far" will be made precise in Definition \ref{definition: extension far from critical axes}.
\subsubsection{Monotone approximating spheres} Fix a point $z\in\CDach$ and consider the line $l_z=\{(z, r):r\in [0, 1]\}$. Then the interpolating coordinates of a segment of $l_z$ are given by \begin{equation*}
	\rho_n(z, t):= (1-t)\cdot r_n(z)+ t\cdot  r_{n+1}(z).
\end{equation*}

We suppose that $z\notin \U_{\crit}^n$. At the point $f(z)$ we have a similar subdivision of $l_{f(z)}$. The coordinates $\rho_{n-1}(f(z), t)$ are given by \begin{equation*}
	\rho_{n-1}(f(z), t)=(1-t)\cdot r_{n-1}(f(z))+t\cdot r_n(f(z)).
\end{equation*} 
We will map the line segment at $z$ parametrized by $[0, 1]\to\R^3, \, t\mapsto  (z, \rho_n(z, t))$ to the segment parametrized by \begin{equation*}
	\big(f(z), \rho_{n-1}(f(z), t)\big)
\end{equation*} 
for $t\in[0, 1]$.
We define \begin{equation*}
	F\big(z, \rho_n(z, t)\big) = \big(f(z), \rho_{n-1}(f(z), t)\big).
\end{equation*}

\subsubsection{General case} Here $N_0$ might be larger than 2. Fix $z\in\CDach$ and $n\in\N$. Suppose that $z\notin \U_{\crit}^n$. There is an index $m\in\N$ such that the interpolating coordinates $\rho_n(z, t)$ are given by \begin{equation*}
	\rho_n(z, t)=(1-t)\cdot r_n(z)+t\cdot r_m(z).
\end{equation*}
We will show in Lemma \ref{lemma: reihenfolge bleibt erhalten} that at the point $f(z)$ a consecutive index of $n-1$ is equal to $m-1$. 
Then we can define the extending map by \begin{equation*}
	F\big(z, \rho_n(z)\big)=\big(f(z), \rho_{n-1}(f(z), t)\big).
\end{equation*}

For proving that the map it is crucial that if $m$ is the consecutive index of $n$ at $z\in\CDach$ it follows that a consecutive index of $n-1$ at $f(z)$ is equal to $m-1$. In the case that there exists $k\neq m$ with $r_k(z)=r_m(z)$ we can show that also $r_{k-1}(f(z))=r_{m-1}(f(z))$.

\begin{definition}[Extension] \label{definition: extension far from critical axes}
	Let $p\in A(\CDach, S_{N_0})$ be a point such that there are coordinates $p=(z, \rho_n(z, t))$ with $z\in\CDach$ and $n, t$ chosen suitably with $z\notin U_c^n$. Then we can define the extension $F$ by \begin{equation}\label{equation: definition of F outside U_c^n}
		F\big(z, \rho_n(z, t)\big)=\big(f(z), \rho_{n-1}(f(z), t)\big).
	\end{equation}
\end{definition}

In case the approximating spheres are not monotone, it is possible that for distinct points $z, w\in\CDach$ the coordinates $\rho_n(z, t)$ and $\rho_n(w, s)$ are given by \begin{align*}
	\rho_n(z, t)&=(1-t)\cdot r_n(z)+t\cdot r_m(z)\\
	\rho_n(w, s)&=(1-s)\cdot r_n(w)+s\cdot r_k(w)
\end{align*}
for $m\neq k$. For proving quasi-regularity of the extension it will be convenient to subdivide $A(\CDach, S_{N_0})$ into sets where the interpolating coordinates have the same indices on the whole set. These sets will be called prisms. 

\begin{definition}\label{definition: n-prism}
	Let $n\in \N$ be arbitrary. Suppose that for a connected subset $U\subset \CDach$ there exists an $m\in\N$ that is a consecutive index of $n$ at all points $z\in U$. Then we define the \textit{$n$-prism above $U$} by \begin{equation*}
		P(U,n):=\{(z, \rho_n(z, t)):z\in U, t\in [0, 1]\}.
	\end{equation*}
\end{definition}

\begin{lemma}\label{lemma: extension is well-defined and continuous}
	Let $p\in A(\CDach, S_{N_0})$ be a point such that there are coordinates $p=(z, \rho_n(z, t))$ for $z\in \CDach\setminus \U_{\crit}^n$ and $n\in\N, t\in [0, 1]$ suitable. Then the extension defined in \ref{definition: extension far from critical axes} is well-defined and continuous at $p$. 
	
	Let $P(U, n)$ be an $n$-prism containing $z$. Then $F(P(U, n))= P(f(U), n-1)$ is an $(n-1)$-prism containing $f(z)$. 
\end{lemma}

In order to show that $F$ is well-defined we will prove that if $m$ and $k$ are consecutive indices of $n$ at $z$ then $m-1$ and $k-1$ are consecutive indices of $n-1$ at $f(z)$. We will show this in Lemma \ref{lemma: reihenfolge bleibt erhalten} and Lemma \ref{lemma: if f(z)in U_c^n}.

\begin{lemma}\label{lemma: reihenfolge bleibt erhalten}
	Let $n\in \N$ be arbitrary and $z\in\CDach\setminus \bigcup_c U_c^n$ be so that $f^j(z)\notin \bigcup_c U_c^{n-j}$ for $j=1, \dots, n-1$. Suppose that $m$ is a consecutive index of $n$ at $z$. Then $m-1$ is a consecutive index of $n-1$ at $f(z)$. If for $k\neq m$ we have $r_m(z)=r_k(z)$ then $r_{m-1}(f(z))=r_{k-1}(f(z))$. 
\end{lemma}

\begin{proof}[Proof of Lemma \ref{lemma: reihenfolge bleibt erhalten}]
	We assume that $m>n$, the other case is very similar. The claim follows from the chain rule. We have that $r_n(z)< r_m(z)$, which implies that \begin{equation}\label{equation: equality of spherical derivatives}
	(f^m)^{\#}(z)> (f^n)^{\#}(z).
\end{equation}
Rewrite the derivative as 
\begin{equation*}
	(f^m)^{\#}(z)
	= 
	f^{\#}(f^{m-1}(z)) \dots f^{\#}(f^n(z))(f^n)^{\#}(z),
\end{equation*}
and conclude that 
\begin{equation}
	\label{equation: order of radii for easy case}f^{\#}
	(f^{m-1}(z))\dots f^{\#}(f^n(z)) > 1.
\end{equation}

At the point $f(z)$ the derivatives involved in the definition of $r_{n-1}(f(z))$ and $r_{m-1}(f(z))$ are 
\begin{align*}
	(f^{n-1})^{\#}(f(z))&=f^{\#}(f^{n-1}(z))\dots f^{\#}(f(z))\\
	(f^{m-1})^{\#}(f(z))&= f^{\#}(f^{m-1}(z))\dots f^{\#}(f(z))\\
	&=f^{\#}(f^{m-1}(z))\dots f^{\#}(f^n(z))(f^{n-1})^{\#}(f(z)).
\end{align*}
Observing that the factor $f^{\#}(f^{m-1}(z))\dots f^{\#}(f^n(z))$ that is not equal in both expressions is the same factor as above and strictly larger than 1 we conclude
\begin{equation*}
	(f^{m-1})^{\#}(f(z))> (f^{n-1})^{\#}(f(z))
\end{equation*}
and hence $r_{n-1}(f(z))< r_{m-1}(f(z))$. 
For the second part, let $k<m$ be another consecutive index of $n$ at $z$. Then \begin{equation*}
	(f^m)^{\#}(z)= f^{\#}(f^{m-1}(z))\dots f^{\#}(f^k(z))(f^k)^{\#}(z)
\end{equation*}
and hence $f^{\#}(f^{m-1}(z))\dots f^{\#}(f^k(z))=1$. At the point $f(z)$ we can compute \begin{equation*}
	(f^{m-1})^{\#}(f(z))= f^{\#}(f^{m-1}(z))\dots f^{\#}(f^k(z))f^{\#}(f^{k-1})^{\#}(f(z)) = (f^{k-1})^{\#}(f(z))
\end{equation*}
which shows that $k-1 $ and $m-1 $ are both consecutive indices of $n-1$ at $f(z)$. 
\end{proof}

Note that here we excluded some cases that we aimed treat already in this section, namely if $z\in\CDach\setminus \U_{\crit}^n$ but $f(z)\in \U_{\crit}^{n-1}$. We will prove the same implications as in Lemma \ref{lemma: reihenfolge bleibt erhalten} for this case.

\begin{lemma}\label{lemma: if f(z)in U_c^n}
	Let $z\in \CDach\setminus \U_{\crit}^n$ be a point such that $f(z)\in \U_{\crit}^{n-1}$. If $m\in\N$ is a consecutive index of $n$ at $z$ then $m-1$ is a consecutive index of $n-1$ at $f(z)$.  
\end{lemma}
\begin{proof}
	We assume that $m>n$, the other case is similar. If $r_m(z)> r_n(z)$ we have that \begin{equation*}
		\| D(f_1\circ\dots\circ f_m)(z)\|_{\CDach}> \|D(f_1\circ \dots \circ f_n)(z)\|_{\CDach}. 
	\end{equation*}
	By the considerations about the sequence $f_1\circ\dots\circ f_m$ and $f_1\circ \dots\circ f_n$ and the chain rule it follows \begin{align*}
		\|D(f_1\circ\dots\circ f_m)(z)\|_{\CDach} &=\|D(f_1\circ\dots\circ f_{m-1})(z)\|_{\CDach} f^{\#}(z)\\
		\|D(f_1\circ\dots\circ f_n)(z)\|_{\CDach} &=\|D(f_1\circ \dots\circ f_{n-1})(f(z))\|_{\CDach} \cdot f^{\#}(z)
	\end{align*}
	therefore \begin{equation*}
		\|D(f_1\circ\dots\circ f_{m-1})(f(z))\|_{\CDach} > \|D(f_1\circ\dots\circ f_{n-1})(f(z))\|_{\CDach}.
	\end{equation*} 
	Note that these are precisely the derivatives appearing in the definition of $r_{n-1}(f(z))$ and $r_{m-1}(f(z))$ and that the inequality implies $r_{m-1}(f(z))> r_{n-1}(f(z))$. 
\end{proof}

\begin{remark}
	In the next section, Lemma \ref{lemma:mon_radii_crit}, will show later that for all $z\in \U_{\crit}^n$ we have that $n+1$ is the consecutive index of $n$ for all $n$.
\end{remark}

\begin{example}
	Intuitively from Lemma \ref{lemma: reihenfolge bleibt erhalten} it follows that the shape of an $n$-prism is roughly preserved. Suppose there is a small region $U\subset \CDach$ such that $r_2(z)\leq r_1(z)$ that is bounded by $\{w: r_1(w)=r_2(w)\}$ and for $z$ outside of $U$ we have that $r_2(z)> r_1(z)$. Then it follows that there exist $d$ regions $f^{-1}(U)$ that are preimages under $f$ with $r_3(z)\leq r_2(z)$ for $z\in f^{-1}(U)$ that are bounded by $\{w\in f^{-1}(\partial U): r_3(w)) =r_2(w)\}$. See Figure \ref{figure: non-monotone spheres} for an illustration.
\end{example}

\begin{proof}[Proof of Lemma \ref{lemma: extension is well-defined and continuous}]
	In order to prove that the extension is well-defined let $z\in\CDach$ and $n\in \N$ be so that there exist distinct $m, k\in\N$ with \begin{equation*}
		\big(z, (1-t)\cdot r_n(z)+t\cdot r_m(z)\big)=\big(z, (1-s)\cdot r_n(z)+ s\cdot r_k(z)\big).
	\end{equation*}
First note that by definition \ref{definition: interpolating coordinates} the pairs $r_n(z), r_m(z)$ and $r_n(z), r_k(z)$ bound a connected component of $[0, 1]\setminus \{r_j(z):j\in\N\}$. Therefore $r_m(z)=r_k(z)$ and $t=s$. By Lemma \ref{lemma: reihenfolge bleibt erhalten} it follows that $r_{m-1}(f(z))=r_{k-1}(f(z))$. And therefore \begin{align*}
	&F\big(z, (1-t)\cdot r_n(z)+t\cdot r_m(z)\big)=\big(f(z), (1-t)\cdot r_{n-1}(f(z))+t\cdot r_{m-1}(f(z))\big) \\
	= &\big(f(z), (1-t)\cdot r_{n-1}(f(z))+ t\cdot r_{m-1}(f(z))\big) = F\big(z, (1-t)r_n(z)+ t\cdot r_k(z)\big)
\end{align*}
and therefore $F$ is well-defined.

Let $P(U, n)$ be an $n$-prism. Let $m$ be the consecutive index of $n$ at $z\in U$. By definition for all $w\in U$ a consecutive index of $n$ at $w$ is $m$. And by Lemma \ref{lemma: reihenfolge bleibt erhalten} and Lemma \ref{lemma: if f(z)in U_c^n} we obtain that $m-1$ is a consecutive index of $n-1$ at $f(z)$ for all $f(z)\in f(U)$. Therefore prisms are preserved. 

The map $F$ is obviously continuous within $n$-prisms, it remains to prove continuity where distinct $n$-prisms intersect. Let $V\subset \CDach$ be a connected set such that there are points $z\in V$ such that $m$ is a consecutive index of $n$ at $z$ and $w\in V$ such that $k$ is a consecutive index of $n$ at $w$ with $m\neq k$, i.e., \begin{align*}
	\rho_n(z, t)&=(1-t)\cdot r_n(z)+t\cdot r_k(z) \, \textup{ and }\\
	\rho_n(w, s)&=(1-s)\cdot r_n(w)+s\cdot r_m(w).
\end{align*}
Note that since the maps $z\mapsto r_n(z)$ are continuous for all $n\in\N$ there exists a set $E\subset V$ with $r_m(z)=r_k(z)$ for all $z\in E$. We can subdivide $V$ into $E$, $V_m :=\{z\in V: r_m(z)<r_k(z)\}$ and $V_k:= \{z\in V: r_k(z)<r_m(z)\}$. Then on $V_m$ and $V_k $ the map $F$ is continuous and since on $E$ both maps agree it follows from the Gluing Lemma that it is continuous on $V$.
\end{proof}

\begin{lemma}\label{lemma: extension is quasi-sim on P(U, n)}
	Let $n\geq N_0$ be arbitrary and $z_0\in \CDach\setminus \U_{\crit}^n$. Suppose the map $f$ is univalent on $B(z_0, R)$ and $B(z_0, R) $ as well as $f(B(z_0, R))$ are contained in a hemisphere. Let $U:=B(z_0, r) \subset B(z_0, R)$. Then the map $F$ defined in \eqref{equation: definition of F outside U_c^n} given by 
		\begin{align*}F\colon P(U, n) &\to P(f(U), n-1) \\
			\big(z, \rho_n(z,t)\big)&\mapsto \big(f(z), \rho_{n-1}(f(z), t)\big).
		\end{align*}
		is an $(L, L',  f^{\#}(z_0))$-quasi-similarity where $L=L(r/R, \min_{z, m}\{r_m(z)\})$ and $L'=L'(r/R, \min_{z, m}\{r_m(z)\})$ can be chosen uniform in $n$. 
\end{lemma}
Recall that $N_0\in\N$ is as in Corollary \ref{cor:S1_in_SN0} and is equal to 2 in case the approximating spheres are monotone. 

Similarly we can show that iterates of $F$ -- if they are defined -- are quasi-similarities on orbits that avoid the critical neighborhoods suitably. 

\begin{corollary}\label{corollary: iterate of extension is quasi-sim on P(U, n)}
		Let $n+m\geq N_0$ and let $z_0\in \CDach\setminus \U_{\crit}^{n+m}$. Suppose that $f^m$ is univalent on $B(z_0, R)$ and that $B(z_0, R)$ as well as $f^m(B(z_0, R))$ are contained in a hemisphere. Let $U:=B(z_0, r)$ with $r<R$. Then $F^m$ is given by 
		\begin{equation*}
			F^m\big(z, \rho_{n+m}(z, t)\big)=\big(f^m(z), \rho_{n}(f^m(z), t)\big)
		\end{equation*}
		for $z\in B(z_0, r)$ and is an $(L, L', (f^m)^{\#}(z_0))$-quasi-similarity where $L, L'$ are as in Lemma \ref{lemma: extension is quasi-sim on P(U, n)}.
\end{corollary}

\begin{lemma}\label{lemma: scaling of distance to sp by f'}
	Suppose $z\in \CDach\setminus \U_{\crit}^n$ and let $m\in \{0, \dots, n-1\} $ be so that $f^j(z)\notin \U_{\crit}^{n-m}$ for all $j=0, \dots, m$. Then we have that 
	\begin{equation*}
		\delta_{n-k}(f^k(z), t)
		=
		(f^k)^{\#}(z)\delta_n(z, t) 
	\end{equation*}
		for all $k\leq m$.
\end{lemma}
\begin{proof}
	We will show the claim for $m=1$ and the case where $f^j(z)\notin \U_{\crit}^{n-1}$ for all $j\leq n-1$. The cases for $m>1$ follow easily.
	
	Let us assume that $f^j\notin\U_{\crit}^{n-j}$ for all $j\leq n-1$. Recall the definition of $d_n(z)= \frac{1}{(f^n)^{\#}(z)}$. To prove the claim we compute \begin{equation*}
		d_{n-1}(f(z))
		=
		\frac{1}{(f^{n-1})^{\#}(f(z))}
		=
		\frac{1}{f^{\#}(f^{n-1}(z))\dots f^{\#}(f(z)))}=\frac{f^{\#}(z)}{(f^n)^{\#}(z)}
	\end{equation*}
	and the last term is equal to $f^{\#}(z)d_n(z)$. This suffices to prove the claim.
	
	In case $m=1$ we have that $f(z)\in U_c^{n-1}$ for some critical point $c\in \crit$ and by definition the distances are given by  
	\begin{align*}
		d_{n}(z)&=\frac{1}{\|Df_1\circ\dots\circ f_n(z)\|_{\CDach}}=\frac{1}{\|D(f_1\circ\dots\circ f_{n-1})(f(z))\|_{\CDach} f^{\#}(z)}\\
		d_{n-1}(f(z))&=\frac{1}{\|Df_1\circ\dots\circ f_{n-1}(f(z))\|_{\CDach}}\\
		&=\frac{1}{\|D(f_1\circ\dots\circ f_{n-1})(f(z))\|_{\CDach} }\\
		&=f^{\#}(z)\cdot d_n(z). \qedhere
	\end{align*}
\end{proof}

\begin{proof}[Proof of Lemma \ref{lemma: extension is quasi-sim on P(U, n)}]
	$1.$
	We use that the metric of $\R^3$ is comparable to a spherical sum norm, i.e., \begin{equation*}
		|(z, \rho)-(w, \tau)| \asymp \frac{1}{2}(\rho+\tau)|z-w|_{\hat{\mathbb{C}}} +|\rho-\tau|.\end{equation*}
	where $z, w \in \CDach$ are contained in a hemisphere and $\rho, \tau \in (0, 1)$.
	
	We examine both summands separately for this map. For the part $|f(z)-f(w)|_{\hat{\mathbb{C}}}$ we can use Koebe's distortion theorem to conclude \begin{equation*}
		C_f'f^{\#}(z)|z-w|_{\hat{\mathbb{C}}}\leq |f(z)-f(w)|_{\hat{\mathbb{C}}}\leq C_f f^{\#}(z)|z-w|_{\hat{\mathbb{C}}}
	\end{equation*}
	where $C_f=C_f(r/R)\searrow 1$ and $C_f'=C_f'(r/R)\nearrow 1$ as $r/R\to 0$.
	
	In order to find an upper bound for the difference in the radii $|\rho_{n-1}(f(z), t)-\rho_{n-1}(f(w), s)|$ we compute 
	\begin{align*}
		|&\rho_{n-1}(f(z), t)-\rho_{n-1}(f(w), s)|= |\delta_{n-1}(f(z), t)-\delta_{n-1}(f(w), s)|\\
		&=|f^{\#}(z)\delta_n(z, t)-f^{\#}(w)\delta_n(w, s)|\\
		&\leq |f^{\#}(z)\delta_n(z, t)-f^{\#}(z)\delta_n(w, s)|+|f^{\#}(z)\delta_n(w, s)-f^{\#}(w)\delta_n(w, s)|\\
		&\leq f^{\#}(z)|\delta_n(z, t)-\delta_n(w, s)|+\delta_n(w, s)|f^{\#}(z)-f^{\#}(w)|\\
		&\leq f^{\#}(z)|\rho_n(z, t)-\rho_n(w, s)|+ \delta_n(w, s)f^{\#}(z)C|z-w|_{\hat{\mathbb{C}}}
	\end{align*}
	where $C=C(r/R)$ is a constant such that \begin{equation*}
		|f^{\#}(z)-f^{\#}(w)|\leq C|z-w|_{\CDach}.
	\end{equation*}

	To summaries: we have \begin{align*}
		|F(z, \rho_n(z, t))&-F(w, \rho_n(w, s))|\leq f^{\#}(z)|\rho_n(z, t)-\rho_n(w, s)|\\
		&+ f^{\#}(z)C_f\frac{1}{2}\big(\rho_{n-1}(f(z),t)+\rho_{n-1}(f(w), s)+C\cdot \delta_n(w, s)\big)|z-w|_{\hat{\mathbb{C}}}.
	\end{align*}
	We estimate the factor that gets multiplied with $|z-w|_{\CDach}$ as follows: \begin{align*}
		\frac{1}{2}&C_f\big(\rho_{n-1}(f(z), t)+\rho_{n-1}(w, s)\big)+C\delta_n(w, s)\\
		\leq &CC_f \left(\frac{1}{2}(\rho_{n-1}(f(z), t))+\rho_{n-1}(f(w), s)) +\delta_{n}(w, s) \right)\\
		\leq &2CC_f
	\end{align*} 
	The goal is to find a bound of this expression in terms of \begin{equation*}
		\frac{1}{2}(\rho_n(z, t)+\rho_n(w, s)). 
	\end{equation*}
	Because there exists a minimum of $r_m(v)$ for $n\in\N$ and $v\in\CDach$ we can find a constant $K$ such that for all $u, v\in\CDach$ \begin{equation*}
		K\cdot \frac{1}{2}(r_m(u)+r_m(v))\geq 2.
	\end{equation*}

	Using this we can estimate \begin{align*}
		|F(z, \rho_n(z, t))&-F(w, \rho_n(w, s))|\leq f^{\#}(z)|\rho_n(z, t)-\rho_n(w, s)|\\
		+KCC_f f^{\#}(z)&\frac{1}{2}\big(\rho_n(z, t)+\rho_n(w, s)\big)|z-w|_{\CDach}\\
		\leq f^{\#}(z) KCC_f&|(z, \rho_n(z, t))-(w, \rho_n(w, s))|.
	\end{align*}	

	In order to find a lower bound we use Lemma \ref{lemma: scaling of distance to sp by f'}. \begin{align*}
		|\rho_n(z, t)&-\rho_n(w, s)|= \left|\delta_n(z, t)-\delta_n(w, s)\right|\\
		&=\left|\frac{\delta_{n-1}(f(z), t)}{f^{\#}(z)}-\frac{\delta_{n-1}(f(w), s)}{f^{\#}(w)}\right|\\
		&\leq \left|\frac{f^{\#}(w)\delta_{n-1}(f(z), t)-f^{\#}(w)\delta_{n-1}(f(w), s)}{f^{\#}(z)f^{\#}(w)}\right|\\
		&+\left|\frac{f^{\#}(w)\delta_{n-1}(f(w), s)-f^{\#}(z)\delta_{n-1}(f(w), s)}{f^{\#}(z)f^{\#}(w)}\right|\\
		&\leq \frac{1}{f^{\#}(z)}\left|\delta_{n-1}(f(z), t)-\delta_{n-1}(f(w), s)\right| +\frac{\delta_{n-1}(f(w), s)}{f^{\#}(z)}C|z-w|_{\mathbb{\hat{C}}} \\
		&\leq \frac{1}{f^{\#}(z)}\left|\delta_{n-1}(f(z), t)-\delta_{n-1}(f(w), s)\right|\\
		&+\frac{\delta_{n-1}(f(w), s)}{(f^{\#}(z))^2}CC_f'|f(z)-f(w)|_{\hat{\mathbb{C}}}\\
		&= \frac{1}{f^{\#}(z)}\left|\delta_{n-1}(f(z), t)-\delta_{n-1}(f(w), s)\right|\\
		&+\frac{\delta_n(w, s)}{f^{\#}(z)}CC_f'\left|f(z)-f(w)\right|_{\hat{\mathbb{C}}}
	\end{align*}
	
	Now we will turn to finding a lower bound, and in the following estimate use some inequalities from above
	\begin{align*}
		|(z, &\rho_n(z, t))-(w, \rho_n(w, s))| = |\rho_n(z, t)-\rho_n(w, s)|+ \frac{1}{2}\big(\rho_n(z, t)+\rho_n(w, s)\big)|z-w|_{\CDach}\\
		&\leq \frac{\delta_n(w, s)}{f^{\#}(z)} CC_f'|f(z)-f(w)|_{\hat{\mathbb{C}}} +\frac{1}{f^{\#}(z)}\left|\rho_{n-1}(f(z), t)-\rho_{n-1}(f(w), s)\right|\\
		&+\frac{1}{2}\big(\rho_{n}(z,t)+\rho_{n}(w, s)\big)\frac{1}{f^{\#}(z) C_f'}|f(z)-f(w)|_{\hat{\mathbb{C}}} \\
		&\leq \frac{1}{f^{\#}(z)}\big(CC'f+K\big)\frac{1}{2}\big(\rho_{n-1}(f(z), t)+\rho_{n-1}(f(w), s)\big)|f(z)-f(w)|_{\hat{\mathbb{C}}}\\		
		&+\frac{1}{f^{\#}(z)}|\rho_{n-1}(f(z), t)-\rho_{n-1}(f(w), s)|\\
		&\lesssim L_{F}'\frac{1}{f^{\#}(z)}|F(z, \rho_n(z, t))-F(w, \rho_n(w, s))|
	\end{align*}
	The constant is $L_F'=CC_f'+K$. Note that both constants $L_F$ and $L_F'$ depend only on $\min_{z, n}\{r_n(z)\}$ and $r/R$.
\end{proof}
\begin{proof}[Proof of Corollary \ref{corollary: iterate of extension is quasi-sim on P(U, n)}] 
	We can prove the claim analogously as Lemma \ref{lemma: extension is quasi-sim on P(U, n)} and note that we will obtain similar constants from Theorem \ref{theorem: spherical Koebe} depending on $r/R$ and $\min_{z\in\CDach}\{r_n(z):n\in\N\}$ and the scaling factor of $(f^m)^{\#}(z_0)$.	
\end{proof}

\subsection{Extension near critical axes}
\label{subsection: extension near critical points}

Now the aim is to define the extension around the critical axes. For the following discussion we will fix a critical point $c$ that gets mapped to $p\in \textup{post}(f)$. We assume that $z\in\CDach$ and $n\in\N$ are so that $z\in U_c^{n+1}$. Then we will show that we can assume the consecutive index of $n+1$ at $z$ to be $n+2$. We will prove this in Lemma \ref{lemma:mon_radii_crit}. That implies that the $(n+1)$-prism above $U_c^{n+1}$ is given by 
\begin{equation*}
	P(U_c^{n+1}, n+1)
	=
	\{(z, r_{n+1}(z)\cdot(1-t)+r_{n+2}(z)\cdot t): z\in U_c^{n+1}, t\in [0, 1] \}.
\end{equation*}

Recall that the aim was to extend the map 
\begin{equation*}
	F\colon S_{n+1}\to S_{n}, \, (z, r_{n+1}(z))\mapsto (f_{n+1}(z), r_{n}(f_{n+1}(z))).
\end{equation*}
If we take a point $z\in U_c^{n+2}\cap U_c^{n+1}$ then we can not
guarantee that $f_{n+1}(z)$ is equal to $f_{n+2}(z)$, so we will not
use the approach of mapping line segments of the form $l_z=\{(z, r): r
\geq r_{N_0}(z)\}$ to a line segment of $l_{f_{n+1}(z)}=\{(f_{n+1}(z), r):
r\geq r_{N_0-1}(f_{n+1}(z))\}$ as in Section \ref{subsection:
	extension far from critical points} naively. If we wanted to
use this approach one has to be careful about the boundary
values. Instead the strategy will be as follows:
\begin{itemize}
	\item We will show that there exists a bi-Lipschitz
	parametrization of the $(n+1)$-prism above $U_c^{n+1}$ from a cylinder
	$|\lambda|^{-n/d}\Db\times [0, q_{n+1}]$. The map will be
	denoted by
	$\alpha_1\colon |\lambda|^{-n/d}\Db\times [0, q_{n+1}] \to
	P(U_c^{n+1}, n+1)$, where $q_{n+1}\in \R$ is given by
	\begin{equation*}
		q_{n+1} 
		\coloneqq
		r_{n+2}(c)-r_{n+1}(c)\asymp d_{n+1}(c)
	\end{equation*}
	for the critical point $c\in U_c^{n+1}$. 
	\item Similarly there exists a bi-Lipschitz parametrization $\alpha_2\colon|\lambda|^{-n}\Db\times [0, q_{n}'] \to P(U_p^{n}, n)$. Here $q_{n}' :=  r_{n+1}(p)-r_{n}(p)\asymp d_{n}(p)$ and $p\in U_p^n$ is the post-critical point. 
	\item In a third step we will show that there exists a map $\beta\colon |\lambda|^{-n/d}\Db\times[0, q_{n+1}]\to |\lambda|^{-n}\Db\times[0, q_{n}']$ that is bi-Lipschitz when it is suitably restricted to sectors $A(\theta_0, \pi/d)\cap |\lambda|^{-n/d}\Db$ (compare \eqref{eq:def_sector}). For points $(z, 0)$ the map $\beta$ is given by 
	\begin{equation*}
		\beta(z, 0)
		=
		(h_{n+1}(z), 0)
	\end{equation*} 
	and for points $(z, q_{n+1})$ the map is given by 
	\begin{equation*}
		\beta(z, q_{n+1})
		=
		(h_{n+1}(z), q'_{n}).
	\end{equation*}
	\item The map $F$ will be given by 
	\begin{equation}\label{equation: def of F on U_c^n}
		\begin{split}
			F\colon P(U_c^{n+1}, n+1)&\to P(U_p^{n}, n)\\ 
			(z, \rho_{n+1}(z, t))&\mapsto \alpha_2\circ \beta\circ\alpha_1^{-1}(z, \rho_{n+1}(z, t)),
		\end{split}
	\end{equation}
	see Figure \ref{fig:F_at_crit} for an illustration.
	\item For the parametrizations $\alpha_1, \alpha_2$ we will use the coordinates $\phi$ and $\psi$. Because the restrictions of $\beta$ satisfy the above equations we can guarantee that the expanding map $F$ agrees with the map $F|_{S_{n+1}}$ defined by $(z, r_{n+1}(z))\mapsto(f_{n+1}(z), r_{n}(f_{n+1}(z)))$ on the approximating sphere $S_{n+1}$.
	The maps $\alpha_1, \alpha_2, \beta$ also depend on $n+1$ but in order to lighten the notation we will not indicate this.	
\end{itemize}

\begin{figure}[htb]
\begin{overpic}[scale=0.6]{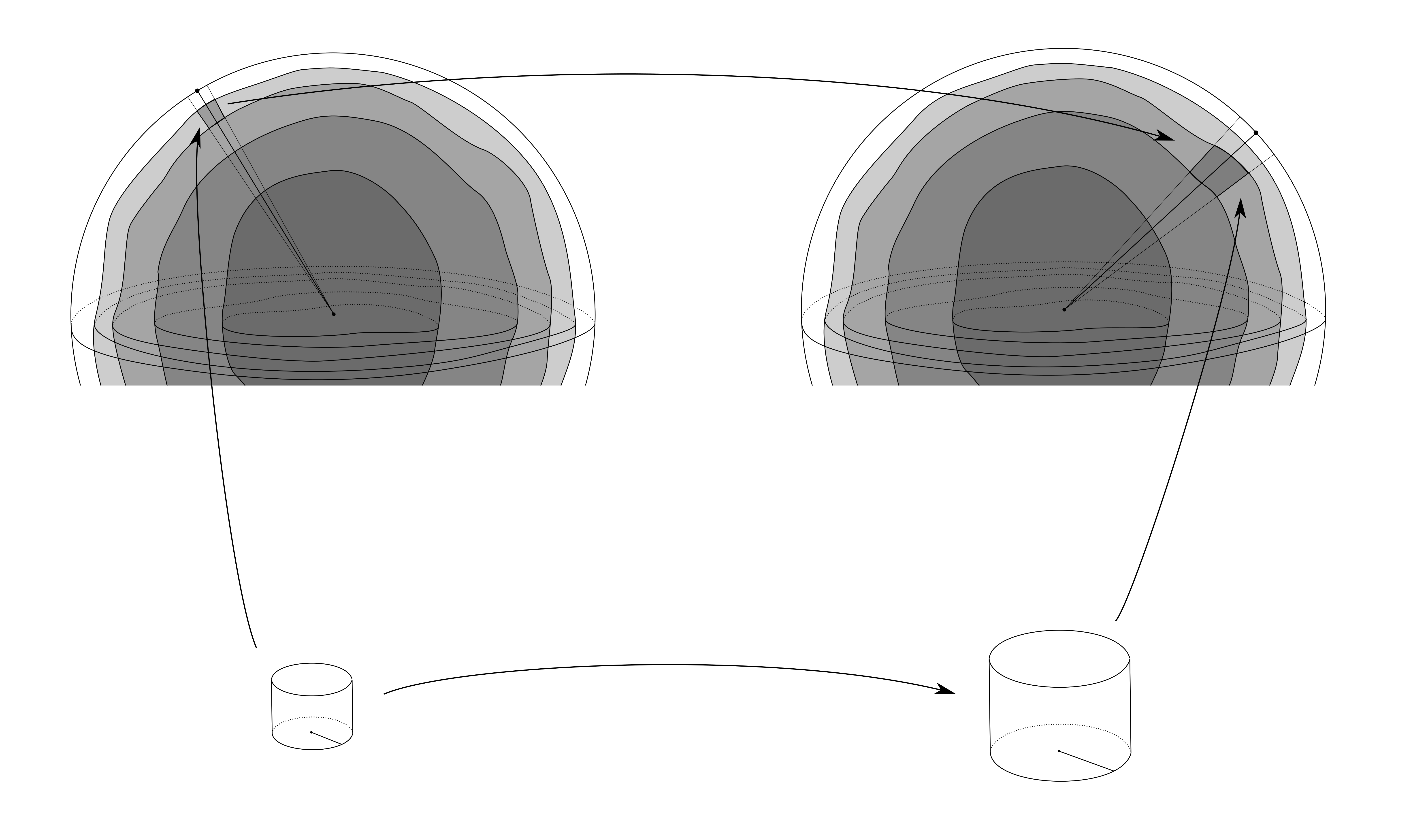}
	\put(13, 53){\footnotesize $c$}
	\put(89, 50){\footnotesize $p$}
	\put(48, 54){\footnotesize $F$}
	\put(12, 23){\footnotesize $\alpha_1$}
	\put(82, 23){\footnotesize $\alpha_2$}
	\put(48, 13){\footnotesize $\beta$}
	\put(24.5, 4){\footnotesize $|\lambda|^{-(n-1)/d}$}
	\put(78.5, 3.5){\footnotesize $|\lambda|^{-(n-1)}$}
\end{overpic}
\caption{We illustrate the extension around critical points in defined by the maps $\alpha_1$, $\alpha_2$ and $\beta$.}
\label{fig:F_at_crit}
\end{figure}
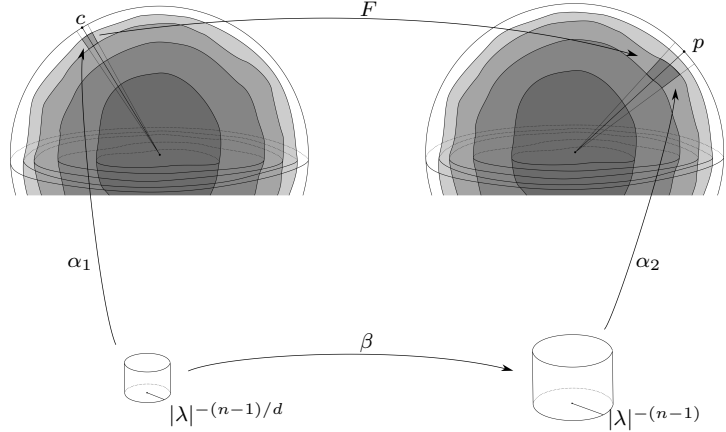

By monotonicity of the radii we have that the $n$-prism over the set $U_c^{n+1}$ is given by 
\begin{equation*}
	P(U^{n+1}_c, n+1)
	=
	\left\{(z, r): z\in U^{n+1}_c,\, r\in [r_{n+1}(z), r_{n+2}(z)]\right\}.
\end{equation*}
Around post-critical points $p$ we have 
\begin{equation*}
	P(U_p^{n}, n)=\{(z, r): z\in U_p^{n}, \, r\in [r_{n}(z), r_{n+1}(z)]\}.
\end{equation*}

\begin{lemma}\label{lemma:mon_radii_crit}
	The neighborhoods in $\U_{\post}^0$ can be chosen so that 
	\begin{enumerate}
		\item 
		\label{item:mon_radii_crit_1}
		for all $n\in \N$ and all $z\in \U_{\crit}^n$ we have that $r_{n+1}(z)>r_n(z)$
		\item 
		\label{item:mon_radii_crit_2}
		for all $n\in \N$ and all $z\in\U_{\post}^n$ we have that $r_{n+1}(z)>r_n(z)$
		\item 
		\label{item:mon_radii_crit_3}
		there exist constants $a, b$ independent of $n$ such that 
		\begin{equation*}
			a\cdot d_n(z)
			\leq 
			r_{n+1}(z)-r_n(z)
			\leq 
			b \cdot d_n(z),
		\end{equation*}
		for all $z\in \U_{\crit}^n$.
	\end{enumerate}
\end{lemma}
\begin{proof}
	\ref{item:mon_radii_crit_1}
	By Lemma \ref{lem:Koebe_phi_psi} there exist constants $K_1(r)<1, K_2(r)>1$ such that for all $x, y\in B(0, r)\subset \D$ we have that 
	\begin{equation*}
		K_1(r)(\psi^{-1})^{\#}(x)
		\leq 
		(\psi^{-1})^{\#}(y)
		\leq 
		K_2(r)(\psi^{-1})^{\#}(x).
	\end{equation*}
	Choose $r$ sufficiently small so that 
	\begin{equation}
		\label{eq:mon_radii_crit_2}
		K_1(r)\cdot |\lambda|^{1/d}\geq 1+\mu \quad \textup{and} \quad |\lambda|^{1/d}> K_2
	\end{equation} 
	for a small number $\mu>0$. From now on we fix the values $r$, $K_1:= K_1(r)$ and $K_2:=K_2(r)$.
	
	With these choices it follows that for all $x, y\in B(0, r)$  
	\begin{equation*}
		\label{eq:mon_radii_crit_1}
		|\lambda|^{1/d}(\psi^{-1})^{\#}(x)-(\psi^{-1})^{\#}(y) 
		\geq 
		(\psi^{-1})^{\#}(y)\big(|\lambda|^{1/d}K_1-1\big)
		> 
		0.
	\end{equation*}
	
	\begin{claim}
		If $U_p^0= \psi^{-1}(B(0, r))$ with $r$ as above then for all $n\in \N$ and all $z\in \U_{\crit}^n$ the claimed inequality holds.
	\end{claim}

	Let $n\in \N$ be arbitrary and suppose $z\in U_c^{n+1}$ for some $c\in \crit$. The difference $r_{n+1}(z)-r_n(z)$ depends on the difference between the derivatives of $f_1\circ\dots\circ f_{n+1}$ and $f_1\circ\dots\circ f_{n}$ at $z$. The point $z$ is mapped by $\phi$ to a point $w:=\phi(z) \in |\lambda|^{-(n+1)/d}\Db$. The map $f_1\circ\dots\circ f_{n+1}$ can be represented as follows  
	
	\begin{equation*}
		\xymatrix{
			U_c^{n+1} \ar[r]^{f}_{f_{n+1}} \ar[d]_{\phi}
			& U_p^{n} \ar[d]^{\psi} \ar[r]^{f^{n}} & U_p \ar[d]^{\psi} \\
			|\lambda|^{-n/d}\Db \ar[r]^{z^d}_{h_{n+1}} 
			& |\lambda|^{-n}\Db \ar[r]^{\delta_{\lambda^{n}}}& \Db 
		}
	\end{equation*}

	and for the map $f_1\circ\dots\circ f_n$ we can easily adjust it by substituting $n+1$ by $n$ and $n$ by $n-1$.
	
	For the map $f_1\circ\dots\circ f_{n+1}$ the point $w$ is mapped to $h_{n+1}(w)$ and $|h_{n+1}(w)| \leq |\lambda|^{-n}$. Composing with $\psi^{-1}$ we see that $\psi^{-1}(h_{n+1}(w))\in U_p^{n}$ where the linearization is defined. 
	
	Denote $w_{n+1}:= \delta_{\lambda^{n}}\circ h_{n+1}(w)\in \Db$, by definition $(f^{n}\circ f_{n+1})(z)=\psi^{-1}(w_{n+1})$ and similarly $w_{n}= \delta_{\lambda^{n-1}}\circ h_n(w)\in\Db$. To prove monotonicity we compute the difference between the derivatives using the formula we proved in Lemma \ref{lem:Df1n}
	\begin{align*}
		\|D(f^{n}&\circ f_{n+1})(z)\|_{\CDach}-\|D(f^{n-1}\circ f_{n})(z)\|_{\CDach}
		= (\psi^{-1})^{\#}(w_{n+1}) d |\lambda|^{n/d} \phi^{\#}(z)\\ 
		&- (\psi^{-1})^{\#}(w_{n}) d|\lambda|^{(n-1)/d}\phi^{\#}(z)\\
		&= d|\lambda|^{(n-1)/d}\phi^{\#}(z)(|\lambda|^{1/d}\big(\psi^{-1})^{\#}(w_{n+1})-(\psi^{-1})^{\#}(w_{n})\big).
	\end{align*}
	We have shown the claim in the first case.
	
	We also need to treat the case when $z\in U^c_{n}\setminus \inte ( U^c_{n+1})$. 
	
	The map $f_1\circ\dots\circ f_{n}(z)=f^{n-1}\circ f_{n}(z)$ as above. But since $z\notin U_c^{n+1}$ we have that $f_{n+1}(z)=f(z)$. Because the point $z$ is still close to a critical point we can use the conjugacy to $z^d$ and the linearization:
	
	\begin{tikzcd}
		U_c^{n}\setminus \inte (U_c^{n+1}) \arrow[r, "f"] \arrow[d, "\phi"]
		& U_p^{n-1}\setminus \inte (U_p^{n}) \arrow[d, "\psi"] \arrow[r, "f^{n}"] & U_p \arrow[d, "\psi"] \\
		\overline{A}(|\lambda|^{-(n-1)/d},|\lambda|^{-n/d}) \arrow[r, "z^d"] 
		&\overline{A}(|\lambda|^{-(n-1)}, |\lambda|^{-n})  \arrow[r, "\lambda^{n}"]& \D
	\end{tikzcd}
	
	Let $\tilde{w}_{n+1}:= \delta_{\lambda^{n}}(w^d) \in \Db$. A lower bound for $(f^{n+1})^{\#}(z)$ can be computed as follows: \begin{align*}
		(f^{n+1})^{\#}(z)
		\geq 
		(\psi^{-1})^{\#}(\Tilde{w}_{n+1}) |\lambda|^n d |\lambda|^{-(n/d) \cdot (d-1)}\phi^{\#}(z)\\ 
		=
		(\psi^{-1})^{\#}(\Tilde{w}_{n+1}) d |\lambda|^{n/d}\phi^{\#}(z). 
	\end{align*} 
	And hence for the lower bound of the difference between $( f^{n+1})^{\#}(z) $ and $\|D(f^{n-1}\circ f_{n})(z)\|_{\CDach}$:
	\begin{align*}
		(f^{n+1})^{\#}(z)&-\|Df^{n-1}\circ f_{n}(z)\|_{\CDach} 
		\geq (\psi^{-1})^{\#}(\Tilde{w}_{n+1})d|\lambda|^{n/d}\phi^{\#}(z)\\
		&-(\psi^{-1})^{\#}(w_{n})d|\lambda|^{(n-1)/d}\phi^{\#}(z)\\
		&=d|\lambda|^{(n-1)/d}\phi^{\#}(z)\left(|\lambda|^{1/d}(\psi^{-1})^{\#}(\Tilde{w}_{n+1})-(\psi^{-1})^{\#}(w_{n})\right)
	\end{align*}
	
	This proves the claim in the second case.
	
	\smallskip
	\ref{item:mon_radii_crit_2}
	Since $z\in U_p^n$ we have that $f_1\circ\dots\circ f_n(z)=f^n(z)$. The inequality claimed above is equivalent to: \begin{equation*}
		(f^{n+1})^{\#}(z)> (f^n)^{\#}(z).
	\end{equation*}
	Then we can reformulate this to \begin{equation*}
		(f^n)^{\#}(z)\big(f^{\#}(f^n(z))-1\big)>0.
	\end{equation*}
	Then choose $U_p^0\subset \psi^{-1}(B(0, r))$ small enough to guarantee the above inequality.
	
	\smallskip
	\ref{item:mon_radii_crit_3}
	We assume $r$ to be chosen as above in \eqref{eq:mon_radii_crit_2} and $U_p^0$ accordingly. 
	
	Let $z\in U_c^{n+1}$ be arbitrary and denote $w:= \phi(z)$ and $w_{n+1}:= \delta_{\lambda^{n}}\circ h_{n+1}(w)\in \Db$ and similarly $w_{n}= \delta_{\lambda^{n-1}}\circ h_n(w)\in\Db$. Using the estimates in part \ref{item:mon_radii_crit_1} of the proof we compute 
	\begin{equation*}
		\begin{split}
			&\left(1-s\frac{1}{\|Df^{n}\circ f_{n+1}(z)\|_{\CDach}}- 1+s\frac{1}{\|Df^{n-1}\circ f_{n}(z)\|_{\CDach}}\right)\\
			=&
			s\frac{\|Df^{n}\circ f_{n+1}(z)\|_{\CDach} -\|Df^{n-1}\circ f_{n}(z)\|}{\|Df^{n}\circ f_{n+1}(z)\|_{\CDach} \cdot \|Df^{n-1}\circ f_{n}(z)\|_{\CDach}}\\
			=& s\frac{d|\lambda|^{(n-1)/d}\phi^{\#}(z)\cdot\big(|\lambda|^{1/d}(\psi^{-1})^{\#}(w_{n+1})-(\psi^{-1})^{\#}(w_{n})\big)}{\|Df^{n}\circ f_{n+1}(z)\|_{\CDach} \cdot \|Df^{n-1}\circ f_{n}(z)\|_{\CDach}}.
		\end{split}
	\end{equation*}
	The expression 
	\begin{equation*}
		|\lambda|^{1/d}(\psi^{-1})^{\#}(w_{n+1})-(\psi^{-1})^{\#}(w_{n})
	\end{equation*}
	can be estimated by 
	\begin{equation*}
		\begin{split}
			(\psi^{-1})^{\#}(w_{n+1})(|\lambda|^{1/d}-K_2)
			&\leq 
			|\lambda|^{1/d}(\psi^{-1})^{\#}(w_{n+1})-(\psi^{-1})^{\#}(w_{n})\\
			&\leq 
			(\psi^{-1})^{\#}(w_{n+1})(|\lambda|^{1/d}-K_1).
		\end{split}
	\end{equation*}
	So we can estimate the difference of the radii 
	\begin{equation*}
		\begin{split}
			|\lambda|^{-1/d}(|\lambda|^{1/d}-K_2)\cdot d_n(z)
			&\geq			
			r_{n+1}(z)-r_n(z)\\
			&\geq 
			|\lambda|^{-1/d}(|\lambda|^{1/d}-K_1)\cdot d_{n}(z)
		\end{split}
	\end{equation*}
	Assume that $z\in U_c^{n}\setminus U_c^{n+1}$. Let $w:=\phi(z)$ and $w_{n}= \delta_{\lambda^{n-1}}\circ h_n(w)\in\Db$ and $\tilde{w}_{n+1}:= \delta_{\lambda^{n}}(w^d) \in \Db$. Using the estimates in the proof of part \ref{item:mon_radii_crit_1}, we obtain 
	\begin{equation*}
		\begin{split}
			\bigg( 1&- s\frac{1}{(f^{n+1})^{\#}(z)}-1+s\frac{1}{\|Df^{n-1}\circ f_{n}(z)\|_{\CDach}} \bigg) \\
			&=
			s\frac{(\psi^{-1})^{\#}(\tilde{w}_{n+1})d|\lambda|^n|w|^d\phi^{\#}(z)- (\psi^{-1})^{\#}(w_n)d|\lambda|^{(n-1)/d}\phi^{\#}(z)}{(f^{n+1})^{\#}(z)\cdot \|Df^{n-1}\circ f_{n}(z)\|_{\CDach}}\\
			&\geq
			s\frac{d|\lambda|^{(n-1)/d}\phi^{\#}(z)\big(|\lambda|^{1/d}(\psi^{-1})^{\#}(\Tilde{w}_{n+1})-(\psi^{-1})^{\#}(w_{n})\big)}{(f^{n+1})^{\#}(z)\cdot \|Df^{n-1}\circ f_{n}(z)\|_{\CDach}}\\
			&\geq 
			C (|\lambda|^{1/d}-K_2) \cdot d_n(z)
		\end{split}
	\end{equation*} 
	where $C$ is a constant such that $C|\lambda|^{(n-1)/d}\geq |\lambda|^{n}|w|^{d-1}$, and the upper bound is similar to the above. 
\end{proof}

With this we can show that on pre-images of $P(U_c^n, n)$ a consecutive index of $n+1$ at $z\in f^{-1}(U_c^n)$ is $n+2$. In particular this implies that the map $F$ defined away from the critical axes will be orientation-preserving and continuous. 
\begin{corollary}
	If $z\in \CDach\setminus \U_{\crit}^{n+j+1}$ and $f^j(z)\in U_c^n$ for some critical point $c$ and $j\in \N$ then $r_{n+j+1}(z)> r_{n+j}(z)$. 
\end{corollary}
\begin{proof}
	We will prove the statement for $j=1$, the case $j>1$ is analogous. The inequality $r_{n+2}(z)>r_{n+1}(z)$ is equivalent to \begin{equation*}
		\|Df_1\circ\dots\circ f_{n+1}(f(z))\|_{\CDach} f^{\#}(z)> \|Df_1\circ\dots\circ f_n(f(z))\|_{\CDach} f^{\#}(z).
	\end{equation*}
	And above we have proved that $\|Df_1\circ\dots\circ f_{n+1}(f(z))\|_{\CDach} > \| Df_1\circ\dots\circ f_n(f(z))\|_{\CDach}$. 
\end{proof}

\begin{remark}
	In the following we will give a parametrization of prisms over $U_c$. We will change the coordinate $\rho_n(z)$ and adapt it to an interval $[0, h]$ where $h$ is not necessarily 1 by setting 
	\begin{equation}
		\label{eq:conv_rho_n_height_h}
		\rho_n(z)=\frac{q-t}{q}r_n(z)+\frac{t}{q}r_{n+1}(z),
	\end{equation}
	the value of $h$ will be clear from the context, so we will not introduce new notation for this expression. 
\end{remark}

\begin{lemma}\label{lemma:param_of_P(U_c)} 
	For all $n\in \N$ and all critical points $c\in \CDach$ we find a parametrization of $P(U_c^n, n)$. It is given by 
	\begin{align*}
		\alpha_1\colon |\lambda|^{-(n-1)/d}\Db\times[0, q_n]&\to P(U_c^{n}, n)\\
		(x, t)&\mapsto (\phi^{-1}(x), \rho_n(\phi^{-1}(x), t))
	\end{align*} 
	with $q_n=r_{n+1}(c)-r_n(c)$ and is a bi-Lipschitz map and the constant $L_{\alpha_1}=L_{\alpha_1}(|\lambda|^{-(n-1)/d}, |\lambda|^{-(n-1)}, d)$ is decreasing as $n$ grows. Here $\rho_n(z)$ is as in Equation \eqref{eq:conv_rho_n_height_h}.
\end{lemma}

\begin{lemma}
	\label{lemma:radius_Lipschitz_Ucrit}
	Let $z\in \U_{\crit}^n$. Then for $z\mapsto r_n(z)$ we can show that 
	\begin{equation*}
		|r_n(z)-r_n(w)|\leq L_{r_n}\cdot d_n(z)|z-w|_{\hat{\mathbb{C}}},
	\end{equation*} 
	for all $w\in U_c^n$ where $L_{r_n}=L_{r_n}(|\lambda|^{-(n-1)/d}, |\lambda|^{-(n-1)}, d, \phi)$ is decreasing as $n$ increases and can therefore be chosen uniformly.
\end{lemma}
\begin{proof}
	Let $z, w\in U_c^n$ be arbitrary. Recall that $f_n$ was defined as \[f_n(z)= (\psi^{-1}\circ h_n\circ \phi) (z)\] in $U_c^n$. The difference $|r_n(w)-r_n(z)|$ can be estimated by  
	\begin{align*}
		\bigg|\frac{1}{\|D(f^{n-1}\circ f_n)(z)\|_{\CDach}}&-\frac{1}{\|D(f^{n-1}\circ f_n)(w)\|_{\CDach}} \bigg|\\
		&\leq 
		\left|\frac{(f^{n-1})^{\#}(f_n(w))\cdot \big(\|Df_n(w)\|_{\CDach} -\|Df_n(z)\|_{\CDach}\big)}{\|D(f^{n-1}\circ f_n)(z)\|_{\CDach}\cdot \|D(f^{n-1}\circ f_n)(w)\|_{\CDach}} \right|\\ &+\left|\frac{\|Df_n(z)\|_{\CDach}\cdot \big((f^{n-1})^{\#}(f_n(w)) -(f^{n-1})^{\#}(f_n(z))\big)}{\|D(f^{n-1}\circ f_n)(z)\|_{\CDach}\cdot \|D(f^{n-1}\circ f_n)(w)\|_{\CDach}} \right|
	\end{align*}
	We will treat the two summands separately. In the following we will establish an estimate 
	\begin{equation*}
		\|D(f^{n-1}\circ f_n)(z)\|_{\CDach}-\|D(f^{n-1}\circ f_n)(w)\|_{\CDach}
		\leq 
		C\|D(f^{n-1}\circ f_n)(w)\|_{\CDach}\cdot |z-w|_{\hat{\mathbb{C}}}\end{equation*} 
	for a constant $C= C(|\lambda|^{-(n-1)/d}, |\lambda|^{-(n-1)}, \phi, d)$ that is decreasing as $n$ increases and can hence be chosen uniformly. We will use the following notation 
	\begin{equation*}x:=\phi(z)\in |\lambda|^{-(n-1)/d}\Db \textup{ and } y:= \phi(w)\in|\lambda|^{-(n-1)/d}\Db.\end{equation*} 
	In order to find an estimate for the first summand we use the chain rule to then make estimates for the derivatives in the product 
	\begin{align*}
		&(f^{n-1})^{\#}(f_n(w))\cdot \left( \|Df_n(w)\|_{\CDach}-\|Df_n(z)\|_{\CDach}\right)\\
		=&(f^{n-1})^{\#}(f_n(w))\cdot \big( (\psi^{-1})^{\#}(h_n(y))\cdot\|Dh_n(y)\|\cdot\phi^{\#}(w)\\
		& -(\psi^{-1})^{\#}(h_n(x))\cdot \|Dh_n(x)\|\cdot \phi^{\#}(z) \big)
	\end{align*}
	Then we use that $\|Dh_n(u)\|=|\lambda|^{-(n-1)(1-1/d)}d$ for all $u\in |\lambda|^{-(n-1)/d}\Db$, and we will denote it by $\|Dh_n\|$. In order to find an estimate it is convenient to separate the expression further: \begin{equation}\label{eq:radius_Lip_Uc1}
		\begin{split}
			&|(f^{n-1})^{\#}(f_n(w))\cdot ( \|Df_n(w)\|_{\CDach}-\|Df_n(z)\|_{\CDach})|\\
			&\leq 
			(f^{n-1})^{\#}(f_n(w))\cdot \|Dh_n\| \cdot \bigg( (\psi^{-1})^{\#}(h_n(x))\cdot  \big|\phi^{\#}(z)-\phi^{\#}(w)\big|\\
			&+\phi^{\#}(w)\cdot  \big|(\psi^{-1})^{\#}(h_n(y))-(\psi^{-1})^{\#}(h_n(x))\big|\bigg)
		\end{split}
	\end{equation}
	For the first summand using Lemma~\ref{lem:Koebe_phi_psi}~\ref{item:Koebe_phi} we find that there exists a constant $C_{\phi}$ such that
	\begin{equation*}
		| \phi^{\#}(z)-\phi^{\#}(w)| 
		\lesssim 
		\phi^{\#}(w) C_\phi |z-w|_{\hat{\mathbb{C}}}
	\end{equation*}
	and for the second summand by Lemma~\ref{lem:Koebe_phi_psi}~\ref{item:Koebe_psi} there exists a constant $C_{\psi^{-1}}$ such that 
	\begin{align*}
		| (\psi^{-1})^{\#}(h_n(x))&-(\psi^{-1})^{\#}(h_n(y)) |\\
		&\leq 
		(\psi^{-1})^{\#}(h_n(x))\cdot C_{\psi^{-1}}|h_n(y)-h_n(x)|.
	\end{align*}
	For $h_n$ we have by Lemma \ref{lemma:hnc is continuous} that there exists a constant $L_{h_n}$ such that 
	\begin{equation*}
		|h_n(x)-h_n(y)|
		\leq 
		L_{h_n}|x-y|.
	\end{equation*}
	And lastly we estimate $|x-y|\leq L_{\phi}\cdot \phi^{\#}(z)|z-w|_{\CDach}$. Putting everything together we obtain the following estimate
	\begin{equation*}
		|(\psi^{-1})^{\#}(h_n(x))-(\psi^{-1})^{\#}(h_n(y))|
		\leq (\psi^{-1})^{\#}(h_n(x))C_{\psi^{-1}}L_{h_n}\phi^{\#}(w)L_{\phi}|z-w|_{\hat{\mathbb{C}}}.
	\end{equation*}
	
	And inserting these estimates in \eqref{eq:radius_Lip_Uc1} we obtain 
	\begin{align*}
		&(f^{n-1})^{\#}(f_n(w))\cdot \left|\|Df_n(w)\|_{\CDach}-\|Df_n(z)\|_{\CDach}\right|\\
		&\leq (f^{n-1})^{\#}(f_n(w))\cdot
		\|Dh_n\|
		((\psi^{-1})^{\#}(h_n(x))\cdot \left|\phi^{\#}(z)-\phi^{\#}(w)\right|\\
		&+(f^{n-1})^{\#}(f_n(w))
		\|Dh_n\|
		\phi^{\#}(w)\cdot \left|(\psi^{-1})^{\#}(h_n(x))-(\psi^{-1})^{\#}(h_n(y))\right|\\	
		&\leq(f^{n-1})^{\#}(f_n(w))
		\|Dh_n\|
		(\psi^{-1})^{\#}(h_n(x))\cdot\phi^{\#}(w)C_{\phi}|z-w|_{\hat{\mathbb{C}}}\\
		&+(f^{n-1})^{\#}(f_n(w))
		\|Dh_n\|
		\phi^{\#}(w)(\psi^{-1})^{\#}(h_n(y))\\
		&\cdot C_{\psi^{-1}}L_{h_n}\phi^{\#}(w)C_{\phi}|z-w|_{\hat{\mathbb{C}}}
	\end{align*}	
	By Lemma~\ref{lem:Koebe_phi_psi}~\ref{item:Koebe_psi} there exists a constant $C_{\psi^{-1}}'$ such that 
	\[(\psi^{-1})^{\#}(h_n(x))\leq C_{\psi^{-1}}'(\psi^{-1})^{\#}(h_n(y)).\]
	So we may summarize $(\psi^{-1})^{\#}(h_n(x))\|Dh_n\|\phi^{\#}(w)\leq C_{\psi^{-1}}' \|Df_n(w)\|_{\CDach}$ and the estimate has the form 
	\begin{align*}
		&(f^{n-1})^{\#}(f_n(w))\cdot \left|\|Df_n(w)\|_{\CDach}-\|Df_n(z)\|_{\CDach}\right|\\
		\leq&\|Df^{n-1}\circ f_n(w)\|_{\CDach}C_{\psi^{-1}}'C_{\phi}|z-w|_{\hat{\mathbb{C}}}\\
		+&\|Df^{n-1}\circ f_n(w)\|_{\CDach}C_{\psi^{-1}}L_{h_n}\phi^{\#}(w)C_{\phi}|z-w|_{\hat{\mathbb{C}}}.
	\end{align*}
	In the last expression we can estimate $\phi^{\#}(w) $ by the maximum of $\phi^{\#}(u)$ for $u\in \U_{\crit}$. 
	
	We will turn to bounding the second summand of \eqref{eq:radius_Lip_Uc1}. By \ref{theorem: spherical Koebe} there exists a constant $C_f$ such that 
	\begin{equation*}
		(f^{n-1})^{\#}(f_n(w))-(f^{n-1})^{\#}(f_n(z))
		\leq 
		C_{f^{n-1}}(f^{n-1})^{\#}(f_n(w))|f_n(z)-f_n(w)|_{\CDach},
	\end{equation*}
	with $C_{f^{n-1}}$ depending on $|\lambda|^{-(n-1)}$ and decreasing as $n$ increases.
	
	With this we can bound the second summand as follows
	\begin{align*}
		&\|Df_n(w)\|_{\CDach}\big((f^{n-1})^{\#}(f_n(w))-(f^{n-1})^{\#}(f_n(z))\big)\\
		\leq 
		&\|Df_n(w)\|_{\CDach}\cdot (f^{n-1})^{\#}(f_n(w))C_{f^{n-1}}|f_n(z)-f_n(w)|_{\hat{\mathbb{C}}}\\
		=&
		\|Df^{n-1}\circ f_n(w)\|_{\CDach}C_{f^{n-1}}|f_n(w)-f_n(z)|_{\hat{\mathbb{C}}}\\
		\leq 
		&\|Df^{n-1}\circ f_n(w)\|_{\CDach} C_{f^{n-1}}L_{f_n}|w-z|_{\hat{\mathbb{C}}}
	\end{align*}
	
	The constant $L_{f_n}$ is the Lipschitz constant of $f_n$ restricted to $U_c^n$.
	Now if we use both estimates we get \begin{align*}
		|r_n(z)-r_n(w)|&\leq \frac{\|Df^{n-1}\circ f_n(w)\|_{\CDach}}{\|Df^{n-1}\circ f_n(z)\|_{\CDach}\cdot \|Df^{n-1}\circ f_n(w)\|_{\CDach}}\\
		&\cdot\big(C_{\psi^{-1}}'C_{\phi}
		+C_{\psi^{-1}}L_{h_n}C_{\phi}\max_{u\in U_c}\phi^{\#}(x)
		+C_{f^{n-1}}L_{f_n}\big)|w-z|_{\hat{\mathbb{C}}}\\
		&=:L_{r_n}d_n(z)|z-w|_{\hat{\mathbb{C}}}
	\end{align*}
	The constants $C_{\phi}$, $C_{\psi^{-1}}'$, $C_{\psi^{-1}}$ and $C_{f^{n-1}}$ are from Lemma \ref{lem:Koebe_phi_psi} and \ref{lem:prop_fn}, they depend on $|\lambda|^{-(n-1)}$, respectively $|\lambda|^{/(n-1)/d}$. The constant $L_{h_n}=L(d, |\lambda|, n)$ depends on the degree, $|\lambda|$ and $n$.
	
	So to summarize we can write 
	\begin{equation*}
		|r_n(z)-r_n(w)|
		\lesssim 
		d_n(z)|z-w|_{\hat{\mathbb{C}}}
	\end{equation*} 
	with $C(\lesssim)=C(|\lambda|^{-n}, d, \phi)$. 
\end{proof}

Now we are in shape to prove Lemma \ref{lemma:param_of_P(U_c)}

\begin{proof}[Proof of Lemma \ref{lemma:param_of_P(U_c)}]
	Let $(x, t), (y, s)\in |\lambda|^{-(n-1)/d}\Db\times [0, q_n]$ be arbitrary. We will denote \[z:= \phi^{-1}(x)\textup{ and }w:= \phi^{-1}(y).\] Then we will use 
	\begin{align*}|(z, \rho_n(z, t))&-(w, \rho_n(w, s))\asymp \frac{1}{2}(\rho_n(z, t)+\rho_n(w, s))|z-w|_{\hat{\mathbb{C}}}\\ &+|\rho_n(z, t)-\rho_n(w, s)|.\end{align*} 
	Treat both coordinates separately, for the spherical component we use Lemma \ref{lem:Koebe_phi_psi} to conclude that
	\begin{equation}\label{equation: estimate lem 4.19} |z-w|_{\hat{\mathbb{C}}}\asymp(\phi^{-1})^{\#}(x)\cdot |x-y|.\end{equation}
	For the difference of the radii we estimate 
	\begin{align*}
		&|\rho_n(z, t)-\rho_n(w, s)|=|\delta(z, t)-\delta(w, s)|\\
		\leq &\left|d_n(z)\frac{q_n-t}{q_n}-d_n(w)\frac{q_n-s}{q_n}\right|+\left|d_{n+1}(z)\frac{t}{q_n}-d_{n+1}(w)\frac{s}{q_n}\right|\\
		\leq &\left|d_{n}(z)\frac{q_n-t}{q_n}-d_{n}(z)\frac{q_n-s}{q_n}\right|\\+ 
		&\left|d_{n}(z)\frac{q_n-s}{q_n}-d_{n}(w)\frac{q_n-s}{q_n}\right|\\
		+&\left|d_{n+1}(z)\frac{t}{q_n}-d_{n+1}(z)\frac{s}{q_n}\right|+ \left|d_{n+1}(z)\frac{s}{q_n}-d_{n+1}(w)\frac{s}{q_n}\right|\\
		\leq &d_{n}(z)\left|\frac{t-s}{q_n}\right| + \frac{q_n-s}{q_n}|r_n(z)-r_n(w)|\\
		+&d_{n+1}(z)\left|\frac{t-s}{q_n}\right|+\frac{s}{q_n}|r_{n+1}(z)-r_{n+1}(w)|.
	\end{align*}
	We have that $d_n(z)\asymp d_n(c)\asymp q_n$ since $z\in U_c^n$. Further, $d_{n+1}(z)\asymp d_n(c)$ with a different constant, so we can estimate \begin{equation*}
		d_n(z)\left|\frac{t-s}{q_n}\right|+d_{n+1}(z)\left|\frac{t-s}{q_n}\right|\lesssim |t-s|.
	\end{equation*}
	
	Further we used that $|r_n(z)-r_n(w)|=|d_n(z)-d_n(w)|$ and at this point we apply Lemma \ref{lemma:radius_Lipschitz_Ucrit} to conclude that \begin{align*}
		\frac{q_n-s}{q_n}|r_n(z)-r_n(w)|&+\frac{s}{q_n}|r_{n+1}(z)-r_{n+1}(w)|\\
		&\leq \frac{q_n-s}{q_n}L_{r_n}d_n(z) |z-w|_{\hat{\mathbb{C}}}+\frac{s}{q_n}L_{r_{n+1}}d_n(z)|z-w|_{\hat{\mathbb{C}}}\\
		&\leq 2L_{r_n}d_n(z) |z-w|_{\hat{\mathbb{C}}}.
	\end{align*}
	Let us summarize the upper bound
	\begin{align*}|\alpha_1(x, t)-\alpha_1(y, s)|&\lesssim (\phi^{-1})^{\#}(z)|x-y| + 2|t-s|+ 2L_{r_n}d_n(z)|z-w|_{\hat{\mathbb{C}}}\\
		&\lesssim 2(\phi^{-1})^{\#}(z)L_{r_n} |x-y| +2|t-s|.
	\end{align*} 
	In the last step we used \eqref{equation: estimate lem 4.19}, that $L_{r_n}\geq 1$ and that $r_{n+1}(z)+d_n(z)\lesssim 1$ with $C(\lesssim)$ uniform. Overall we get a constant $L= L(|\lambda|^{-(n-1)}, |\lambda|^{-(n-1)/d}, d, D\phi)$ that is decreasing as $n$ increases.

	For a lower bound note that since 
	\begin{align*}
		\rho_n(z, t)
		&=
		r_{n}(z)\frac{q_n-t}{q_n}+r_{n+1}(z)\frac{t}{q_n}\\
		\rho_n(w, s)
		&=
		r_{n}(w)\frac{q_n-s}{q_n}+r_{n+1}(w)\frac{s}{q_n}
	\end{align*}
	then we find that 
	\begin{align*}
		t
		&=
		q_n\cdot \frac{\rho_n(z, t)-r_{n}(z)}{r_{n+1}(z)-r_{n}(z)}\\
		s
		&= 
		q_n\cdot \frac{\rho_n(w, s)-r_{n}(w)}{r_{n+1}(w)-r_{n}(w)}.
	\end{align*}
	
	Using the above expressions we can find a lower bound to
	\begin{equation*}
		|t-s|
		=
		q_n\left|\frac{\rho_n(z, t)-r_{n}(z)}{r_{n+1}(z)-r_{n}(z)}-\frac{\rho_n(w, s)-r_{n}(w)}{r_{n+1}(w)-r_{n}(w)}\right|,
	\end{equation*}
	first we use Lemma \ref{lemma:mon_radii_crit} \ref{item:mon_radii_crit_3} to conclude that there exists a constant $a$ such that $r_{n+1}(z)-r_n(z)\geq a\cdot d_n(z)$ and the same statement for $w$. We also expand the expression using the triangle inequality 
	\begin{align*}
		|t-s| 
		\leq 
		\frac{q_n}{a^2\cdot d_n(z)d_n(w)}|\big(&\rho_n(z, t)-r_n(z)\big)\big(r_{n+1}(z)-r_n(w)\big)\\
		&-
		\big(\rho_n(w, s)-r_n(w)\big)\big(r_{n+1}(z)+r_{n+1}(w)\big)|
	\end{align*}
	Then we can estimate \begin{align*}
		|\big(&\rho_n(z, t)-r_n(z)\big)\cdot \big(r_{n+1}(z)-r_n(w)\big)\\
		&-\big(\rho_n(w, s)-r_n(w)\big)\cdot \big(r_{n+1}(z)+r_{n+1}(w)\big)|\\
		&\leq \big(\rho_n(z, t)-r_n(z)\big)\cdot |r_{n+1}(w)-r_n(w)-r_{n+1}(z)+r_n(z)|\\
		&+|\rho_n(z, t)-r_n(z)-\rho_n(w, s)+r_n(w)|\cdot\big(r_{n+1}(z)-r_n(z)\big)\\
		&\lesssim d_{n+1}(z)\cdot \big(2L_{r_n}d_n(z)+L_{r_{n+1}}(z)d_{n+1}(z)\big)|z-w|_{\hat{\mathbb{C}}}\\
		&+ d_{n+1}(z)|\rho_n(z, t)-\rho_n(w, s)|\\	
		&\lesssim d_{n+1}(z)\cdot \big(2L_{r_n}d_n(z)+L_{r_{n+1}}(z)d_{n+1}(z)\big)|z-w|_{\hat{\mathbb{C}}}\\
		&+ d_{n+1}(z)|\rho_n(z, t)-\rho_n(w, s)|\\	
		&\lesssim d_{n+1}(z) d_n(z)|z-w|_{\hat{\mathbb{C}}} + d_{n+1}(z)|\rho_n(z, t)-\rho_n(w, s)|
	\end{align*}
	where $C(\lesssim)=C(|\lambda|^{-(n-1)}, |\lambda|^{-n-1/d}, d, \phi^{\#})$ as in Lemma \ref{lemma:radius_Lipschitz_Ucrit}. 
	
	We can summarize this estimate to \begin{align*}
		|t-s|&\lesssim q_n |z-w|_{\hat{\mathbb{C}}}+|\rho_n(z, t)-\rho_n(w, s)|
	\end{align*}
	So we can estimate that \begin{align*}
		|(x, t)-(y, s)|&\asymp |x-y|+|t-s|\\
		&\lesssim \phi^{\#}(z)\cdot |z-w|_\mathbb{C} + q_n|z-w|_{\CDach}+ |\rho_n(z, t)-\rho_n(w, s)|\\
		&\leq
		L' |(z, \rho_n(z, t))-(w, \rho_n(w, s))|
	\end{align*}
	with $L'=L'(|\lambda|^{-(n-1)}, |\lambda|^{-(n-1)/d}, \phi^{\#}, d)$. Now take the maximum of $L, L'$.
\end{proof} 

\begin{lemma}\label{lemma:parap_of_P(Up)}
	For all $n\in \N$ and all post-critical points $p\in \CDach$ we find a parametrization of $P(U_p^{n-1}, n-1)$. For $q_{n-1}'= d_{n-1}(p)$ the map is defined as follows \begin{align*}
		\alpha_2:|\lambda|^{-(n-1)}\D\times[0, q_{n-1}']&\to P(U^{n-1}_p, n-1)\\
		(x, t)&\mapsto (\psi^{-1}(x), \rho_{n-1}(\psi^{-1}(x), t)).
	\end{align*}
	It is a bi-Lipschitz map with constant $L_{\alpha_2}=L_{\alpha_2}(|\lambda|^{-(n-1)/d}, |\lambda|^{-(n-1)}, d)$ that is decreasing as $n$ increases and can hence be chosen uniformly.
	Here $\rho_{n-1}(z)$ is as in Equation \eqref{eq:conv_rho_n_height_h}.
\end{lemma}

This is very similar to Lemma \ref{lemma:param_of_P(U_c)}. We will first prove two lemmas that will be used in the proof.

\begin{lemma}
	\label{lemma:radii_Lipsch_Upost}
	Let $z\in \U_{\post}^{n}$. Then for $z \mapsto r_{n}(z)$ we can show that  \begin{equation*}
		|r_{n}(z)-r_{n}(w)|
		\leq 
		L_r \cdot d_{n}(z)|z-w|_{\hat{\mathbb{C}}},
	\end{equation*} 
	for all $w\in \U_{\post}^{n}$ where $L_{r}=L_r(|\lambda|^{-(n-1)})$ is decreasing as $n$ increases.
\end{lemma}
\begin{proof}
	Let $z, w\in U_p^n$ be arbitrary. By Theorem \ref{theorem: spherical Koebe} there exists a constant $L_r$ such that 
	\begin{equation*}
		|(f^n)^{\#}(z)-(f^n)^{\#}(w)|
		\leq 
		(f^n)^{\#}(w)|L_r|z-w|_{\hat{\mathbb{C}}}.
	\end{equation*} 
	So we can estimate \begin{align*}
		|r_n(z)-r_n(w)|
		&=\left|\frac{1}{(f^n)^{\#}(w)}-\frac{1}{(f^n)^{\#}(z)}\right|
		=\frac{|(f^n)^{\#}(z)-(f^n)^{\#}(w)|}{(f^n)^{\#}(z)(f^n)^{\#}(w)}\\
		&\leq \frac{L_{r}(f^n)^{\#}(w) |z-w|_{\hat{\mathbb{C}}} }{(f^n)^{\#}(z)(f^n)^{\#}(w)} \leq L_{r}\cdot d_n(z)|z-w|_{\hat{\mathbb{C}}}.
	\end{align*}
	The constant $L_{r_n}=L(|\lambda|^{-(n-1)})$ is decreasing as $n$ increases and can hence be chosen uniformly over $n$.
\end{proof}
\begin{lemma}
	\label{lemma:lower_bound_mon_radii_Up}
	Let $z\in \U_{\post}^n$. Then there exist constants $a$ and $b$ such that
	\begin{equation*}
		a\cdot d_{n-1}(z)\leq r_n(z)-r_{n-1}(z)\leq b\cdot d_{n-1}(z)
	\end{equation*}
	where $a, b$ are independent of $n$.
\end{lemma}
\begin{proof}
	We can compute \begin{align*}
		r_n(z)-r_{n-1}(z)=1-\frac{1}{(f^n)^{\#}(z)}-(1-\frac{1}{(f^{n-1})^{\#}(z)})=\frac{f^{\#}(f^{n-1}(z))-1}{(f^{n-1})^{\#}(z)}
	\end{align*}
	For the upper bound note that $f^{n-1}(z)$ is contained in a compact subset in $\U_{\post}$ hence there is a maximum of $f^{\#}(f^{n-1}(z))$ independent of $n$ so we set 
	\begin{equation*}
		b
		=
		\max_{w\in \U_{\post}}f^{\#}(w)-1.
	\end{equation*} 
	For the lower bound recall that $U_p$ is a domain where we have a linearization and the map $f$ is expanding, hence $f^{\#}(w)\geq 1+\epsilon$ for some $\epsilon>0$. Then the claim holds with $a\leq \epsilon$. 
\end{proof}
\begin{proof}[Proof of Lemma \ref{lemma:parap_of_P(Up)}]
	Denote \begin{equation*}
		z:=\psi^{-1}(x)\textup{ and }w:= \psi^{-1}(y).
	\end{equation*}
	We will use that 
	\begin{align*}|(z, \rho_{n-1}(z, t))&-(w, \rho_{n-1}(w, s))|\asymp \frac{1}{2}\big(\rho_{n-1}(z, t)+\rho_{n-1}(w, s)\big)|z-w|_{\hat{\mathbb{C}}}\\ 
		&+|\rho_{n-1}(z, t)-\rho_{n-1}(w, s)|.
	\end{align*}
	For the first summand we can estimate the
	distance $|z-w|_{\hat{\mathbb{C}}}$ using Lemma~\ref{lem:Koebe_phi_psi}~\ref{item:Koebe_psi}
	\begin{equation*}
		|z-w|_{\hat{\mathbb{C}}}
		\asymp 
		(\psi^{-1})^{\#}(x)\cdot |x-y|,
	\end{equation*} 
	where $C(\asymp)$ is a uniform constant.
	By Lemma \ref{lemma:radii_Lipsch_Upost} we obtain that 
	\begin{equation*}
		|r_n(z)-r_n(w)|
		\leq 
		L_r \cdot d_n(z)|z-w|_{\CDach}
	\end{equation*}
	where $L_r=L_r(n)$ is decreasing as $n$ increases.
	
	For estimates of the difference in the radii we compute similar as in Lemma \ref{lemma:param_of_P(U_c)}
	\begin{align*}
		&|\rho_{n-1}(z, t)-\rho_{n-1}(w, s)|=|\delta_{n-1}(z)-\delta_{n-1}(w)|\\
		&=\left|d_{n-1}(z)\frac{q_{n-1}'-t}{q_{n-1}'}+d_{n}(z)\frac{t}{q_{n-1}'}
		-\left(d_{n-1}(w)\frac{q_{n-1}'-s}{q_{n-1}'}
		+d_{n}(w)\frac{s}{q_{n-1}'}\right)\right|\\
		&\leq \big(d_{n-1}(z)+d_{n}(z)\big)\left|\frac{s-t}{q_{n-1}'}\right|
		+L_r\cdot d_{n-2}(z)|x-y|\\
		&+L_r\cdot d_{n-1}(z)|x-y|.
	\end{align*}
	To summarize, we get that \begin{align*}
		|(z, &\rho_{n-1}(z, t))-(w, \rho_{n-1}(w, s))|\\
		&\lesssim (\psi^{-1})^{\#}(x)|x-y|
		+L_r \big(d_{n-2}(z)+d_{n-1}(z)\big)|x-y|\\
		&+ \frac{d_{n-1}(z)+d_n(z)}{q_n'}|s-t|\\
		&\leq L\big(|x-y|+|t-s|\big)
	\end{align*}
	with $L=L(|\lambda|^{-(n-1)}, (\psi^{-1})^{\#})$.
	
	For proving the lower bound recall that  
	\begin{equation*}
		\rho_{n-1}(z, t)
		=
		r_{n-1}(z)\frac{q_{n-1}'-t}{q_{n-1}'}+r_{n}(z)\frac{t}{q_{n-1}'}
	\end{equation*}
	solving for $t$ yields
	\begin{equation*}
		t
		=
		q_{n-1}' \frac{\rho_{n-1}(z, t)-r_{n-1}(z)}{r_{n}(z)-r_{n-1}(z)}.
	\end{equation*}
	The analogous expression is obtained for $\rho_{n-1}(w, s), y$ and $s$. By Lemma \ref{lemma:lower_bound_mon_radii_Up} there exists a uniform constant $a$ such that $r_n(z)-r_{n-1}(z)\geq a\cdot d_{n-1}(z)$. Using this Lemma and the estimates above we obtain: 
	\begin{align*}
		&|t-s|
		=
		q_{n-1}'\left|\frac{\rho_{n-1}(z, t)-r_{n-1}(z)}{r_{n}(z)-r_{n-1}(z)}-\frac{\rho_{n-1}(w, s)-r_{n-1}(w)}{r_{n}(z)-r_{n-1}(w)}\right|\\
		\leq
		&\frac{q_{n-1}'}{a^2\cdot d_{n-1}(z)d_{n-1}(w)}\big|\big(\rho_{n-1}(z, t)-r_{n-1}(z)\big)\big(r_n(w)-r_{n-1}(w)\big)\\
		-&\big(\rho_{n-1}(w, s)-r_{n-1}(w)\big)\big(r_n(z)-r_{n-1}(z)\big)\big|\\
		\lesssim &\frac{q_{n-1}'}{a^2\cdot d_{n-1}(z)d_{n-1}(w)}\big(d_n(z)d_{n-1}(z)|z-w|_{\CDach}+d_{n}(z)|\rho_{n-1}(z, t)-\rho_{n-1}(w, s)|\big)\\
		\lesssim 
		&q_{n-1}'|z-w|_{\hat{\mathbb{C}}}+ |\rho_{n-1}(z, t)-\rho_{n-1}(w, s)|
	\end{align*}
	as in the proof of Lemma \ref{lemma:param_of_P(U_c)}. To summarize: \begin{align*}
		|x-y|+|t-s|&\lesssim \psi^{\#}(z) |z-w|_{\CDach} q_n'|z-w|_{\CDach} \\
		&+ |\rho_{n-1}(z, t)-\rho_{n-1}(w, s)|\\
		&\lesssim L'|(z, \rho_{n-1}(z, t))-(w, \rho_{n-1}(w, s))|
	\end{align*}
	with $L'=L'(|\lambda|^{-(n-1)}, (\psi)^{\#})$. Then we take the maximum of $L, L'$ as the constant. 
\end{proof}

\begin{lemma}\label{lemma: map between cylinders}
	The map 
	\begin{align*}
		\beta:|\lambda|^{-n/d}\Db\times [0, q_n] &\to |\lambda|^{-(n-1)}\Db\times[0, q_{n-1}']\\
		(x, t) &\mapsto (h_{n}(x)\frac{q_n-t}{q_n}+h_{n+1}(x)\frac{t}{q_n}, \mu t)
	\end{align*}
	with $\mu=\frac{q_{n-1}'}{q_n}\asymp |\lambda|^{-(n-1)(1-1/d)}$ 
	is Lipschitz and when restricted to sectors (recall the notation\eqref{eq:def_sector}) $|\lambda|^{-n/d}\Db(\theta_0, \pi/d)\times [0, q_n]$ for some $\theta_0\in [0, 2\pi]$ the map is bi-Lipschitz.
\end{lemma}
\begin{proof}
	Let $(x, t), (y, s) \in |\lambda|^{-(n-1)/d}\D\times [0, q_n]$ be arbitrary. To compute the difference \begin{align*}
		&\left|(h_{n}(x)\frac{q_n-t}{q_n}+h_{n+1}(x)\frac{t}{q_n}, \mu t)-(h_{n}(y)\frac{q_n-s}{q_n}+h_{n+1}(y)\frac{s}{q_n}, \mu s)\right|\\
		\asymp &\left|h_{n}(x)\frac{q_n-t}{q_n}+h_{n+1}(x)\frac{t}{q_n}-(h_{n}(y)\frac{q_n-s}{q_n}+h_{n+1}(y)\frac{s}{q_n})\right|+\mu|t-s|\\
		\leq &\left|h_{n}(x)\frac{q_n-t}{q_n}-h_{n}(x)\frac{q_n-s}{q_n}\right|+\left|h_{n}(x)\frac{q_n-s}{q_n}-h_{n}(y)\frac{q_n-s}{q_n}\right|\\
		+&\left|h_{n+1}(x)\frac{t}{q_n}-h_{n+1}(x)\frac{s}{q_n}\right|+\left|h_{n+1}(x)\frac{s}{q_n}-h_{n+1}(y)\frac{s}{q_n}\right|+\mu|t-s|\\
		\leq&\left|\frac{s-t}{q_n}\right|\big(|h_{n}(x)|+\left|h_{n+1}(x)\right|\big)+\frac{q_n-s}{q_n}L_{h_{n}}|x-y|+\frac{s}{q_n}L_{h_{n+1}}|x-y|+\mu|t-s|\\
		\leq &\left|\frac{s-t}{q_n}\right|\big(|h_{n}(x)|+|h_{n+1}(x)|\big)+\mu|t-s|+2\max\{L_{h_{n}}, L_{h_{n+1}}\}|x-y|\\
		\lesssim&\frac{2|\lambda|^{-(n-1)}}{q_n}|t-s|+\mu |t-s|+2 L_{h_n}|x-y|
	\end{align*}
	The maps $h_{n}$ and $h_{n+1}$ are bounded hence the map is Lipschitz. The constants $L_{h_n}=d|\lambda|^{-(n-1)(1-1/d)}$ and $|h_n(z)|\leq |\lambda| ^{-(n-1)}$. 
	We can estimate \begin{equation*}
		\frac{2|\lambda|^{-(n-1)}}{q_n}\asymp |\lambda|^{-(n-1)(1+1/d)}.
	\end{equation*}
	
	It remains to find a lower bound to the difference, we use that the sum norm is equivalent to a maximum norm.
	In order to prove that it is bi-Lipschitz on suitable subsets we restrict to sectors $|\lambda|^{-n/d}\Db(\theta_0, \pi/d)\times [0, q_n]$.
	By Lemma \ref{lemma: h is bilipschitz} \ref{item:prop_hd_1} the map $h_n$ restricted to sectors $\Db(\theta_0, \pi/d)$ is bi-Lipschitz.
	First find a lower bound for points $(x, t), (y, t)\in |\lambda|^{-n/d}\Db(\theta_0, \pi/d)\times [0, q_n]\times [0, q_n]$ for any $\theta_0\in [0, 2\pi]$. 
	\begin{align*}
		&\left|\left(h_n(x)\frac{h_n-t}{q_n}+h_{n+1}(x)\frac{t}{q_n}, \mu t\right)-\left(h_n(y)\frac{q_n-t}{q_n}+h_{n+1}(y)\frac{t}{q_n}, \mu t\right)\right|\\
		\asymp &\left|h_n(x)\frac{q_n-t}{q_n}+h_{n+1}(x)\frac{t}{q_n}-\left(h_n(y)\frac{q_n-t}{q_n}+h_{n+1}(y)\frac{t}{q_n}\right)\right|+\mu|t-t|
	\end{align*}
	In polar coordinates we get that if $x=r_xe^{i\theta}$ we have $h_n(x)=|\lambda|^{-(n-1)(1-1/d)}r_xe^{id\theta}$ and $h_{n+1}(x)=|\lambda|^{-n(1-1/d)}r_xe^{id\theta}$. Therefore 
	\begin{align*}
		\frac{q_n-t}{q_n}h_n(x)+\frac{t}{q_n}h_{n+1}(x)
		&= r_x\left(\frac{q_n-t}{q_n}|\lambda|^{-(n-1)(1-1/d)}+\frac{t}{q_n}|\lambda|^{-n(1-1/d)}\right)e^{id\theta}\\
		&=:r_x\Lambda(n-1, t)e^{id\theta}.
	\end{align*}
	For $y=r_ye^{i\varphi}$ we get a similar expression. To estimate the difference we compute \begin{align*}
		&\left|h_n(x)\frac{q_n-t}{q_n}+h_{n+1}(x)\frac{t}{q_n}-\left(h_n(y)\frac{q_n-t}{q_n}+h_{n+1}(y)\frac{t}{q_n}\right)\right|\\
		\asymp &\frac{1}{2}\big(r_x\Lambda(n-1, t)+r_y\Lambda(n-1, t)\big)|d\theta-d\varphi|+|r_x\Lambda(n-1, t)-r_y\Lambda(n-1, t)|\\
		\gtrsim &\Lambda(n-1, t)|x-y|
	\end{align*}
	Now for $t\geq s$ we get that \begin{align*}
		&\left|\left(h_n(x)\frac{q_n-t}{q_n}+h_{n+1}(x)\frac{t}{q_n}, \mu t\right)-\left(h_n(y)\frac{q_n-s}{q_n}+h_{n+1}(y)\frac{s}{q_n}, \mu s\right)\right|\\
		\gtrsim &\left|h_n(x)\frac{q_n-t}{q_n}+h_{n+1}(x)\frac{t}{q_n}-\left(h_n(y)\frac{q_n-s}{q_n}+h_{n+1}(y)\frac{s}{q_n}\right)\right|+\mu|s-t|\\
		\geq &\Lambda(n-1, t)|x-y|+\mu|t-s| \geq |\lambda|^{-n(1-1/d)}(|x-y|+|t-s|)
	\end{align*}
	The scaling factor can be chosen to be $|\lambda|^{-(n-1)(1-1/d)}$ and there are uniform Lipschitz constants $L_{\beta}, L_{\beta}'$.
\end{proof}

\subsection{Quasi-regularity of the extension}
In this section we will show that the map $F\colon\Omega\to \R^3$ is continuous and quasi-regular. 

It is clear that the extension $F$ is continuous on prisms $P(U_c^n, n)$ near the critical axes as well as on prisms $P(U, n)$ where $U\subset \CDach$ is open and disjoint from $\U_{\crit}^n$. It remains to show that the boundary values $F(z, \rho_n(z, t))$ for $z\in \partial \U_{\crit}^n$ agree.

\begin{lemma}
	The extending map $F$ is continuous. 
\end{lemma}

\begin{proof}
	We have shown continuity whenever $z\in P(U, n)$ with $U\subset \CDach\setminus \U_{\crit}^n$ and when $z\in P(U_c^n, n)$ for a critical point $c$. By construction $\alpha_2\circ\beta\circ \alpha_1^{-1}$ agrees with $F|_{S^n}$. It remains to show that when $z\in\partial \U_{\crit}^n$ the map $F$ defined outside agrees with the extension of $\alpha_2\circ \beta\circ\alpha_1^{-1}$ to the boundary of $P(U_c^n, n)$. Now fix a critical point $c$ and $n\in \N$. We find that 
	\begin{equation*}
		\alpha_1^{-1}(z, \rho_n(z, t))=(\phi(z), q_n\cdot t).
	\end{equation*}
	Denote $w:=\phi(z)$, because $z\in\partial U_c^n$ we have that $|w|=|\lambda|^{-(n-1)/d}$, which implies that $h_n(w)=h_{n+1}(w)=w^d$. So if we apply $\beta$ we obtain \begin{align*}
		\beta (w, q_n\cdot t)&= \left(h_n(w)(1-t)+h_{n+1}(w)t, \frac{q_n'}{q_n}q_nt\right)\\
		&= (w^d, q_n't).
	\end{align*}
	Then the composition with $\alpha_2$ yields
	\begin{equation*}
		\alpha_2(w^d, q_n't)= \big(\psi^{-1}(w^d), \rho_{n-1}(\psi^{-1}(w^d), t)\big).
	\end{equation*}
	Recall that on $U_c^1\setminus U_c^n$ we have that $f_n(z)=f(z)$ and $f(z)=\psi^{-1}(w^d)$ in the coordinates above. 
	
	Therefore the values at the boundary agree and by the gluing lemma the extension is continuous. 
\end{proof}

Now we are in shape to prove the main theorem. 

\begin{proof}[Proof of Theorem \ref{theorem: extension result}]
	We will prove the assertion on a subset of full measure of $A(\CDach, S_{N_0})$. Recall that $N_0$ is as in Corollary \ref{cor:S1_in_SN0}.
	Let 
	\begin{equation*}
		N
		:= 
		\bigcup_{n, c} \left\{(f^{-n}(c), r):r\in[0, 1]\right\}\cup  \bigcup_{n} \partial P(\U_{\crit}^n, n),
	\end{equation*}
	i.e., the pre-critical axes and the boundaries of prisms around critical points. Note that $N$ is a set of Lebesgue measure 0. Now let $p\in A(\CDach)\setminus N$ be arbitrary. Suppose $P(U, n)$ is an $n$-prism such that $p\in P(U, n)$, and $p=(z, \rho_n(z, t))$ for $z\in \CDach$ and $t\in[0, 1]$ suitable. The orbit of $z$ under $f_1\circ\dots\circ f_m$ for $m=n-N_0$ is of either of the forms below \begin{equation*}
		\begin{cases}
			f^{j-1}\circ f_j\circ f^{m-j}(z) \quad &\textup{if } f^{m-j}(z)\in \U_{\crit}^{j}, \, j\leq m\\
			f^m(z) \quad &\textup{else.}
		\end{cases}
	\end{equation*}
	In the first case, if needed restrict to $U'\subset U$ such that $f^{m-j}(U')$ is contained in a hemisphere, $f^{m-j}|_{U'}$ is univalent and further assume that $f^{m-j}(U')\subset U_c^j$. Then by Corollary \ref{corollary: iterate of extension is quasi-sim on P(U, n)} the map $F^{m-j}$ is a $(L, L', (f^{m-j})^{\#}(z))$-quasi-similarity.
	
	Let $P(V, j):= F^{m-j}(P(U', n)) \subset P(U_c^j, j)$. We can restrict $U'$ further such that $V\subset U_c^j(z_0, \pi/d)$ (recall \eqref{eq:Sectors_in_sphere}) for some $z_0\in V$, in particular where $\alpha_2\circ \beta\circ \alpha_1^{-1}|_{P(V, j)}$ is injective. In that case $\alpha_2\circ \beta\circ \alpha_1^{-1}$ is a $(C, C')$-bi-Lipschitz map by the Lemmas \ref{lemma:param_of_P(U_c)}, \ref{lemma: map between cylinders} and \ref{lemma:parap_of_P(Up)}.
	
	Let $w:= f_j\circ f^{m-j}(z)$ and let $P(W, j-1)=\alpha_2\circ \beta\circ \alpha_1^{-1}\subset P(U_p^j, j-1)$. Then since $W\subset U_p^j$ we can apply Corollary \ref{corollary: iterate of extension is quasi-sim on P(U, n)} to conclude that $F^{j-1}$ is a $(L, L', (f^{j-1})^{\#}(w))$ quasi-similarity. 
	
	Let $r>0$ be such that $B(p, \epsilon)\subset P(U', n)$ and on $f^{m-j}(B(z, r))$ the map $\alpha_2\circ\beta\circ\alpha_1^{-1}$ is injective. Then for all $\epsilon <r$ we have that 
	\begin{align*}
		K(p)=\liminf_{r>\epsilon \to 0}\frac{\max_{|q-p|=\epsilon} |F^m(p)-F^m(q)|}{\min_{|p-q'|=\epsilon}|F^m(p)-F^m(q')|} \leq \frac{L}{L'}\cdot \frac{C}{C'}\cdot \frac{L}{L'}.
	\end{align*}
	Since the involved constants are uniform, we have proved the assertion in the first case. 
	
	The second case is the same as the first part of the first case, so we pick $U\subset \CDach$ such that $p\in P(U, n)$ and $f^j(U)$ is contained in a hemisphere as well as $f^j(U)$ is univalent for all $j\leq n-N_0$. Then for $r>0$ such that $B(p, r)\subset P(U, n)$ we have that 
	\begin{equation*}
		K(p)\leq \frac{L}{L'}. \qedhere
	\end{equation*}
\end{proof}

\section{Removing the assumption that critical values are fixed}\label{section:removing_assumption}
There are several cases to consider:
\begin{enumerate}
	\item there is a critical point $c$ such that there is another critical point $c'$ in its orbit that is then mapped to a post-critical fixed point,
	\item the orbit of a critical point passes to $p_1, \dots, p_i$ and then lands on a post-critical fixed point $p$,
	\item a critical point lands in a periodic orbit of post-critical points.
\end{enumerate}
In the following we will treat all these cases separately and in case a map has all three properties then one can patch the methods for dealing with each of the cases together.
\subsection{Several critical points in an orbit}
For simplicity we assume that a critical point $c$ has the orbit $c\mapsto c' \mapsto p$ where $c'$ is another critical point and $p$ is a fixed point. Let $d$ and $d'$ denote the degree of $c$ respectively $c'$. Because the map $f$ is expanding, $p$ is a repelling fixed point and as above we can linearize the map around a closed neighborhood $U_p^1$. 

\begin{equation*}
	\xymatrix{
		U_p^1 \ar[r]^{f} \ar[d]_{\psi}
		& U_p^0 \ar[d]^{\psi} \\
		|\lambda|^{-1}\Db \ar[r]^{\delta_{\lambda}}
		&  \Db,
	}
\end{equation*}

As above, we define $U_p^n$ as the unique connected component of $f^{-n}(U_p^0)$ containing $p$. In particular $U_p^n=\psi^{-1}(\lambda^{-n}\Db)$. Then we define the neighborhoods around $c'$ as above with $U_{c'}^n$ as the unique component of $f^{-n}(U_p^0)$ containing $c'$. Similarly $U_c^n$ is the unique component of $f^{-n}(U_p^0)$ containing $c$ for $n\geq 2$. There are maps $\phi'\colon U_{c'}\to \Db$ and $\phi\colon U_c\to \Db$ such that 

\begin{equation*}
	\xymatrix{
		U_c^2 \ar[r]^{f} \ar[d]_{\phi} & U_{c'}^1 \ar[r]^{f} \ar[d]^{\phi'} 
		& U_p^0 \ar[d]^{\psi} \\
		\Db \ar[r]^{z^d}  & \Db
		\ar[r]^{z^{d'}}
		& \Db 
	}
\end{equation*}

Define $h_{n, c}$ and $h_{n, c'}$ analogously to the definition in Section \ref{sec:defining-maps-f_n}. Let \begin{align*}
	h_{n+1, c}\colon |\lambda|^{-(n-1)/d}\overline{\D}&\to |\lambda|^{-(n-1)}\overline{\D}\\
	z&\mapsto \delta_{|\lambda|^{-(n-1)}} \circ h_d\circ \delta_{|\lambda|^{(n-1)/d}}(z)
\end{align*}
and \begin{align*}
	h_{n, c'}\colon|\lambda|^{-(n-1)/d'}\overline{\D}&\to |\lambda|^{-(n-1)}\overline{\D}\\
	z&\mapsto \delta_{|\lambda|^{-(n-1)}} \circ h_{d'}\circ \delta_{|\lambda|^{(n-1)/d}}(z).
\end{align*}

On the sets $U_c^n$ and $U_{c'}^n$ we will define the function $f_n$ differently from $f$. 
\begin{equation*}
	\xymatrix{
		U_c^{n+1} \ar[r]^{f}_{f_{n+1}} \ar[d]_{\phi} & U_{c'}^n \ar[r]^{f}_{f_n} \ar[d]^{\phi'} 
		& U_p^{n-1} \ar[d]^{\psi} \\
		|\lambda|^{-(n-1)/dd'}\Db \ar[r]^{z^d}_{h_{n-1}}  & |\lambda|^{-(n-1)/d'}\Db
		\ar[r]^{z^{d'}}_{h_n'}
		& |\lambda|^{-(n-1)}\Db}
\end{equation*}

For $n\in \N$ we define \begin{equation*}f_n(z)=\begin{cases}
		(\psi^{-1}\circ h_n'\circ \phi')(z) \quad &\textup{if } z \in U_{c'}^n\\
		(\phi'^{-1}\circ h_n\circ\phi)(z) \quad &\textup{if } z\in U_c^n\\
		f(z) \quad &\textup{else.}
\end{cases}\end{equation*}

With this we define a local radius as follows 
\begin{equation*}r_n(z)= 1-s\, \frac{1}{\|D(f_1\circ\dots\circ f_n)(z)\|_{\CDach}}
\end{equation*}
where $s\in(0, 1]$ is a constant chosen to guarantee that $r_n(z)\in (0, 1)$ for all $n\in \N $ and all $z\in \CDach$. 

Then the computations and results from above can be adapted to the new sequence $f_n$. A difference is that one can get $f_1\circ\dots\circ f_n= f^{n-2}\circ f_{n-1}\circ f_{n}$ for $z\in U_c^n$. 

\subsection{Pre-fixed critical value}
We assume that there is a critical point $c$ with the orbit \[c \mapsto p_1\mapsto \dots \mapsto p_{N-1}\mapsto p_N\circlearrowleft.\] We assume that none of the points $p_i$ for $i=1, \dots, N$ is a critical point.

We will use the linearization and define $U_p^n$ as above, i.e., $U_p^1$ such that there exists $\psi\colon U_p^0\to \Db$ with 
\begin{equation*}
	\xymatrix{ U_{p_N}^1\ar[r]^{f} \ar[d]_{\psi} & U_{p_N}^0 \ar[d]^{\psi}\\
		|\lambda|^{-1/d}\Db \ar[r]^{z^d} & \Db.
	}
\end{equation*}
We will then define $U_{p_i}^n$ as the unique connected component of $f^{-n}(U_{p_N}^0)$ containing $p_i$ for $i=1, \dots, N$. For the critical point $c$ there are neighborhoods $f^{-n}(U_p^0)$ containing $c$ for all $n\geq N$
\begin{equation*}
	\xymatrix{ U_c^N \ar[r]^{f} & U_{p_1}^{N-1}\dots \ar[r]^{f} & U_{p_{N-1}}^1\ar[r]^{f}&U_{p_N}^0.
	}
\end{equation*}
By choosing $U_p^0$ is small enough we can assume that $f^N|_{U_{p_1}^{N-1}}$ is a homeomorphism. 

\begin{equation}
	\xymatrix{
		U_c^{N+1} \ar[r]^{f} \ar[d]_{\phi} & U_{p_1}^N \ar[r]^{f} \ar[d]^{\varphi} \dots & U_{p_{N-1}}^2 \ar[r]^{f} & U^1_{p_N} \ar[r]^{f} \ar[d]^{\psi} & U_{p_N} \ar[d]^{\psi}\\
		|\lambda|^{-1/d}\Db \ar[r]^{z^d} & |\lambda|^{-1}\Db &  & |\lambda|^{-1}\Db \ar[r]^{\delta_{|\lambda|}} &\Db\\
	}
\end{equation}

Then we define $f_n$ for $n\geq N$ by \begin{equation*}f_n(z)=\begin{cases}
		(\varphi\circ h_{n-N}\circ \phi)(z) \quad &\textup{if } z\in U_c^{n}\\
		f(z) \quad &\textup{else}
\end{cases}\end{equation*}

Then if we look at points $z\in U_c^n$ the orbit $f_1\circ\dots\circ f_n$ for $n\geq N$ can be expressed as \begin{align*}f_1\circ\dots\circ f_n(z)&=f^{n-1}\circ f_n(z) \\ &= \psi^{-1}\circ \delta_{|\lambda|^{n-N}}\circ\psi\circ f^{N-1}\circ \varphi^{-1}\circ h_n\circ\phi(z).\end{align*} The part $\psi^{-1}\circ \delta_{|\lambda|^{n-N}}\circ \psi$ comes from the linearization around $p_N$.

\begin{equation}
	\xymatrix{
		U_c^{n+N} \ar[r]^{f}_{f_{n+N}} \ar[d]_{\phi} & U_{p_1} \ar[r]^{f} \ar[d]^{\varphi} \dots & U_{p_{N-1}} \ar[r]^{f} & U^1_{p_N} \ar[r]^{f} \ar[d]^{\psi} & U_{p_N} \ar[d]^{\psi}\\
		|\lambda|^{-n/d}\Db \ar[r]^{z^d}_{h_n} & |\lambda|^{-1}\Db &  & |\lambda|^{-1}\Db \ar[r]^{\delta_{\lambda}} &\Db
	}
\end{equation}

For $n\geq N$ we will define the functions $r_n$ as 
\begin{equation*}r_n(z)=1- s\, \frac{1}{\|D(f_1\circ\dots\circ f_n)(z)\|_{\CDach}} 
\end{equation*}
with $s\in (0, 1]$ so that for all $n\in \N$ and all $z\in \CDach$ $r_n(z)\in (0, 1)$. Then we define $S_n$ for $n\geq N$ as above and extend the map as above. 

\subsection{Periodic critical value}

For the sake of simplicity we will assume that the critical value is of period 2 and the pre-image of the critical point is regular, i.e.,

\begin{tikzcd}[cramped, sep=small]
	&v\arrow[r] &c \arrow[r] &p_1 \arrow[r] &p_2 \arrow[l, bend right=30]
\end{tikzcd}

\noindent where $c$ is a critical point of degree $d$, the points $v, p_i$ for $i=1, 2$ are not critical points. We will denote $f^2$ by $F$ and use capital letters $F_n, D_n, R_n$ for the analogs of $f_n, d_n, r_n$ for the function $F$. We will denote an extension of $f$ by $\overline{f}$. By $\Lambda$ we denote the multiplier of $F$ at $p_1$, i.e., $\Lambda=F^{\#}(p_1)=f^{\#}(p_1)f^{\#}(p_2)$. Let $U_{p_1} \ni p_1$ be a closed neighborhood on which the linearization from Koenig's Linearisation Theorem is defined,

\begin{equation*}
	\xymatrix{
		U_{p_1}^2 \ar[r]^{F} \ar[d]_{\psi}
		& U_{p_1}^0 \ar[d]^{\psi} \\
		\Lambda^{-1} \Db \ar[r]^{\delta_{\Lambda}}
		& \Db.
	}
\end{equation*}

We assume additionally that $\partial U_{p_i}^0$ is a Jordan curve of zero Lebesgue measure. Define a sequence of neighborhoods $U_{p_1}^{2n}$ for $n\in \N$ as the unique component of $f^{-2n}(U_{p_1}^0)$ containing $p_1$. Note that with this definition $U_{p_1}^{2n}$ is defined as $\Tilde{U}_{p_1}^n$ for the map $f^2$ would be defined. Then define $U_c^{2n+1}$ as the unique component of $f^{-1}(U_{p_1}^{2n})$ containing $c$. 
\begin{equation*} 
	\xymatrix{U_c^{1}\ar[r]^{f} \ar[d]_{\phi} & U_{p_1}^0 \ar[d]^{\psi}\\
		\Db \ar[r]^{z^d} & \Db
	}
\end{equation*}

and for $n\geq 1$ we obtain \begin{equation*}
	\xymatrix{U_c^{2n+1} \ar[r]^{f} \ar[d]_{\phi} & U_{p_1}^{2n} \ar[r]^{f^{2n}}\ar[d]^{\psi} & U_{p_1}^0\ar[d]^{\psi}\\
		|\Lambda|^{-n/d}\Db \ar[r]^{z^d} & |\Lambda|^{-n}\Db \ar[r]^{\delta_{|\Lambda|^n}} &\Db
	}
\end{equation*}

We will further define $U_{z}^{2n+2}$ as the unique component of $f^{-2}(U_{p_1}^{2n})$ containing $z$. We can visualize this as follows 
\begin{equation*}
	\xymatrix{U_{v}^{2n+2} \ar[r]^f & U_{c}^{2n+1} \ar[r]^{f} & U_{p_1}^{2n} \ar[r]^f &U_{p_2}^{2(n-1)+1}\ar[r]^f& U_{p_1}^{2(n-1)}
	}
\end{equation*}
Similar to above we will define a sequence $f_n$ by 
\begin{equation*}f_n(z)=\begin{cases}
		(\psi^{-1}\circ\delta_{|\Lambda|^{-m}}\circ h\circ \delta_{|\Lambda|^{n/d}}\circ \phi)(z) \quad &\textup{if } n=2m+1, \, z \in U_c^n\\
		f(z) \quad &\textup{else,}
\end{cases}\end{equation*}
similar to above, we can represent this with diagram:

\begin{equation*}
	\xymatrix{
		U_c^n \ar[r]^{f} \ar[d]_{\phi}
		& U_{p_1}^n \ar[d]^{\psi} \\
		|\Lambda|^{-m/d}\Db \ar[r]^{z^d}_{h_n}
		& |\Lambda|^{-m}\Db}
\end{equation*}

Using this we can define \begin{equation}
	\label{definition: F_m for per. c}
	F_m(z):=\begin{cases}
		f_{2m-1}\circ f(z)  &z\in U_v^{2m}\\
		f\circ f_{2m+1}(z) &z\in U_c^{2m+1}\\
		f^2(z) &\textup{else.}
\end{cases}\end{equation}

For an overview of what happens around the critical and pre-critical neighborhoods the following diagram might help:

\begin{equation*}
	\xymatrix{
		U_{v}^{2m+2} \ar[r]^{f} & U_{c}^{2m+1} \ar[r]^{f}_{f_{2m+1}} \ar[d]^{\phi}
		& U_{p_1}^{2m} \ar[d]^{\psi} \ar[r]^{f^{2m}} & U_{p_1}^{0} \ar[d]^{\psi} \\
		&
		|\Lambda|^{-(m-1)/d}\Db \ar[r]^{z^d}_{h_{2m+1}} & |\Lambda|^{-(m-1)}\Db \ar[r]^{\delta_{\Lambda^m}} & \Db.
	}
\end{equation*}

Now we are in shape to define $D_m(z)$ for $F$. With this we define \begin{equation*}D_m(z):= s\, \frac{1}{\|DF_1\circ \dots\circ F_m(z)\|_{{\CDach}}},
\end{equation*}
As shown in Corollary \ref{corollary: r_n in (0, 1)}, there exists a factor $s\in (0, 1]$ to guarantee that for all $z\in \CDach$ and all $m\in \N$ we have $D_m(z)\in (0, 1)$. Then we define a sequence of radii 

\begin{equation*}
	R_m(z)= 1-D_m(z)
\end{equation*}
and approximating spheres \begin{equation*}
	\Tilde{S}_m= \{(z, r_m(z)): z\in \CDach, m\in \N\}.
\end{equation*}

Using this sequence, we will define $d_n$ for $f$. The idea is to interpolate with geometric means between $\Tilde{S}_m$ and $\Tilde{S}_{m+1}$. 

Observe that for points $z\in \CDach$ with $z=f^{-2(m+1)}(w) $ for some $w\in \CDach\setminus (U_{p_1}^0\cup U_{p_2}^1)$. Then we find that   \begin{align*}D_{m+1}(z)&=\frac{1}{(f^{2(m+1)})^{\#}}=\frac{1}{f^{\#}(f^{2m+1}(z))f^{\#}(f^{2m}(z))(f^{2m})^{\#}(f^2(z))}\\
	&=D_n(z)\cdot \frac{1}{f^{\#}(f^{2m+1}(z))f^{\#}(f^{2m}(z))}.\end{align*}

Now the idea is to split the factor \begin{equation*}\frac{1}{f^{\#}(f^{2m+1}(z))f^{\#}(f^{2m}(z))}
	=
	\frac{1}{F^{\#}(F^{m+1}(z))}
\end{equation*}
and for $n=3m+k$ for $k=0, 1, $ define 
\begin{equation*}d_n(z):=D_m(z)\cdot \frac{1}{[F^{\#}(F^{m+1}(z))]^{k/2}}.\end{equation*}
Note that for $k=2$ the above definition coincides with $d_{n+2}(z)=D_{m+1}(z)$.

In order to define $d_n$ everywhere, i.e., also on pre-images of $U_{p_1}^0$ and $U_{p_2}^1$, for $n=2m+k$ with $k=0, 1$ we set 
\begin{equation*}
	d_n(z)= D_m(z)\left[\frac{D_{m+1}(z)}{D_m(z)}\right]^{k/2}.
\end{equation*}

Around the critical and post-critical points we will define
\begin{align*}
	U_c^{2m}&:= U_c^{2m-1} & U_{v}^{2m+1}&:= U_c^{2m}\\
	U_{p_1}^{2m+1}&:= U_{p_1}^{2m} & U_{z_2}^{2m}&:= U_{z_2}^{2m-1}
\end{align*}
for $m\geq 1$.

Recall that above we defined $f_n$ around $c$ only for $n=2m+1$. We will define $f_n$ for $n=2m$ on $U_c^{2m}$ by 
\begin{equation*}
	f_n(z)=f_{2m-1}(z) \quad z\in U_c^{2m}.
\end{equation*}

Setting $r_m(z)=1-d_m(z)$ we define and the maps
\begin{equation*}
	\overline{f}|_{S_n}: S_n\to S_{n-1}, \, (z, r_n(z))\mapsto (f_n(z), r_{n-1}(f_n(z))).
\end{equation*}

\subsubsection{Orbits that avoid regions around critical and pre-critical points}

The following Lemma resembles Lemma \ref{lemma: scaling of distance to sp by f'} that was key point for proving quasi-regularity outside the critical neighborhoods. 

\begin{lemma}
	Let $z\in \CDach \setminus f^{-n}\big(\U_{\post}\cup \U_{\crit}^1\big)$. Then \begin{equation*}
		\begin{split}
			d_{n-1}(f(z)) &\asymp f^{\#}(z)\cdot d_n(z) \\
			d_{n-2}(f^2(z)) &= (f^2)^{\#}(z)\cdot d_n(z)
		\end{split}
	\end{equation*}
	where $C(\asymp)$ can be chosen uniform for all $n\in \N$. 
\end{lemma}

\begin{proof}
	For the first inequality we will distinguish the cases $k=0$ and $k=1$. For $k=0$, we have that \begin{equation*}
		d_n(z)
		= 
		D_{2m}(z) \quad
		\textup{and}\quad 
		d_{n-1}(f(z))
		=
		D_{2(m-1)+1}(f(z)).
	\end{equation*}
	If we write the full expressions in terms of derivatives of $f$ we obtain \begin{equation*}
		\begin{split}
			d_n(z)
			&=
			\frac{1}{(f^{2m})^{\#}(z)}\\
			d_{n-1}(f(z))
			&= 
			\frac{1}{(f^{2(m-1)})^{\#}(f(z))} \left[\frac{1}{f^{\#}(f^{2m}(z))f^{\#}(f^{2m-1}(z))}\right]^{1/2}.
		\end{split}
	\end{equation*}
	\begin{equation*}
		\begin{split}
			d_{n-1}(f(z))
			&=
			\frac{1}{f^{\#}(f^{2m-2}(z))\cdot \dots\cdot f^{\#}(f(z))}\left[\frac{1}{f^{\#}(f^{2m}(z))f^{\#}(f^{2m-1}(z))}\right]^{1/2}\\
			&=
			\frac{f^{\#}(f^{2m-1}(z))f^{\#}(z)}{(f^{2m})^{\#}(z)} \left[\frac{1}{f^{\#}(f^{2m}(z))f^{\#}(f^{2m-1}(z))}\right]^{1/2}\\
			&=
			(f^{\#}(z)) \cdot d_n(z) \left[\frac{f^{\#}(f^{2m-1}(z))}{f^{\#}(f^{2m}(z))}\right]^{1/2}.
		\end{split}
	\end{equation*}
	For $k=1$ and $n=2m+k$ the computations are similar, \begin{equation*}
		\begin{split}
			d_n(z)
			&=
			\frac{1}{(f^{2m})^{\#}(z)}\left[\frac{1}{f^{\#}(f^{2m+1}(z))f^{\#}(f^{2m}(z))}\right]^{1/2}\\
			d_{n-1}(f(z))
			&=
			\frac{1}{(f^{2m})^{\#}(f(z))}.	
		\end{split}
	\end{equation*}
	We can rewrite $d_{n-1}(z)$ as follows 
	\begin{equation*}
		\begin{split}
			d_{n-1}(f(z))
			&= 
			f^{\#}(z)\cdot d_n(z) \cdot \frac{[(f^{\#}(f^{2m +1}(z))f^{\#}(f^{2m}(z))]^{1/2}}{f^{\#}(f^{2m}(z))}\\
			&= 
			f^{\#}(z)\cdot d_n(z) \cdot \left[\frac{(f^{\#}(f^{2m +1}(z))}{f^{\#}(f^{2m}(z))}\right]^{1/2}
		\end{split}
	\end{equation*}
	Because we chose $z\in \CDach \setminus f^{-n}\big( \bigcup_{v\in\crit(f)\cup \post(f)} U_v^0\big)$ we obtain that $f^{\#}(f^n(z))$ and $f^{\#}(f^{-(n-1)}(z))$ are uniformly bounded from below and above. We can estimate the factors by \begin{equation*}
		C
		:=
		\max_{w\in \CDach\setminus \bigcup_{v\in \crit(f)\cup \post(f)}U_v}\left\{ \left[\frac{f^{\#}(w)}{f^{\#}(f^{-1}(w))}\right]^{1/2}, \left[\frac{f^{\#}(f^{-1}(w))}{f^{\#}(w)}\right]^{1/2}\right\},
	\end{equation*}
	and with this we have shown that \begin{equation*}
		\frac{1}{C}f^{\#}(z)d_{n}(z)
		\leq 
		d_{n-1}(f(z))
		\leq 
		Cf^{\#}(z)d_n(z),
	\end{equation*}
	and as we claimed, $C$ is a uniform constant.

	For the second equality, let $n=2m +k$. If $k=0$ then we have that \begin{equation*}
		\begin{split}
			d_{n-2}(f^2(z))&
			= 
			D_{m-1}(z) 
			= 
			\frac{1}{(f^{2(m-1)-1})^{\#}(f^2(z))} \\
			&= 
			\frac{1}{f^{\#}(f^{2m-1}(z))\cdot \dots\cdot f^{\#}(f^2(z))}\\
			&=
			\frac{(f^2)^{\#}(z)}{(f^{2m})^{\#}(z)}
			=
			(f^2)^{\#}(z)D_m(z)
			=
			(f^2)^{\#}(z)d_{n}(z).	
		\end{split}
	\end{equation*}
	If $k=1$ we verify that \begin{equation*}
		\frac{D_m(f^2(z))}{D_{m-1}(f^2(z))}
		=
		\frac{D_{m+1}(z)}{D_m(z)}.
	\end{equation*}
	\begin{equation*}
		\begin{split}
			\frac{D_m(f^2(z))}{D_{m-1}(f^2(z))}
			&=
			\frac{1}{f^{\#}(f^{2m-1}(f^2(z))\cdot f^{\#}(f^{2m-2}(f^2(z)))}\\
			&=
			\frac{1}{f^{\#}(f^{2m+1}(z))f^{\#}(f^{2m}(z))}
		\end{split}
	\end{equation*}
	and the other factor can be rewritten as \begin{equation*}
		\begin{split}
			\frac{D_{m+1}(z)}{D_m(z)} 
			&=
			\frac{1}{f^{\#}(f^{2(m+1)-1}(z)f^{\#}(f^{2(m+1)-2}(z))}\\
			&= \frac{1}{f^{\#}(f^{2m+1}(z))f^{\#}(f^{2m}(z))},
		\end{split}
	\end{equation*}
	so since the two are equal, we obtain the claimed result that \begin{equation*}
		d_{n-2}(f^2(z))
		=
		(f^2)^{\#}d_n(z).\qedhere
	\end{equation*}
\end{proof}

\subsubsection{Near critical and pre-critical points}
Similarly to Lemma \ref{lemma:mon_radii_crit} we may choose $U_{p_1}^0$ smaller so that for all $z\in U_c^{2m+1}$ we have that $R_{m+1}(z)> R_m(z)$ and for all $w\in U_c^{2m}$ we have that $R_{m+1}(w)>R_m(w)$.

\begin{lemma}
	\label{lemma: derivative formula per. c}
	\begin{enumerate}
		\item 
		\label{item: derivative formula per. c 1}
		Let $z\in U_c^{2m+1}$ for $m\in \N$ and let $w=\phi(z)\in |\Lambda|^{-m/d}\D$. Then the derivative of the composition of the functions $(F_1\circ \dots\circ F_m)(z)$ can be expressed as \begin{equation*}
			\|D(F_1\circ \dots\circ F_m)(z)||_{\CDach}
			= 
			f^{\#}(z_{m})\cdot (\psi^{-1})^{\#}(w_m)\cdot |\Lambda|^{(m-1)/d} \cdot \phi^{\#}(z),
		\end{equation*}
		where $w_m = \Lambda^{m-1}\cdot h_{2m+1}(w)$ and $z_m=\psi^{-1}(w_m)$. 
		\item 
		\label{item: derivative formula per. c 2}
		Let $z'\in U_v^2m$ for $m\in \N$ and let $w'=\phi(f(z'))$. Then the derivative of $(F_1\circ \dots\circ F_m)(z')$ is equal to 
		\begin{equation*}
			\|D(F_1\circ\dots\circ F_m)(z)\|_{\CDach} 
			=
			(\psi^{-1})^{\#}(w_m')\cdot |\Lambda|^{-(m-1)/d}d\cdot \phi^{\#}(f(z'))f^{\#}(z').
		\end{equation*}
		when we set $w'= \phi(f(z'))$ and $w_m'= \Lambda^{-(m-1)}\cdot h_{2m-1}(w')$
	\end{enumerate}	
\end{lemma}
\begin{proof}
	\ref{item: derivative formula per. c 1} Let $z\in U_c^{2m+1}$ and $m\in \N$ be arbitrary. Then by \eqref{definition: F_m for per. c} and the linearization of $F=f^2$ on $U_{p_1}^{2m}$ we can rewrite $F_1\circ\dots\circ F_m$ as follows: 
	\begin{equation*}
		\begin{split}
			(F_1\circ\dots\circ F_m)(z)
			&= 
			(F^{m-1}\circ f\circ f_{2m+1})(z)\\
			&= 
			(f\circ F^{m-1}\circ f_{2m+1})(z)\\
			&= 
			(f\circ \psi^{-1}\circ \delta_{\Lambda^{m-1}}\circ h_{2m+1}\circ \phi)(z).
		\end{split}
	\end{equation*}
	Now can use the chain rule to find an explicit expression for the derivative and using the notation $w_m$ and $z_m$ as above we obtain: 
	\begin{equation*}
		\begin{split}
			\|D(F_1\circ \dots\circ F_m)(z)\|_{\CDach}
			&=
			f^{\#}(z_m)\cdot (\psi^{-1})^{\#}(w_m)\cdot |\Lambda|^{m-1} |\Lambda|^{-(m-1)} d |\Lambda|^{(m-1)/d} \phi^{\#}(z)\\
			&= 
			f^{\#}(z_m)\cdot (\psi^{-1})^{\#}(w_m)|\Lambda|^{(m-1)/d}\cdot d\cdot \phi^{\#}(z).
		\end{split}
	\end{equation*}
	\ref{item: derivative formula per. c 2} Let $z'\in U_v^{2m}$ and $m\in \N$ be arbitrary. Then $f(z)\in U_c^{2m-1}$ and by definition $F_m(z)= f_{2m-1}\circ f(z)$. Then $f_{2m-1}(F_m(z)) \in U_{p_1}^{2(m-1)}$ and we can use the linearization around $U_{p_1}$. 
	\begin{equation*}
		\begin{split}
			(F_1\circ \dots\circ F_m)(z')
			&= 
			(\psi^{-1}\circ \delta_{\Lambda^{m-1}}\circ h_{2m-1}\circ \phi\circ f)(z')
		\end{split}
	\end{equation*}
	For the derivative we obtain with the chain rule that 
	\begin{equation*}
		\begin{split}
			\|D(&F_1\circ\dots\circ F_m)(z)\|_{\CDach} \\
			=
			&(\phi^{-1})^{\#}(w_m')\cdot |\Lambda|^{-(m-1)} |\Lambda|^{m-1} d |\Lambda|^{-(m-1)/d}\cdot \phi^{\#}(f(z'))\cdot f^{\#}(z')\\
			=
			& (\psi^{-1})^{\#}(w_m')|\Lambda|^{-(m-1)/d}d\cdot \phi^{\#}(f(z'))f^{\#}(z').
		\end{split}
	\end{equation*}
	with $w'= \phi(f(z'))$ and $w_m'= \Lambda^{-(m-1)}\cdot h_{2m-1}(w')$ as above.
\end{proof}

\begin{lemma}
	\label{lemma: monotonicity of Rm around critical pt}
	The set $U_{p_1}^0$ can be chosen such that for all $m\in \N$ and all $z\in U_c^{2m+1}$ we have that $R_{m+1}(z)> R_m(z)$. 
\end{lemma}

\begin{proof} We will distinguish several cases, first, we let $z\in U_c^{2(m+1)+1}$. The difference $R_{m+1}(z)-R_m(z)$ depends on $\|D(F_1\circ\dots\circ F_{m+1})(z)\|_{\CDach}\| D(F_1\circ \dots\circ F_m)(z)\|_{\CDach}$. So using Lemma \ref{lemma: derivative formula per. c}
	\begin{equation*}
		\begin{split}
			\|D(&F_1\circ\dots\circ F_{m+1})(z)\|_{\CDach}-\| D(F_1\circ \dots\circ F_m)(z)\|_{\CDach} \\
			=
			&f^{\#}(z_{m+1})(\psi^{-1})^{\#}(w_{m+1}) d|\Lambda|^{m/d}\phi^{\#}(z) 
			-
			f^{\#}(z_m)(\psi^{-1})^{\#}(w_m)d|\Lambda|^{(m-1)/d} \phi^{\#}(z)\\
			=
			&d |\Lambda|^{(m-1)/d} \phi^{\#}(z) \big(|\Lambda|^{1/d}f^{\#}(z_{m+1})(\psi^{-1})^{\#}(w_{m+1})- f^{\#}(z_m) (\psi^{-1})^{\#}(w_m)\big).
		\end{split}
	\end{equation*}
	Note that the above expression is positive whenever \begin{equation*}
		\label{equation: Monotonicity of R_m1}
		|\Lambda|^{1/d}f^{\#}(z_{m+1})(\psi^{-1})^{\#}(w_{m+1})- f^{\#}(z_m) (\psi^{-1})^{\#}(w_m)
	\end{equation*}
	is positive. We will first turn to the case when $z\in U_c^{2m+1}\setminus U_c^{2(m+1)+1}$ and then make an argument that will apply to both cases.

	Since $z\notin U_c^{2(m+1)+1}$ the radius $R_{m+1}$ depends on $(F^{m+1})^{\#}(z)$ and using the conjugacy of $f$ to $z^d$ on $U_c^{1m+1}$ we can express that function as \begin{equation*}
		\begin{split}
			F^{m+1}(z)
			&= 
			(f\circ f^{2m}\circ \psi^{-1}\circ p_d\circ \phi)(z)\\
			&=
			(f\circ \psi^{-1}\circ \delta_{\Lambda^m}\circ p_d\circ \phi)(z)
		\end{split}
	\end{equation*}
	with $p_d(z)=z^d$ and using the linearization on $U_{p_1}^{2m}$. 
	Compute a lower bound to the derivative. Let again $w=\phi(z) $ and we denote $\Tilde{w}_{m+1}:= \Lambda^{m}\cdot w^d$ and $\Tilde{z}_{m+1}:= \psi^{-1}(\Tilde{w}_{m+1})$. Note that $|\Tilde{w}_{m+1}|\geq |\Lambda|^{-m/d}$
	\begin{equation*}
		\begin{split}
			(F^{m+1})^{\#}(z) 
			&= 
			f^{\#}(\Tilde{z}_{m+1})\cdot (\psi^{1})^{\#}(\Tilde{w}_{m+1})|\Lambda|^m \cdot d|w|^{d-1}\cdot \phi^{\#}(z)\\
			&\geq 
			f^{\#}(\Tilde{z}_{m+1})\cdot (\psi^{1})^{\#}(\Tilde{w}_{m+1})|\Lambda|^{m}\cdot d|\Lambda|^{-m+m/d} \phi^{\#}(z)\\
			&=
			f^{\#}(\Tilde{z}_{m+1})\cdot (\psi^{1})^{\#}(\Tilde{w}_{m+1}) |\Lambda|^{m/d}\cdot d \phi^{\#}(z).
		\end{split}
	\end{equation*}
	Using this expression and the same formula for $\|D(F_1\circ\dots\circ F_m)(z)\|_{\CDach}$ we obtain \begin{equation*}
		\begin{split}
			(F^{m+1})^{\#}(z)&-\|D(F_1\circ\dots\circ F_m)(z)\|_{\CDach} \\
			&= 
			f^{\#}(\Tilde{z}_{m+1})\cdot (\psi^{1})^{\#}(\Tilde{w}_{m+1})|\Lambda|^m \cdot d|w|^{d-1}\cdot \phi^{\#}(z)\\
			-
			&f^{\#}(z_m)(\psi^{-1})^{\#}(w_m)d|\Lambda|^{(m-1)/d} \phi^{\#}(z)\\
			&\geq d|\Lambda|^{(m-1)/d} \phi^{\#}(z)\big(|\Lambda|^{1/d}f^{\#}(\Tilde{z}_{m+1})(\psi^{-1})^{\#}(\Tilde{w}_{m+1})- f^{\#}(z_m)(\psi^{-1})^{\#}(w_m)\big).
		\end{split}
	\end{equation*}
	So the expression is positive whenever \begin{equation*}
		\label{equation: monotonicity R_m2}
		|\Lambda|^{1/d}f^{\#}(\Tilde{z}_{m+1})(\psi^{-1})^{\#}(\Tilde{w}_{m+1})- f^{\#}(z_m)(\psi^{-1})^{\#}(w_m)
	\end{equation*}
	is positive.
	
	Recall that $|\Lambda|>1$ and so $|\Lambda|^{1/d}>1$. Further, $z_m, z_{m+1}, \Tilde{z}_{m+1}\in U_{p_1}^2$ and $w_m, w_{m+1}, \Tilde{w}_{m+1}\in |\Lambda|^{-1}\D$. Because $w_m, w_{m+1}, \Tilde{w}_{m+1}\in |\Lambda|^{-1}\D$ we may use apply Lemma~\ref{lem:Koebe_phi_psi}~\ref{item:Koebe_psi} and Theorem \ref{theorem: spherical Koebe} to conclude that for the holomorphic map $f\circ \psi^{-1}$ we have that 
	\begin{equation*}
		(f\circ \psi^{-1})^{\#}(u)\asymp (f\circ \psi^{-1})^{\#}(u')
	\end{equation*}
	for all $u, u'\in r\D$ with $C(\asymp)\searrow 1 $ as $r\to 1$. 
	
	So if we pick $r\in (0, 1)$ sufficiently small the expressions \eqref{equation: Monotonicity of R_m1} and \eqref{equation: monotonicity R_m2} are positive when we choose $U_{p_1}^0= \psi^{-1}(r\D)$. 
\end{proof}

\begin{corollary} 
	Let $z\in f^{-n}(w)$ for some $w\in U_c^{2m+1}$ and $n\in \N$ arbitrary. Then $R_{2m+n+2}\geq R_{2m+n+1}$. 
\end{corollary}
\begin{proof}
	content...
\end{proof}

\begin{lemma}
	If $U_{p_1}^0$ and $U_{p_2}^1$ are chosen such that $(f^2)^{\#}(z)>1 $ for all $z\in U_{p_1}^0\cup U_{p_2}^1$ then we have that \begin{equation*}
		R_{m+1}(w)> R_m(w)
	\end{equation*}
	for all $w\in U_c^{2m+1}$. 
\end{lemma}

Recall that around critical neighborhoods the extension was defined by parametrizations of cylindrical neighborhoods around $c$. Here we will consider similar neighborhoods around $c, u$ for each $n$. 

Around critical points we aim to use the same strategy as above. Recall that we needed the radii to be monotone around critical points. That depended on the difference 
\begin{equation}\label{equation: removing assumption}\|DF^{m}\circ F_{m+1}(z)\|_{{\CDach}}-\|DF^{m-1}\circ F_m(z)\|_{{\CDach}}.
\end{equation} 
Let $n=2m+1$. Using the linearization around $U_{p_1}^{2m}$ we can find the following expression for $F^{m-1}\circ F_{m}$: \[F^{n-1}\circ F_{n}(z)= \begin{cases}
	f\circ (\psi^{-1}\circ \delta_{\Lambda^{m-1}}\circ\psi)\circ f_{2m+1} (z)\quad &z\in U_c^{2n+1}\\
	(\psi^{-1}\circ \delta_{\Lambda^{m-1}}\circ \psi)\circ f_{2m-1}\circ f(z) &z\in U_{v}^{2m}
\end{cases}\]

In the following we will focus on the case with $z\in U_c^n$ but the other cases are similar. The radii are monotone in case the difference \eqref{equation: removing assumption} is positive. We will use the notation 
\begin{equation*}
	z_n:= \delta_{\Lambda^{n-1}}\circ \psi\circ f_n(z), \, z_{n+1}:= \delta_{\Lambda^n}\circ \psi\circ f_{2m+1}(z)
\end{equation*}
The difference of the radii depends on 
\begin{align*} &\|DF^{n}\circ F_{n+1}(z)\|_{{\CDach}}-\|DF^{n-1}\circ F_n(z)\|_{{\CDach}} \\
	= &(f^2)^{\#}(\psi^{-1}(z_{n+1}))\cdot (\psi^{-1})^{\#}(z_{n+1}) |\Lambda|^{-1+(n+1)/d}\cdot d\cdot \phi^{\#}(z)\\
	-&(f^2)^{\#}(\psi^{-1}(z_n))\cdot (\psi^{-1})^{\#}(z_n)\|\Lambda^{-1+n/d}\cdot  d\cdot \phi^{\#}(z)\\
	= &|\Lambda|^{-1+n/d}\cdot d\cdot \phi^{\#}(z) \cdot ((f^2)^{\#}(\psi^{-1}(z_{n+1}))\cdot (\psi^{-1})^{\#}(z_{n+1})|\Lambda^{-1/d}|\\  &-(f^2)^{\#}(\psi^{-1}(z_n))\cdot (\psi^{-1})^{\#}(z_n)).
\end{align*}
Here we used the formula for the derivative proved in Lemma \ref{lemma: derivative formula per. c}.

\begin{lemma}
	There exists a choice of $U_{p_1}^0$ such that for all $n\in \N$ and all $z\in U_{p_i}^n$ we have that $r_{n+1}(z)\geq r_n(z)$.
\end{lemma}
\begin{proof}
	Let $z\in U_{p_i}^n$ be arbitrary for $i\in\{1, 2\}$, let $n=2m+k$ for $k\in \{0, 1\}$.
	
	We can distinguish several cases, first let $n+1=2m+1$. Then \begin{equation*}
		\begin{split}
			r_{n+1}(z)-r_n(z) 
			&=
			D_m(z)-D_{m+1}(z)\left[\frac{D_{m+1}(z)}{D_m(z)}\right]^{1/2}\\
			&= 
			D_m(z)\left(1- \frac{1}{\big(F^{\#}(F^m(z))\big)^{1/2}}\right).
		\end{split}
	\end{equation*}
	Note that $r_{n+1}(z)-r_n(z)>0 $ when $F^{\#}(F^m(z))>1$. The point $F^m(z)\in U_{p_1}^0$.
	
	The second case is when $n+1=2m$. In that case, the difference is \begin{equation*}
		\begin{split}
			r_{2m}(z)-r_{2(m-1)+1}(z)
			&=
			D_{m-1}(z)\left[\frac{D_m(z)}{D_{m-1}(z)}\right]^{1/2} - D_m(z)\\
			&= 
			D_m(z)^{1/2}\left(1-\frac{1}{\big(F^{\#}(F^{m-1}(z))\big)^{1/2}}\right).
		\end{split}
	\end{equation*}
	In this case, $r_{2m}(z)-r_{2(m-1+1)}(z)>0$ when $F^{\#}(F^{m-1}(z))>1$. Here $F^{m-1}(z)\in U_{p_1}^2$. 
	
	Because $F$ is expanding, we may choose $U_{p_1}^0 $ and hence $U_{p_1}^n$ small enough to guarantee that $F^{\#}(w)>1$ for all $w\in U_{p_1}^0$. 
	
	We can do the analogous computations for neighborhoods of $p_2$ and if necessary pick $U{p_1}p^0$ so that $F^{\#}(w)>1$ for all $w\in U_{p_2}^0$. 
\end{proof}

\begin{lemma}
	\begin{enumerate}
		\item Let $z\in U_c^n$ and $n\in \N$ be arbitrary. Then $r_{n+1}(z)> r_n(z)$. 
		\label{item: monotonicity crit neigh per. p1}
		\item Let $z'\in U_v^n$ and $n\in \N$ be arbitrary. Then $r_{n+1}(z)> r_n(z)$.
		\label{item: monotonicity crit neigh per. p2}
	\end{enumerate}
\end{lemma}

\begin{proof}\ref{item: monotonicity crit neigh per. p1} 
	Let $z\in U_c^n$ and suppose that $n=2m+k$ for $k\in \{0, 1\}$. If $n+1=2m+1$ we can compute the difference as \begin{equation*}
		\begin{split}
			r_{n+1}(z)-r_n(z)
			&=
			D_m(z)- D_{m}(z)\left[\frac{D_{m+1}(z)}{D_{m}(z)}\right]^{1/2}\\
			&=
			D_m(z)\left(1-\left[\frac{D_{m+1}(z)}{D_m(z)}\right]^{1/2}\right)
		\end{split}
	\end{equation*}
	Again, this depends on $D_{m+1}(z)/D_m(z)$. By Lemma \ref{lemma: monotonicity of Rm around critical pt} we can conclude that $D_{m+1}(z)> D_m(z)$ and so the difference is positive. 
	
	\smallskip
	\ref{item: monotonicity crit neigh per. p2} This can be shown analogously. 
\end{proof}

We will proceed to construct the extension around the critical neighborhoods similarly as in Section \ref{subsection: extension near critical points}. Recall that we defined $U_c^{2m}= U_c^{2m-1}$ for $m\in \N$.

\begin{lemma}
	Let $n= 2m+1$ for some $m\in \N$ and $z\in U_c^n$. Then the map \begin{equation*}
		\begin{split}
			\overline{f}: P(U_c^n, n)&\to P(U_p^{n-1}, n-1)\\
			(z, \rho_n(z))&\mapsto (f_n(z), \rho_{n-1}(f_n(z)))
		\end{split}
	\end{equation*}
	is a quasi-similarity when the spherical coordinate is restricted to suitable subsets of $U_c^n\setminus \{c\}$. 
\end{lemma}

We will define parametrizations of $U_c^n$ and $U_p^{n-1}$ for $n= 2m$ with $m\in \N$ arbitrary in a similar fashion to Lemma \ref{lemma:param_of_P(U_c)} and Lemma \ref{lemma:parap_of_P(Up)}. 

Similarly to the above we will parametrize the set $P(U_c^n, n)$ with a cylinder $|\Lambda|^{(n-2)/d}\times [0, q_n]$ where $\overline{q}_n= r_{n+1}(c)-r_n(z)$. In the following we will use the notation $\rho_n(z, t)$ also when $t\in [0, h]$ with $h$ not necessarily equal to one for 
\begin{equation}
	\label{equation: convention for rhon with h not necc 1 period}
	\rho_n(z, t):= \frac{h-t}{h}r_n(z)+\frac{t}{h}r_{n+1}(z).
\end{equation}
The parameter $h$ will be clear from the context by specifying the interval in which $t$ lives. 
\begin{lemma}
	Let $n=2m$ for some $m\in \N$. Set $\overline{q}_n := r_{n+1}(c)-r_n(c)$. Then the following parametrization of $P(U_c^n, n)$ given by
	\begin{equation*}
		\begin{split}
			\overline{\alpha}_1: |\Lambda|^{(n-2)/d}\times [0, \overline{q}_n']&\to P(U_c^n, n)\\
			(x, t)&\mapsto \big(\phi^{-1}(x), \rho_n(\phi^{-1}(x))\big) 
		\end{split}
	\end{equation*}
	is bi-Lipschitz with constant $L_{\overline{\alpha}_1}= L_{\overline{\alpha}_1}(|\Lambda|^{-(n-2)/d}, |\Lambda|^{-(n-2)}, d)$ that is decreasing as $n$ increases.
	Here $\rho_n$ is as in Equation \eqref{equation: convention for rhon with h not necc 1 period}.
\end{lemma}
\begin{proof}
	The proof is analogous to the proof of Lemma \ref{lemma:param_of_P(U_c)}.
\end{proof}

\begin{lemma}
	Let $n=2m$ for some $m\in\N$. Then the following parametrization of $P(U_{p_1}^{n-1}, n-1)$ given by \begin{equation*}
		\begin{split}
			\overline{\alpha}_2: |\Lambda|^{n-1}_{p_1}\times [0, \overline{q}_n']&\to P(U_{p_1}^{n-1}, n-1) \\
			(x, t)&\mapsto \big(\psi^{-1}(x), \rho_{n-1}(\psi^{-1}(x), \rho_{n-1}(\psi^{-1}(x), t))\big)
		\end{split}
	\end{equation*}
	is bi-Lipschitz with constant $L_{\overline{\alpha}_2}= L_{\overline{\alpha}_2}(|\Lambda|^{n-2}, d)$. 
\end{lemma}
\begin{proof}
	This is analogous to the proof of Lemma \ref{lemma:parap_of_P(Up)}.
\end{proof}

Continuity is by construction and uniform quasi-regularity can be proved analogously to the case of fixed critical values. The quasi-regularity constant will be larger due to the factors.

\bibliographystyle{alpha}
\bibliography{bibliography}
\end{document}